\documentclass{article}
\usepackage{arxiv}
\usepackage[utf8]{inputenc} % allow utf-8 input
\usepackage[T1]{fontenc}    % use 8-bit T1 fonts
\usepackage{hyperref}       % hyperlinks
\usepackage{url}            % simple URL typesetting
\usepackage{booktabs}       % professional-quality tables
\usepackage{amsthm,amssymb,amsmath,amsfonts,amscd,mathrsfs,bm}       % blackboard math symbols
\usepackage{nicefrac}       % compact symbols for 1/2, etc.
\usepackage{microtype}      % microtypography
\usepackage{array,multirow}
\usepackage{graphicx}

\newtheorem{propo}{Proposition}
\newtheorem{lem}{Lemma}
\newtheorem{rem}{Remark}

\title{The operating diagram for 
a two-step anaerobic digestion model}

\author{ \href{https://orcid.org/0000-0002-6274-7826}{\includegraphics[scale=0.06]{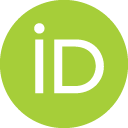}\hspace{1mm}Tewfik Sari}
\\
	ITAP, Univ Montpellier,\\ 
	INRAE, Institut Agro,
	Montpellier, France \\
	\texttt{tewfik.sari@inrae.fr}
	\and
	\And
	\href{https://orcid.org/0000-0001-6652-2141}{\includegraphics[scale=0.06]{orcid.png}\hspace{1mm}Boumediene Benyahia} \\
	Laboratoire d'Automatique de Tlemcen,\\ 
	Université de Tlemcen, Tlemcen, Algeria\\
	\texttt{b.benyahia.ut1@gmail.com}	
}

\date{\today}
\begin{document}
\maketitle

\begin{abstract}
The Anaerobic Digestion Model No. 1 (ADM1) is a complex model which is widely accepted as a common
platform for anaerobic process modeling and simulation. However, it has a large number of parameters and states that hinder its analytical study. Here, we consider the two-step reduced model of anaerobic digestion (AM2) which is a four-dimensional system of ordinary differential equations. The AM2 model is able to adequately capture the main dynamical behavior of the full anaerobic digestion model ADM1 and has the advantage that a complete analysis for the existence and local stability of its steady states is available. We describe its operating diagram, which is the bifurcation diagram
which gives the behavior of the system with respect to the operating parameters
represented by the dilution rate and the input concentrations of the substrates. This diagram, is very useful to understand the model from
both the mathematical and biological points of view.	
\end{abstract}

% keywords can be removed

\section{Introduction}\label{sec1}

The anaerobic digestion is a complex process in which organic material is converted into biogas (methane) in an environment without oxygen. Anaerobic digestion enables the water industry to treat waste water as a resource for generating energy and recovering valuable by-products.
The complexity of the anaerobic digestion process has motivated the development of complex models, such as the widely used Anaerobic Digestion Model No. 1 (ADM1) \cite{Batstone}. This model has 
a large number of state variables and parameters.
It is impossible to obtain an analytical characterization of the steady states and to describe the operating diagram, that is to say, to identify  the asymptotic behaviour of existing steady-states as a function of chemostat operating parameters (substrates inflow concentrations and dilution rate). To the knowledge author, only numerical investigations are available \cite{Bornhoft}.
 
Due to the analytical intractability of the full ADM1, work has been made towards the construction of simpler models that preserve biological meaning whilst reducing the computational effort required to find mathematical solutions of the model equations, to obtain a better understanding of the anaerobic digestion process. 
The most simple model of the chemostat with only one biological reaction, where one substrate is consumed by one microorganism is well understood \cite{HLRS,Monod,SW}. However such models are too simple to encapsulate the essence of the anaerobic digestion process.

More realistic models of anaerobic digestion are two-step models, with a cascade of two biological reactions, where one substrate $S_1$ is consumed by one microorganism $X_1$, to produce a product $S_2$ that serves as the main limiting substrate for a second microorganism $X_2$ as schematically represented by the following reaction scheme:
\begin{equation}\label{twostep}
k_1S_1\stackrel{\mu_1}{\longrightarrow}X_1+k_2S_2,
\quad
k_3S_2\stackrel{\mu_2}{\longrightarrow}X_2 +k_4{\rm CH}_4
\end{equation} 
where $\mu_1$ and $\mu_2$ are the kinetics of the reactions and $k_i$ are pseudo-stoichiometric coefficients associated to the bioreactions.
An important contribution on the modelling of anaerobic digestion as a two-step is presented by Bernard et al. \cite{Bernard}, hereafter denoted as AM2. The model has a Monod kinetics  for the first reaction and a Haldane one for the second and was extended with general growth functions characterized by qualitative properties by Benyahia et al. \cite{Benyahia} and Sbarciog et al. \cite{Sbarciog}. It has been shown by 
García-Diéguez et al. \cite{Garcia}
that under some circumstances, this very simple two-step model is able to adequately capture the main dynamical behavior of the full anaerobic digestion model ADM1. Moreover, it has been shown that the reduced AM2 model can support on-line control, optimization and supervision strategies,
through the synthesis of state observers and control feedback laws, see for instance\cite{Alcaraz2002,Alcaraz2005}.

Another simple two-step model of anaerobic digestion is the model presented by Xu et al. \cite{Xu}, where the product of the first microorganism, that serves as the substrate for the second microorganism, inhibits the growth of the first microorganism. The model incorporates a Monod with product inhibition kinetics for he first reaction and Monod kinetics for the second one and was extended with general growth functions characterized by qualitative properties 
by Daoud et al. \cite{Daoud} and Sari and Harmand \cite{Sari2016}.

The two-step models studied in \cite{Benyahia,Bernard,Sbarciog} present a commensalistic relationship between the microorganisms. According to Stephanopolous \cite{Stephanopoulos}, the commensalism is characterized by the fact that the second population (the commensal population) benefits for its growth from the  first population (the host population) while the host population is not affected by the growth of the commensal population and hence, the first population can grow without the second one. 
On the contrary, the two-step models studied in \cite{Daoud,Sari2016,Xu} present a syntrophic relationship between the microorganisms: the first population is affected by the growth of the second population, and hence no population can grow without the other.  
For more details and information on commensalism and syntrophy the reader is referred to 
\cite{Burchard,ElHajji,Reilly,Sari2012,
Sari2016,Stephanopoulos,WadeReview} and the references therein.

Another important and interesting extensions of the two-step anaerobic digestion models are the mathematical models, which include syntrophy and substrate inhibition, considered by Weedermann et al. \cite{Weedermann2013,Weedermann2015} and the three-step models, which consist in introducing an additional microorganism and substrate in a two-step syntrophic model, considered by Wade et al.  \cite{Wade2016} and Sari and Wade \cite{Sari2017}. 

In this paper we will consider the two-step model of Bernard et al. \cite{Bernard}, with general growth functions as in \cite{Benyahia,Sbarciog}, denoted here after AM2, and we describe its {\em operating diagram}. 
The operating diagram has the {\em operating parameters} as its coordinates and the various regions defined in it correspond to qualitatively different asymptotic behavior. A two-step model has three operating parameters that are the input concentration of substrate for each reaction and the dilution rate. These parameters are {\em control parameters} since they are under the control of the experimenter. Apart from these three parameters, that can vary, all other parameters have biological meaning and are fixed depending on the organisms and substrate considered.

Therefore the operating diagram is the bifurcation diagram that shows how the system behaves when we vary the control parameters.  As it was claimed by Smith and Waltman in their monography on the mathematical theory of the chemostat (see \cite{SW}, p. 252), the operating diagram is probably the most useful answer for the discussion of the behavior of the model with respect of the parameters. This diagram shows how robust or how extensive is the parameter region where some asymptotic behavior occur. 

This bifurcation diagram is very useful to understand the model from both the mathematical and biological points of view. Its importance for ecological modeling was emphasized by De Freitas et al. \cite{Freitas} and for bioreactors by Pavlou \cite{Pavlou}. These authors attributed its introduction to Jost et al. \cite{Jost}, who studied the dynamics of predator and prey interactions in a chemostat.
This diagram is often constructed
both in the biological literature
\cite{Freitas,Jost,Pavlou,Sbarciog,Wade2016,Xu} and the mathematical literature
\cite{Abdellatif,Bar,Bornhoft,Daoud,Dellal,Fekih,HLRS,Khedim,
Sari2016,Sari2017,Weedermann2013,Weedermann2015}.

The operating diagram of the AM2 model was only partially described by Sbarciog et al. \cite{Sbarciog}. In this paper we give a complete description of the diagram. AM2 model can have up to six steady state. Its operating diagram presents nine regions according to the steady state and their stability, that can exist in each region. The operating
diagram summarizes the effect of the operating conditions on the long-term dynamics of the AM2 model and shows six type of behavior visualized in the figures by six different colors. Since AM2 model has three operating parameters, and it is not easy to visualize regions in the three-dimensional operating parameter space, two of the
operating parameters  are used as coordinates
of the operating diagram and the effect of the third parameter are shown in a series of operating
diagrams.

This paper is organized as follows: in section \ref{sec2}, we present the mathematical model and recall the necessary and sufficient condition of existence and local, and global stability of its steady states. Next, in section \ref{sec3}, we present the operating diagram in the three-dimensional operating parameters space, in sections \ref{ODDconstant} and \ref{ODS2inconstant} we present the operating diagrams in two-dimensional operating parameters space  when one of the parameters is kept fixed. In section \ref{sec4}, we present some bifurcations diagram, with the dilution rate as the bifurcation parameter. Then, we conclude by discussing our results in section \ref{sec5}. Proofs and Tables are given in the appendix.
 
\section{Mathematical model}\label{sec2}

We consider the AM2 model of anaerobic digestion given in \cite{Bernard}, which takes the form of a two-step reactions (\ref{twostep}) where, in the first step, the organic substrate $S_1$ is consumed by the acidogenic bacteria $X_1$ and produces a substrate $S_2$ (Volatile Fatty Acids), while, in the second step, the methanogenic population $X_2$ consumes $S_2$ and produces methane. Let $D$ be the dilution rate, $S_{1{\rm in}}$ and $S_{2{\rm in}}$  the concentrations of influent substrate $S_{1}$ and $S_{2}$, respectively. The dynamical equations of the model take the form:
\begin{equation}
\label{AM2}
\begin{array}{lcl}
\dot{S}_1
&=&
D\left(S_{1{\rm in}}-S_1\right)-k_1\mu_1\left(S_1\right)X_1,\\
\dot{X}_1
&=&
\left(\mu_1\left(S_1\right)-\alpha D\right)X_1,\\
\dot{S}_2
&=&
D\left(S_{2{\rm in}}-S_2\right)+k_2\mu_1\left(S_1\right)X_1-k_3\mu_2\left(S_2\right)X_2,\\
\dot{X}_2
&=&
\left(\mu_2\left(S_2\right)-\alpha D\right)X_2,
\end{array}
\end{equation}
where $k_i$ are pseudo-stoichiometric coefficients associated to the bioreactions and
$\alpha\in[0,1]$ is a parameter allowing us to decouple the HRT (Hydraulic Retention Time) and the SRT (Solid Retention Time). 
In \cite{Bernard}, the kinetics $\mu_1$ and 
$\mu_2$ are of Monod and Haldane type, respectively:
\begin{equation}\label{MonodHaldane}
\mu_1\left(S_1\right)=\displaystyle\frac{m_{1}S_1}{K_1+S_1},\qquad
\mu_2\left(S_2\right)=\displaystyle\frac{m_{2}S_2}{K_2+S_2+\frac{S_2^2}{K_I}},
\end{equation}
The mass flow of the methane production, denoted by
$Q_{{\rm CH}_4}$, is proportional to the microbial activity, see \cite{BastinDochain}:
$$Q_{{\rm CH}_4}=k_4\mu_2\left(S_2\right)X_2$$ 
where $k_4$ is the coefficient in (\ref{twostep}). In this model the biogas is simply a product of the biological reactions and it has no feedback on the dynamical equations (\ref{AM2}).

Following \cite{Benyahia,Sbarciog}, we will consider (\ref{AM2}) with general $\mathcal{C}^1$  kinetics functions  $\mu_1$ and $\mu_2$ satisfying the following qualitative properties:

\noindent
{\em Hypothesis 1}.
$\mu_1(0)=0$, $\mu_1(+\infty)=m_1$ and $\mu'_1\left(S_1\right)>0$ 
for $S_1>0$.

\noindent
{\em Hypothesis 2}.
$\mu_2\left(0\right)=0$, $\mu_2(+\infty)=0$ and
there exists  $S_2^M>0$ such that $\mu'_2\left(S_2\right)>0$ 
for $0< S_2<S_2^M$, and
$\mu'_2\left(S_2\right)<0$ for $S_2>S_2^M$.

As it is usual in the mathematical theory of the chemostat, see for instance \cite{Sari2012}, we can use a change of variables that reduces the pseudo-stochiometric coefficients 
$k_i$ to 1. Indeed, the linear change of variables
$$s_1=\frac{k_2}{k_1}S_1,
\quad
x_1=k_2X_1,
\quad
s_2=S_2,
\quad
x_2=k_3X_2
$$
transforms (\ref{AM2}) into
\begin{equation}
\label{AM2reduit}
\begin{array}{lcl}
\dot{s}_1
&=&
D\left(s_{1{\rm in}}-s_1\right)-f_1\left(s_1\right)x_1,\\
\dot{x}_1
&=&
\left(f_1\left(s_1\right)-\alpha D\right)x_1,\\
\dot{s}_2
&=&
D\left(s_{2{\rm in}}-s_2\right)+f_1\left(s_1\right)x_1-f_2\left(s_2\right)x_2,\\
\dot{x}_2
&=&
\left(f_2\left(s_2\right)-\alpha D\right)x_2,
\end{array}
\end{equation}
where 
$$s_{1{\rm in}}=\frac{k_2}{k_1}S_{1{\rm in}},
\quad
s_{2{\rm in}}=S_{2{\rm in}},
\quad
f_1(s_1)=\mu_1\left(\frac{k_1}{k_2}s_1\right),
\quad
f_2(s_2)=\mu_2(s_2)
$$ 
However, since
the stoichiometric coefficients have their own importance for the biologist and since we are interested in giving the biologist a useful tool for the understanding of the role of the operating parameters, we do not make this reduction and we present the results in the original model (\ref{AM2}). This model can have at most six steady states, labeled below as in \cite{Benyahia}:
%=====================================================================
\begin{itemize}
\item $E_1^0$, where $X_1=0$ and $X_2=0$: the washout steady state where acidogenic and methanogenic bacteria are extinct.
%=====================================================================
\item $E_1^i$ ($i=1,2$), where $X_1=0$ and $X_2>0$: acidogenic bacteria are washed out, while  methanogenic bacteria are maintained.
%=====================================================================
\item $E_2^0$, where $X_1>0$ and $X_2=0$: methanogenic bacteria are washed out, while acidogenic bacteria are maintained.
%=====================================================================
\item $E_2^i$ ($i=1,2$), where $X_1>0$ and $X_2>0$: both 
acidogenic and  methanogenic bacteria are maintained.
\end{itemize}
%=====================================================================
%%%%%%%%%%%%%%%%%%%
%=====================================================================
\begin{table}[ht]
\caption{Auxiliary functions} \label{tableFunctions}
\vspace{-0.1cm}
\begin{center}
\begin{tabular}{ll}
\hline
\begin{tabular}{l}
$S_1^*(D)$
\end{tabular}
&
\begin{tabular}{l}
$S_1^*(D)$ is the unique solution 
of equation $\mu_1\left(S_1\right)=\alpha D$\\
It is defined for $0\leq D<D_1$,
where $D_1={m_{1}}/{\alpha}$\\
If $D\geq D_1$, by convention we let $S_1^*(D)=+\infty$
\end{tabular}
%=====================================================================
\\ \hline
\begin{tabular}{l}
$S_2^{i*}(D)$, $i=1,2$
\end{tabular}
&
\begin{tabular}{l}
$S_2^{1*}(D)<S_2^{2*}(D)$ are the solutions  
of equation $\mu_2\left(S_2\right)=\alpha D$\\
They are defined for $0\leq D\leq D_2$, 
where $D_2= {{\mu}_2\left(S_2^M\right)}/{\alpha}$\\ 
If $D=D_2$, one has 
$S_2^{1*}(D)=S_2^{2*}(D)$\\
If $D>D_2$, by convention we let $S_2^{1*}(D)=+\infty$
\end{tabular}
\\ \hline
%====================
\begin{tabular}{l}
$H_i(D)$, $i=1,2$
\end{tabular}
&
\begin{tabular}{l}
$H_i(D)=S_2^{i*}(D)+\frac{k_2}{k_1}S_1^*(D)$\\
It is defined for $0\leq D<\min(D_1,D_2)$
\end{tabular}
\\ \hline
%%%%%%%%%%%%%%%%%%%%%%%%%%%%%%%
\begin{tabular}{l}
${S}^*_{2{\rm in}}\left(D,S_{1{\rm in}},S_{2{\rm in}}\right)$
\end{tabular}
&
\begin{tabular}{l}
${S}^*_{2{\rm in}}\left(D,S_{1{\rm in}},S_{2{\rm in}}\right)=S_{2{\rm in}}+\frac{k_2}{k_1}\left(S_{1{\rm in}}-S_1^*(D)\right)$\\
It is defined for $0\leq D<D_1$ and $S_{1{\rm in}}>S_1^*(D)$
\end{tabular}
%==============================================
\\ \hline
\begin{tabular}{l}
$X_{1}^*\left(D,S_{1{\rm in}}\right)$
\end{tabular}
&
\begin{tabular}{l}
$X_{1}^*\left(D,S_{1{\rm in}}\right)=\frac{1}{k_1\alpha}\left(S_{1{\rm in}}-S_{1}^{*}(D)\right)$\\
It is defined for $0\leq D<D_1$ and $S_{1{\rm in}}>S_1^*(D)$
\end{tabular}
%===============================================
\\ \hline
%=====================================================================
\begin{tabular}{l}
%$X_{2}^{i}=
$X_{2}^{i}\left(D,S_{2{\rm in}}\right)$, $i=1,2$
\end{tabular}
&
\begin{tabular}{l}
$
X_{2}^{i}\left(D,S_{2{\rm in}}\right)=\frac{1}{k_3\alpha}
\left(S_{2{\rm in}}-S_{2}^{i*}(D)\right)$\\
It is defined for $0\leq D<D_2$ and $S_{2{\rm in}}>S_2^{i*}(D)$
\end{tabular}
\\ \hline
%=====================================================================
\begin{tabular}{l}
$X_{2}^{i*}\left(D,S_{1{\rm in}},S_{2{\rm in}}\right)$, $i=1,2$
\end{tabular}
&
\begin{tabular}{l}
$
X_{2}^{i*}\left(D,S_{1{\rm in}},S_{2{\rm in}}\right)=\frac{1}{k_3\alpha}
\left(S_{2{\rm in}}^*\left(D,S_{1{\rm in}},S_{2{\rm in}}\right)-S_{2}^{i*}(D)\right)$\\
It is defined for $0\leq D<\min(D_1,D_2)$, $S_{1{\rm in}}>S_1^*(D)$ and\\
$S_{2{\rm in}}+\frac{k_2}{k_1}S_{1{\rm in}}
>H_i(D)$
\end{tabular}
\\\hline
\end{tabular}
\end{center}
\vspace{-0.3cm}
\end{table}
%=====================================================================

For the description of the steady states, we need to define some auxiliary functions that are given in Table~\ref{tableFunctions}. For the particular case of Monod and Haldane functions (\ref{MonodHaldane}), the auxiliary functions can be computed analytically and are given in Table \ref{AuxiliaryFunctionsMonodHaldane}. We have the following result.
%===================================
\begin{propo}                                           \label{proSSi}
Assume that Hypotheses 1 and 2 hold. 
The steady states 
$E_1^0$, $E_1^i$ ($i=1,2$), $E_2^0$ and
$E_2^i$ ($i=1,2$) are given in Table \ref{tableSSi}, where 
$S_1^{*}$, $S_2^{i*}$, $S_{2{\rm in}}^{*}$, $X_1^{*}$, $X_2^{i}$, and $X_2^{i*}$, for $i=1,2$,
are defined in Table \ref{tableFunctions}.
Their conditions of existence and stability are given in Table \ref{TableSumExisStab}.
\end{propo}
%=====================================================================
\begin{proof}
The proof is given in Appendix \ref{ProofproSSi}.	
\end{proof}

%=====================================
\begin{table}[ht]
\caption{The steady states of (\ref{AM2}). $S_1^{*}$, $S_2^{i*}$, 
$S_{2{\rm in}}^{*}$,
$X_1^{*}$, $X_2^{i}$ and $X_2^{i*}$
are defined in Table \ref{tableFunctions}.} \label{tableSSi}
\vspace{-0.2cm}
\begin{center}
\begin{tabular}{lllll}
\hline
$E_1^0$  &
$S_1=S_{1{\rm in}}$ &
$S_2=S_{2{\rm in}}$ &
$X_1=0$ & 
$X_2=0$
\\ 
%=====================================================================
$E_1^i$, $i=1,2$  &
$S_1=S_{1{\rm in}}$ &
$S_2=S_2^{i*}$ &
$X_1=0$ & 
$X_2=X_{2}^{i}$
\\
%=====================================================================
$E_2^0$         &
$S_1=S_{1}^{*}$ &
$S_2={S}^*_{2{\rm in}}$ &
$X_1=X_{1}^*$ &                                               
$X_2=0$
\\ 
%=====================================================================
$E_2^i$, $i=1,2$        &
$S_1 =S_{1}^{*}$ &
$S_2=S_2^{i*}$ &
$X_1=X_{1}^*$ &
$X_2=X_{2}^{i*}$
\\
\hline
\end{tabular}
	\end{center}
\end{table}
%=====================================================================

%=====================================================================

%=====================================================================
\begin{table}[ht]
\caption{Necessary and sufficient conditions of existence and local stability of steady states of (\ref{AM2}). $S_1^{*}(D)$, $S_2^{i*}(D)$ and 
$H_i(D)$ are defined in Table \ref{tableFunctions}. } \label{TableSumExisStab}
\begin{center}	
\begin{tabular}		{lll}
\hline
                        &  Existence conditions                          &    Stability conditions
                        \\\hline
%=====================================================================
$E_1^0$
&
Always exists
&
$S_{1{\rm in}}<S_1^*(D)$ and $S_{2{\rm in}}\notin\left[S_2^{1*}(D),S_2^{2*}(D)\right]$\\
$E_1^1$
&
$S_{2{\rm in}}>S_2^{1*}(D)$&
$S_{1{\rm in}}<S_1^*(D)$\\
$E_1^2$
&
$S_{2{\rm in}}>S_2^{2*}(D)$
&
Unstable if it exists\\
$E_2^0$
&
$S_{1{\rm in}}>S_1^*(D)$
&
$S_{2{\rm in}}+\frac{k_2}{k_1}S_{1{\rm in}}
\notin\left[H_1(D),H_2(D)\right]$\\
$E_2^1$
&
$S_{1{\rm in}}>S_1^*(D)$ and 
$S_{2{\rm in}}+\frac{k_2}{k_1}S_{1{\rm in}}>H_1(D)$
&
Stable if it exists\\
$E_2^2$
&
$S_{1{\rm in}}>S_1^*(D)$ and 
$S_{2{\rm in}}+\frac{k_2}{k_1}S_{1{\rm in}}>H_2(D)$
& 
Unstable if it exists\\
\hline
\end{tabular}
	\end{center}
\end{table}
%=====================================================================

\begin{rem}
In Table \ref{TableSumExisStab}, since the function $S_1^*$ is defined on $(0,D_1)$, the condition 
$S_{1{\rm in}}>S_1^*(D)$ means 
$0<D<D_1$ and $S_{1{\rm in}}>S_1^*(D)$. Conversely, since by convention $S_1^*(D)=+\infty$ for $D\geq D_1$, the condition  $S_{1{\rm in}}<S_1^*(D)$ means 
$D\geq D_1$ and $S_{1{\rm in}}>0$ or 
$0<D<D_1$ and $0<S_{1{\rm in}}<S_1^*(D)$. On the other hand, since 
the function $S_2^{i*}$ is defined on $(0,D_2)$, the condition 
$S_{2{\rm in}}>S_2^{i*}(D)$ means 
$0<D<D_2$ and $S_{2{\rm in}}>S_2^{i*}(D)$ and, conversely, since by convention 
$S_2^{1*}(D)=+\infty$ for $D> D_2$, the condition 
$S_{2{\rm in}}\notin\left[S_2^{1*}(D),S_2^{2*}(D)\right]$ means 
$D\geq D_2$ and $S_{2{\rm in}}>0$ or
$0<D< D_2$ and $S_{2{\rm in}}\notin\left[S_2^{1*}(D),S_2^{2*}(D)\right]$. Similar remarks can be made concerning the conditions involving functions $H_i(D)$, $i=1,2$.
\end{rem}

\section{Operating diagram}\label{sec3}
Let us consider the surfaces $\Gamma_i$, $i=1\cdots 6$, defined by Table \ref{Gamma}. 
Notice that  $S_2^{1*}(D)<S_2^{2*}(D)$ for $0<D<D_2$ and equality holds for $D=D_2$. Similarly
$H_{1}(D)<H_{2}(D)$ for $0<D<\min(D_1,D_2)$,
and equality holds for $D=\min(D_1,D_2)$. Therefore, the $\Gamma_i$ surfaces separate the operating space 
$(D,S_{1{\rm in}},S_{2{\rm in}})$ into nine regions, denoted $\mathcal{I}_k$, $k=0\cdots8$, and defined in 
Table \ref{Table9Regions}. 
These regions of the operating parameters space 
$(D,S_{1{\rm in}},S_{2{\rm in}})$ are corresponding to different system behaviors, as stated in the following result.  

%%%%%%%%%%%%%%%%%%%%%%%%%%%%%
\begin{table}[ht]
\caption{The surfaces $\Gamma_i$, $i=1\cdots 6$. }\label{Gamma}
\begin{center}	
\begin{tabular}{l}
\hline
$\Gamma_1=
\left\{(D,S_{1{\rm in}},S_{2{\rm in}}):0< D<D_1\mbox{ and }S_{1{\rm in}}=S_1^*(D)\right\}
$
\\
$\Gamma_2=
\left\{(D,S_{1{\rm in}},S_{2{\rm in}}):0< D<D_2\mbox{ and }S_{2{\rm in}}=S_2^{1*}(D)\right\}$
\\
$\Gamma_3=
\left\{(D,S_{1{\rm in}},S_{2{\rm in}}):0< D<D_2\mbox{ and }S_{2{\rm in}}=S_2^{2*}(D)\right\}$
\\
Notice that\\
\qquad$\Gamma_1=
\left\{(D,S_{1{\rm in}},S_{2{\rm in}}):
S_{1{\rm in}}>0\mbox{ and }\alpha D=\mu_1\left(S_{1{\rm in}}\right)\right\}$\\ 
\qquad $\Gamma_2\cup \Gamma_3=
\left\{(D,S_{1{\rm in}},S_{2{\rm in}}):
S_{2{\rm in}}>0\mbox{ and }\alpha D=\mu_2\left(S_{2{\rm in}}\right)\right\}$
\\
$\Gamma_4=
\left\{(D,S_{1{\rm in}},S_{2{\rm in}}):0< D<\min(D_1,D_2),
S_{1{\rm in}}>S_1^*(D)
\mbox{ and }
S_{2{\rm in}}+\frac{k_2}{k_1}S_{1{\rm in}}=H_1(D)\right\}$
\\
$\Gamma_5=
\left\{(D,S_{1{\rm in}},S_{2{\rm in}}):0< D<\min(D_1,D_2),
S_{1{\rm in}}>S_1^*(D) \mbox{ and }
S_{2{\rm in}}+\frac{k_2}{k_1}S_{1{\rm in}}=H_2(D)\right\}$
\\
$\Gamma_6=
\left\{(D,S_{1{\rm in}},S_{2{\rm in}}):D=D_2%\mbox{ and }S_{2{\rm in}}\geq S_2^M
\right\}$\\
\hline
\end{tabular}
\end{center}
\end{table}
%%%%%%%%%%%%%%%%%%%%%%%%%%%%%% 

%=============================
\begin{table}[ht]
\caption{Definitions of the nine regions corresponding to the nine cases in \cite{Benyahia}.}\label{Table9Regions}
\begin{center}
\begin{tabular}{c|c|l}
\hline
Case of \cite{Benyahia}&
{Region} 
  & Definition
  \\ 
  \hline
{\bf 1.1} 
  &$\mathcal{I}_0$
  & $S_{1{\rm in}}\!<\!S_1^*(D)$ and $S_{2{\rm in}}<S_2^{1*}(D)$ 
\\
{\bf 1.2}
&$\mathcal{I}_1$
&$S_{1{\rm in}}\!<\!S_1^*(D)$ and $S_2^{1*}(D)\!<\!S_{2{\rm in}}\!\leq\!S_2^{2*}(D)$
\\
{\bf 1.3}
&$\mathcal{I}_2$
&$S_{1{\rm in}}\!<\!S_1^*(D)$ and 
$S_{2{\rm in}}>S_2^{2*}(D)$ 
\\ \hline
   {\bf 2.1}
&$\mathcal{I}_3$  
  &
  $S_{1{\rm in}}\!>\!S_1^*(D)$ and 
$S_{2{\rm in}}+\frac{k_2}{k_1}S_{1{\rm in}}<H_1(D)$  
\\
{\bf 2.2}
&$\mathcal{I}_4$
&
$S_{1{\rm in}}\!>\!S_1^*(D)$, 
$S_{2{\rm in}}\!\leq\!S_2^{1*}(D)$
and
$H_1(D)\!<S_{2{\rm in}}\!+\!\frac{k_2}{k_1}S_{1{\rm in}}\!\leq\!H_2(D)$ 
\\
{\bf 2.3}
&$\mathcal{I}_5$
&
$S_{1{\rm in}}\!>\!S_1^*(D)$,
$S_{2{\rm in}}\!\leq\!S_2^{1*}(D)$
and
$S_{2{\rm in}}\!+\!\frac{k_2}{k_1}S_{1{\rm in}}\!>\!H_2(D)$ 
\\
{\bf 2.4}
&$\mathcal{I}_6$
&
$S_{1{\rm in}}\!>\!S_1^*(D)$,
$S_{2{\rm in}}\!>\!S_2^{1*}(D)$
and
$S_{2{\rm in}}\!+\!\frac{k_2}{k_1}S_{1{\rm in}}\!\leq\!H_2(D)$ 
\\
{\bf 2.5}
&$\mathcal{I}_7$
&
$S_{1{\rm in}}\!>\!S_1^*(D)$,
$S_2^{1*}(D)<S_{2{\rm in}}\!\leq\!
S_2^{2*}(D)$
and
$S_{2{\rm in}}\!+\!\frac{k_2}{k_1}S_{1{\rm in}}\!>\!H_2(D)$ 
\\
{\bf 2.6}
&$\mathcal{I}_8$
&
$S_{1{\rm in}}\!>\!S_1^*(D)$ and $S_{2{\rm in}}\!>\!S_2^{2*}(D)$   
  \\
 \hline
\end{tabular}  
\end{center}
  \end{table}
%====================================

\begin{propo}                                           \label{proIk}
Assume that Hypotheses 1 and 2 hold. 
The existence and stability properties of the steady states of (\ref{AM2}) are given in Table \ref{Table9cases}, where the regions $\mathcal{I}_k$, $k=0\cdots 8$ 
are defined in Table \ref{Table9Regions}.
\end{propo}
%=====================================================================
\begin{proof}
The proof is given in Appendix \ref{ProofproIk}.	
\end{proof}
%%%%%%%%%%%%%%%%%%%%%%%%

\begin{table}[ht]
\caption{Existence and stability of steady states of (\ref{AM2}) in the nine regions of the operating space. GAS, S and U stand for
{\em Globally asympottically stable}, {\em Locally exponentially stable} and {\em Unstable} respectively. The last colum show the color in which the region is depiced in Figs.  \ref{S1inS2inBenyahia6},  \ref{S1inS2inBenyahia5},  \ref{S1inS2inBenyahia4}, \ref{DS1inBenyahia5}, \ref{DS1inBenyahia4}  and \ref{figBifS1inBenyahia5S2in0}.}\label{Table9cases}
\begin{center}
\begin{tabular}{c|cccccc|l}
\hline
{\bf Region} 
 &$E_{1}^{0}$ &$E_{1}^{1}$ & $E_{1}^{2}$ & $E_{2}^{0}$ & $E_{2}^{1}$ & 
 $E_{2}^{2}$
  & Color
  \\ 
  \hline
$\mathcal{I}_0$
  & GAS& &&&&
  &Red\\
$\mathcal{I}_1$
& U& GAS&&&&
&Blue\\
$\mathcal{I}_2$
&S&S&U&&&
&Cyan\\
$\mathcal{I}_3$  
& U & & & GAS & & 
& Yellow \\
$\mathcal{I}_4$
& U & & & U & GAS & 
& Green\\
$\mathcal{I}_5$
& U & & & S & S & U 
& Pink\\
$\mathcal{I}_6$
& U & U && U & GAS & 
& Green\\
$\mathcal{I}_7$
& U & U & & S & S & U 
& Pink \\
$\mathcal{I}_8$
& U & U & U & S & S & U 
& Pink  
  \\
 \hline
\end{tabular}  
\end{center}
  \end{table}
%%%%%%%%%%%%%%%%%

\begin{rem}\label{remarkcolors}
In in Figs.  \ref{S1inS2inBenyahia6},  \ref{S1inS2inBenyahia5},  \ref{S1inS2inBenyahia4}, \ref{DS1inBenyahia5} and \ref{DS1inBenyahia4} presenting operating diagrams, a region is colored according to the color in 
Table \ref{Table9cases}. 
Each color corresponds to different asymptotic behavior: 
\begin{itemize}
\item
Red for the washout of both species, that is, the steady state $E_1^0$ is Globally asymptotically stable (GAS), which occurs in region $\mathcal{I}_0$.
\item 
Blue for the washout of acidogenic bacteria while methanogenic bacteria are maintained, that is, the steady state $E_1^1$ is GAS, which occurs in region $\mathcal{I}_1$. 
\item
Cyan for the bistability of
$E_1^0$ and $E_1^1$ which are both (locally) stable. This behavior  occurs in region $\mathcal{I}_1$. Depending on the initial condition the system can go to the washout of both species or the washout of only the acidogenic bacteria.
\item 
Yellow for the washout of methanogenic bacteria while acidogenic bacteria are maintained, that is the steady state $E_2^0$ is GAS, which occurs in region $\mathcal{I}_3$.
\item 
Green for the global asymptotic stability of the positive steady state $E_2^1$, which occur in  $\mathcal{I}_4$ and  $\mathcal{I}_6$. These regions differ only by the existence, in the second region, of the unstable boundary steady state $E_1^1$.
\item
Pink for the bistability of
$E_2^0$ and $E_2^1$ which are both locally asymptotically stable. This behavior  occurs in regions $\mathcal{I}_5$, $\mathcal{I}_7$ and $\mathcal{I}_8$. 
These regions differ only by the possible existence of the unstable boundary steady states $E_1^1$ or $E_1^2$.
Depending on the initial condition the system can go to the washout of methanogenic bacteria or the coexistence of both species.
\end{itemize}
\end{rem}
%%%%%%%%%%%%%%%
%======================================
\begin{figure}[ht]
\setlength{\unitlength}{1.0cm}
\begin{center}
\begin{picture}(8.5,6.5)(0,-0.2)		
\put(-4,0){{\includegraphics[scale=0.25]{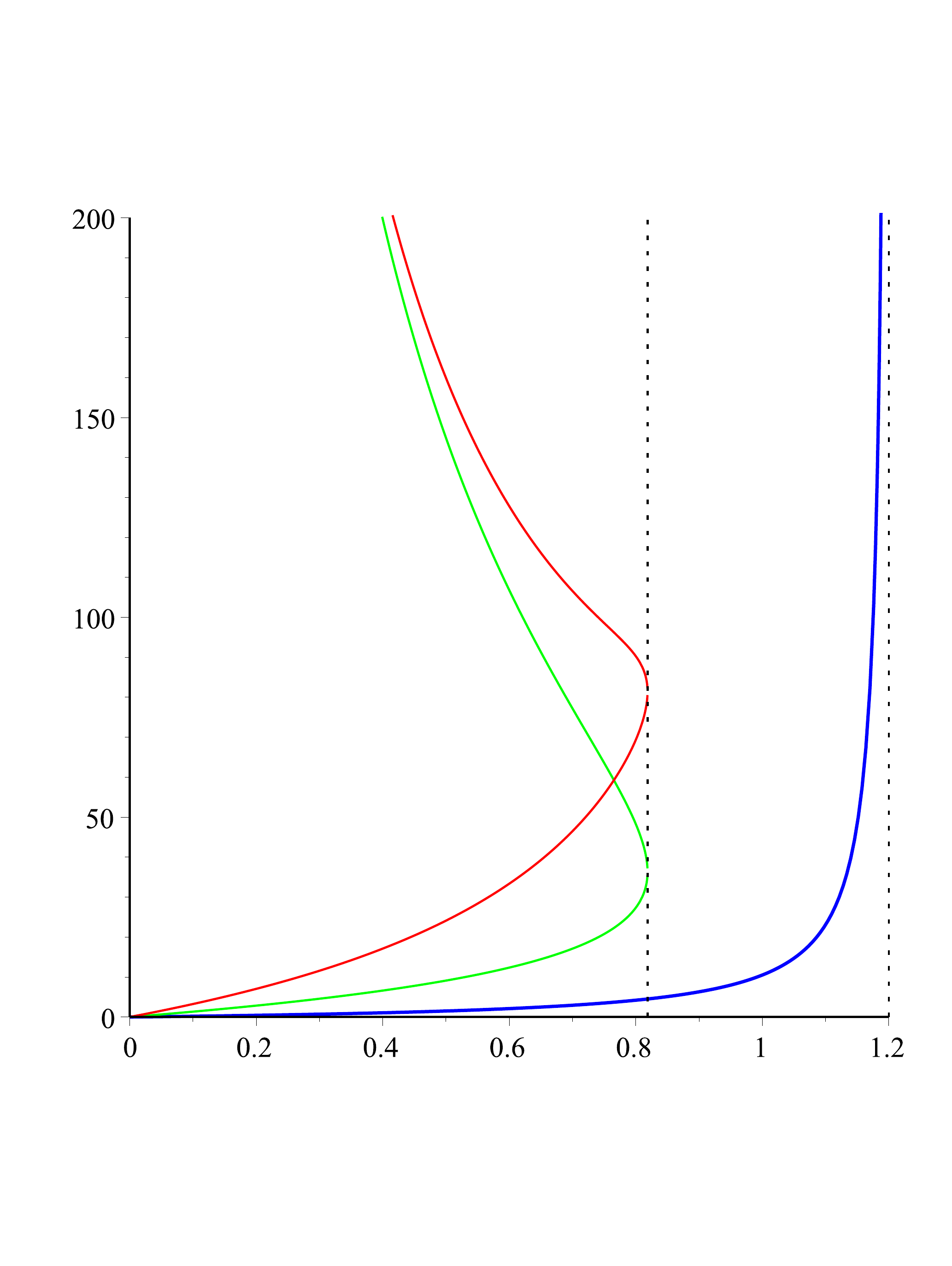}}}
\put(1.5,0){{\includegraphics[scale=0.25]{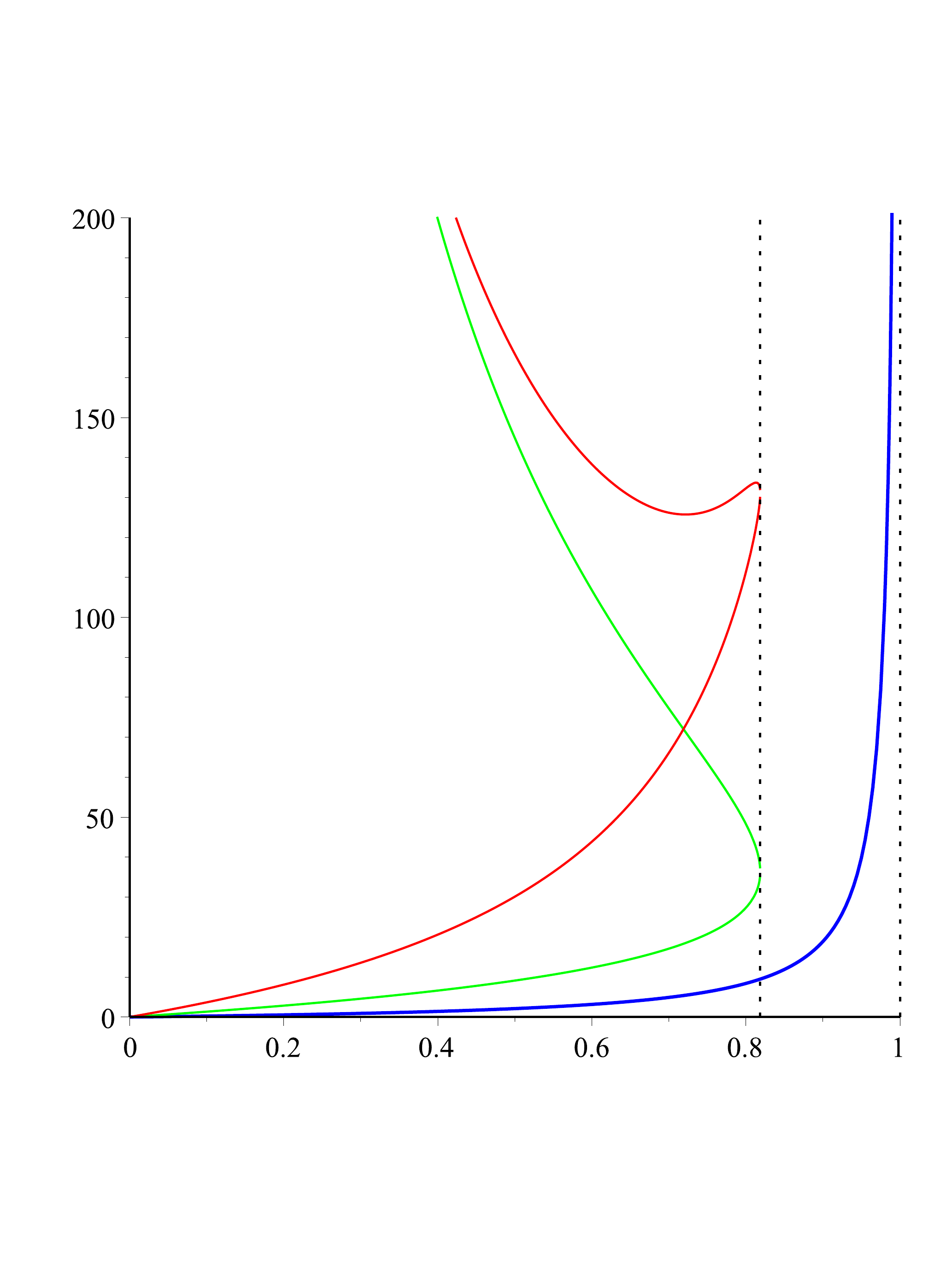}}}
\put(7,0){{\includegraphics[scale=0.25]{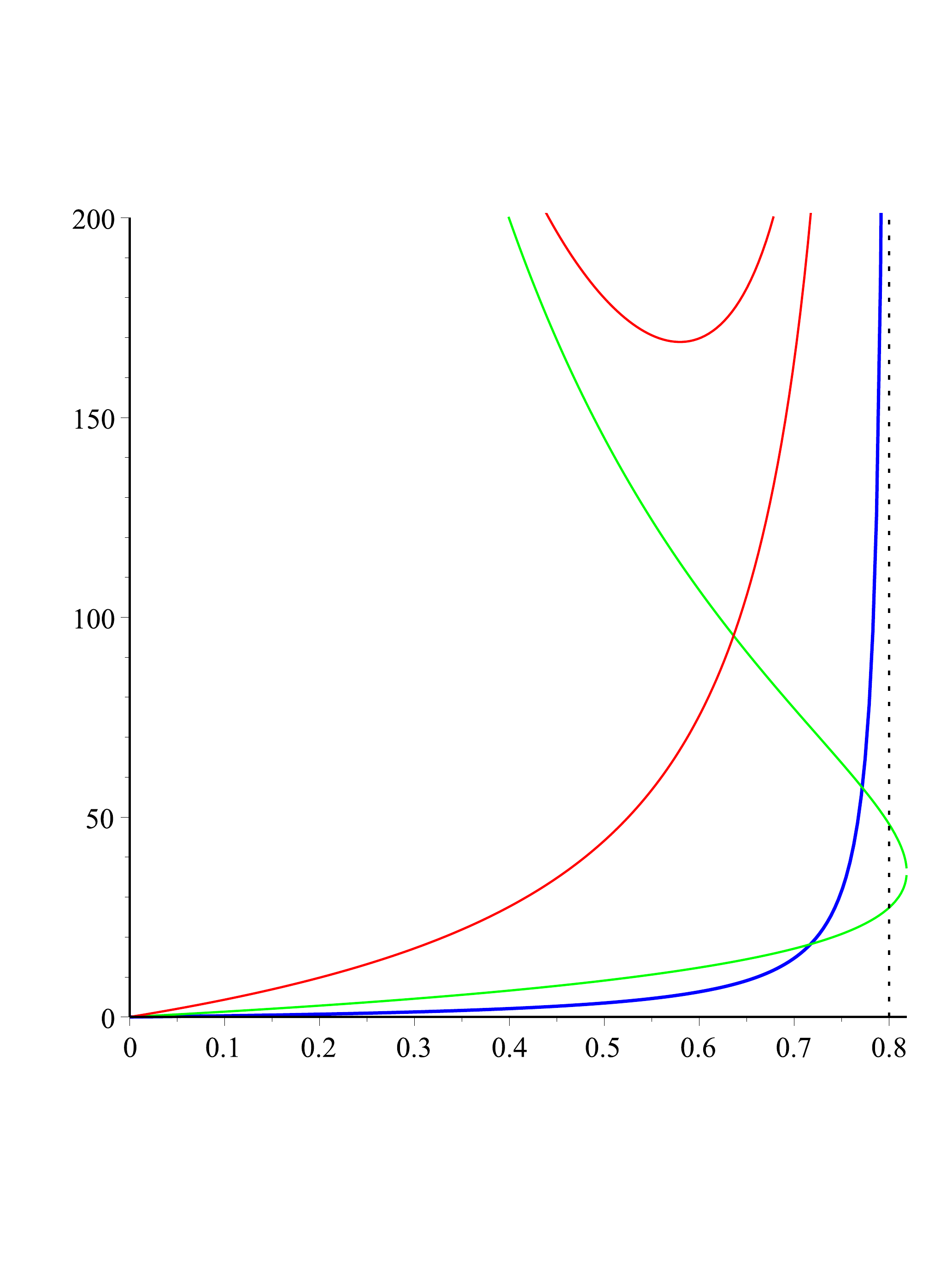}}}
%==============================================
\put(-3.5,0.5){(a) $D_1>D_2$, $dH_2\!/dD<0$}
\put(1.5,0.5){(b) $D_1>D_2$, $H_2$ non monotonous }	
\put(9,0.5){(c) $D_1<D_2$}	
\put(-1,0.9){$D$}
\put(4.5,0.9){$D$}	
\put(10,0.9){$D$}	
\put(-3.7,4){$S$}
\put(1.8,4){$S$}	
\put(7.3,4){$S$}
%%%%%%%%%%%%%%%%%%%%%
\put(-1.3,1.6){{$S_2^{1*}$}}
\put(-2,4){{$S_2^{2*}$}}
\put(0.4,4){{$S_1^{*}$}}
\put(-1.7,2){{$H_{1}$}}
\put(-1.6,5.2){{$H_{2}$}}
%%%%%%%%%%%%%%%%%%%%%%%%%%%%%%
\put(4.9,1.7){{$S_2^{1*}$}}
\put(3.9,4){{$S_2^{2*}$}}
\put(5.9,4){{$S_1^{*}$}}
\put(3.8,2){{$H_{1}$}}
\put(4.3,5.2){{$H_{2}$}}
%%%%%%%%%%%%%%%%%%%%%%%%
\put(10.2,1.7){{$S_2^{1*}$}}
\put(9.9,4){{$S_2^{2*}$}}
\put(11.4,4){{$S_1^{*}$}}
\put(9.4,2){{$H_{1}$}}
\put(10.5,5.2){{$H_{2}$}}
\end{picture}
\end{center}
\vspace{-0.9cm}
\caption{The graphs of functions 
$S=S_1^*(D)$ (in Blue), $S=S_2^{i*}(D)$, $i=1,2$ (in Green) and
$S=H_{i}(D)3$, $i=1,2$ (in Red). (a): $m_1=0.6$; (b): $m_1=0.5$ (c): $m_1=0.4$. Other biological parameter values are given in Table \ref{parametervalues}. Compare with Fig.~4 of \cite{Sbarciog}} \label{figHi}
\end{figure}
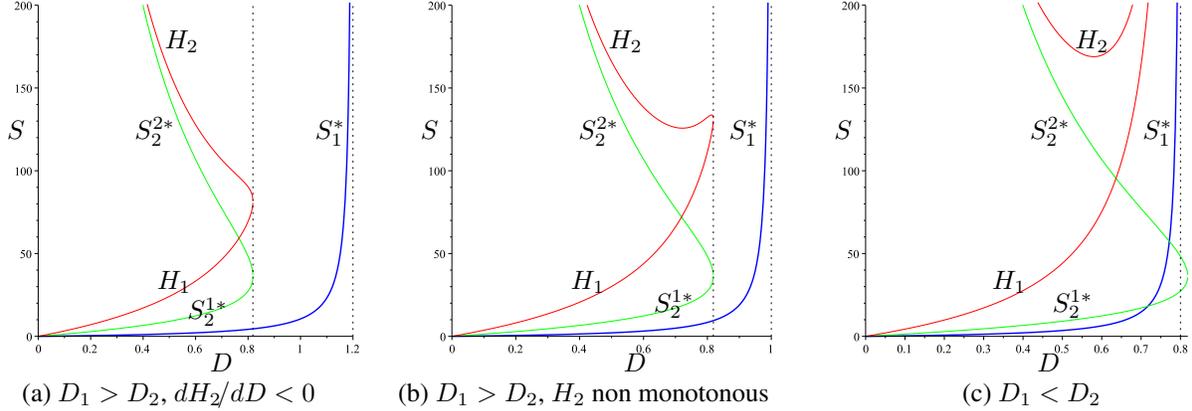
%=======================================

%%%%%%%%%%%%%%%%%%%%%%
\begin{table}[ht]
\caption{Three behaviors for functions $H_i$, $i=1,2$}\label{Table3cases} 
\begin{center}
\begin{tabular}{l}
\hline
Case (A),  where $D_1>D_2$ and $dH_2\!/dD<0$.\\
Case (B), where $D_1>D_2$ and $H_2$ non monotonous.\\
Case (C), where where $D_1<D_2$.
\\	 
\hline
\end{tabular}
\end{center}
\end{table}
%%%%%%%%%%%%%%%%%%%%%

%===========================
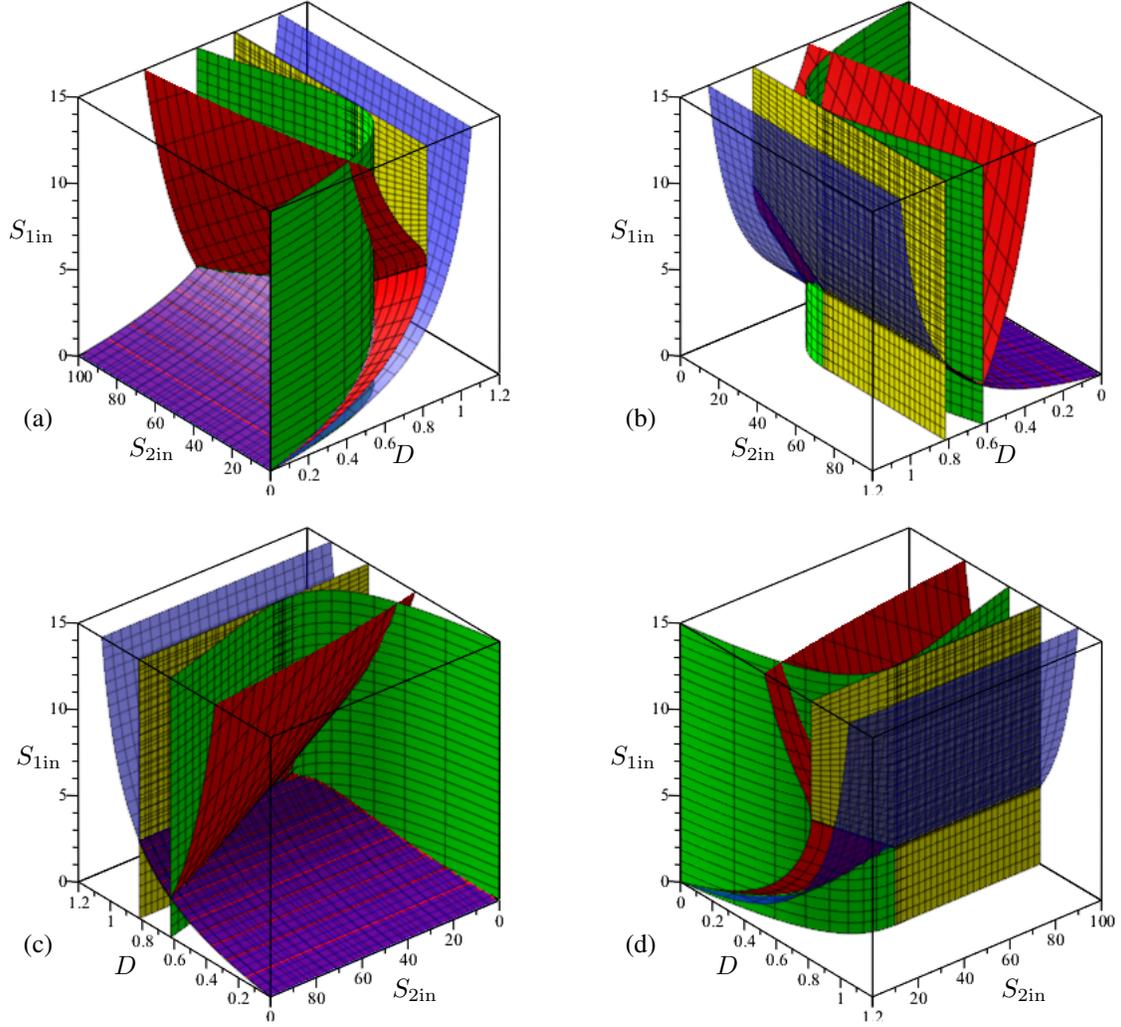
\begin{figure}[ht]
\setlength{\unitlength}{1.0cm}
\begin{center}
\begin{picture}(8.5,14.5)(0,0)		
\put(-3,7.5){{\includegraphics[scale=0.5]{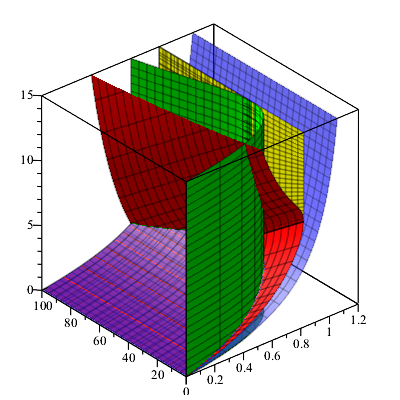}}}
\put(5,7.5){{\includegraphics[scale=0.5]{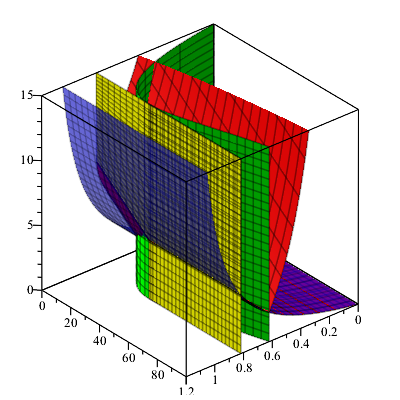}}}
\put(-3,0.5){{\includegraphics[scale=0.5]{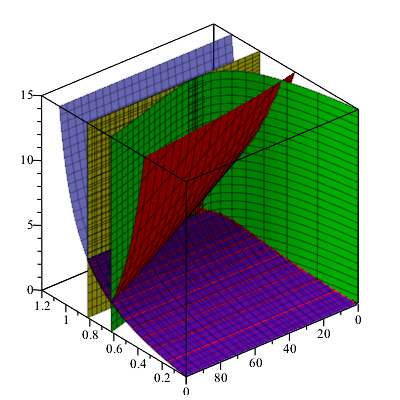}}}
\put(5,0.5){{\includegraphics[scale=0.5]{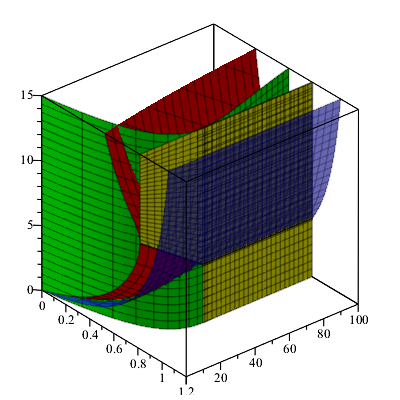}}}
%==============================================
\put(-3,8.5){(a)}
\put(5,8.5){(b)}	
\put(-3,1.5){(c)}
\put(5,1.5){(d)}	
\put(1.9,8){$D$}
\put(9.9,8){$D$}	
\put(-1.8,1.2){$D$}
\put(6.2,1.2){$D$}	
\put(-3.2,11){$S_{1{\rm in}}$}
\put(4.8,11){$S_{1{\rm in}}$}	
\put(-3.1,4){$S_{1{\rm in}}$}
\put(4.8,4){$S_{1{\rm in}}$}	
\put(-1.6,8.1){$S_{2{\rm in}}$}
\put(6.4,8.1){$S_{2{\rm in}}$}	
\put(1.9,0.9){$S_{2{\rm in}}$}
\put(10,0.9){$S_{2{\rm in}}$}	
\end{picture}
\end{center}
\vspace{-0.9cm}
\caption{The surfaces $\Gamma_1$ (in Blue), $\Gamma_2$ and $\Gamma_3$ (in Green),
$\Gamma_4$ and $\Gamma_5$ (in Red) and 
$\Gamma_6$ (in Yellow), corresponding to Fig.~\ref{figHi}(a). The surfaces separate the 3-dimensional operating space 
$\left(D,S_{1{\rm in}},S_{2{\rm in}}\right)$ in 9 regions $\mathcal{I}_k$, $k=0\cdots8$. Front (a), rear (b), left (c) and right (d) view of the surfaces 
$\Gamma_i$. Compare with Fig.~6 of \cite{Sbarciog}.} \label{figOD3D}
\end{figure}
%%%%%%%%%%%%%%%%%%%%%%%

The operating diagram highly depends on the shape of $\Gamma_4$ and $\Gamma_5$ surfaces, that is to say, on the behaviors of functions $H_i$, $i=1,2$, defined in Table \ref{tableFunctions}. 
Notice that these functions are defined on 
$(0,\min(D_1,D_2))$ and
$H_1$ is increasing, since it is the sum of two increasing functions. 
We have
$$
\lim_{D\to0}H_1(D)=0,
\quad
\lim_{D\to0}H_2(D)=+\infty,
\quad
\lim_{D\to0}\frac{dH_2}{dD}(D)=-\infty,
$$
For the limits at right of the domain of definition of these functions,  
we must distinguish two cases:
\begin{itemize}
\item
When $D_1<D_2$, the functions $H_i$, $i=1,2$ are defined on $(0,D_1)$ and 
$$
\lim_{D\to D_1}H_1(D)=
\lim_{D\to D_1}H_2(D)=+\infty,
$$ 
\item
When $D_2<D_1$, the functions $H_i$, $i=1,2$ are defined on $(0,D_2)$ and 
$$
\lim_{D\to D_2}H_1(D)=
\lim_{D\to D_2}H_2(D)=S_2^M+\frac{k_2}{k_1}S_1^*(D_2),
\quad
\lim_{D\to D_2}\frac{dH_1}{dD}(D)=+\infty,
\quad
\lim_{D\to D_2}\frac{dH_2}{dD}(D)=-\infty.
$$ 
Two qualitatively different sub-cases can be distinguished: either $H_2$ is  decreasing on $(0,D_2)$ or it is not monotonous. Since $H_2$ is decreasing near the extremities of its definition interval, a typical example  is where it is decreasing, then increasing and then decreasing.
\end{itemize}

Therefore there are three cases summarized in Table~\ref{Table3cases} and illustrated in Fig.~\ref{figHi}. The role of $H_i$-functions, in the description of the operating diagram, has already been highlighted, see Fig. 4 in \cite{Sbarciog}, where cases $D_2<D_1$ and $D_1<D_2$ are distinguished.

%%%%%%%%%%%%%%%%%%%%%%%%%% 
Since the surfaces $\Gamma_i$, $i=1\cdots 6$, which are the boundaries of the various regions have been derived analytically, the operating diagrams can be drawn qualitatively in each of these cases.  
Instead of giving a general qualitative description of the operating diagram, and without loss of generality, we present the specific examples shown in Fig.~\ref{figHi}. These examples are obtained with the Monod and Haldane functions \ref{MonodHaldane}. Notice that these functions satisfy Hypotheses 1 and 2. Therefore, the results of Propositons \ref{proSSi} and \ref{proIk} apply. The analytical expressions of the auxiliary functions defined in Table \ref{tableFunctions} and needed in the defintions of the regions $\mathcal{I}_k$ of the operating diagrams are given in Table \ref{AuxiliaryFunctionsMonodHaldane}, in the particular case of functions \ref{MonodHaldane}.  The biological parameter values used in the figures are given in Table \ref{parametervalues}. For the sake of practical applicability, these parameter values were chosen in a range that can be found in the literature \cite{Benyahia,Bernard}.

For the biological parameter values corresponding Fig. \ref{figHi}(a), the surfaces $\Gamma_i$, $i=1\cdots6$ are  shown in Fig.~\ref{figOD3D}.
It is difficult to visualize the regions 
$\mathcal{I}_k$, $k=0\cdots8$ of the three-dimensional operating diagram.  
We can have a better understanding of these regions by showing cuts along 2 dimensional planes where one of the operating parameters is kept constant. For instance, if $D$ is kept constant, we obtain then the operating diagram in the 2-dimensional plane 
$\left(S_{1{\rm in}},S_{2{\rm in}}\right)$. These operating diagrams are described in section \ref{ODDconstant}. If $S_{2{\rm in}}$ is kept constant, we obtain then the operating diagram in the 2-dimensional plane $\left(D,S_{1{\rm in}}\right)$. These operating diagrams are described in section \ref{ODS2inconstant}.

%==================================
\begin{figure}[th]
\setlength{\unitlength}{1.0cm}
\begin{center}
\begin{picture}(8.5,14.6)(0,0.3)		
\put(-3,6.5){{\includegraphics[scale=0.34]{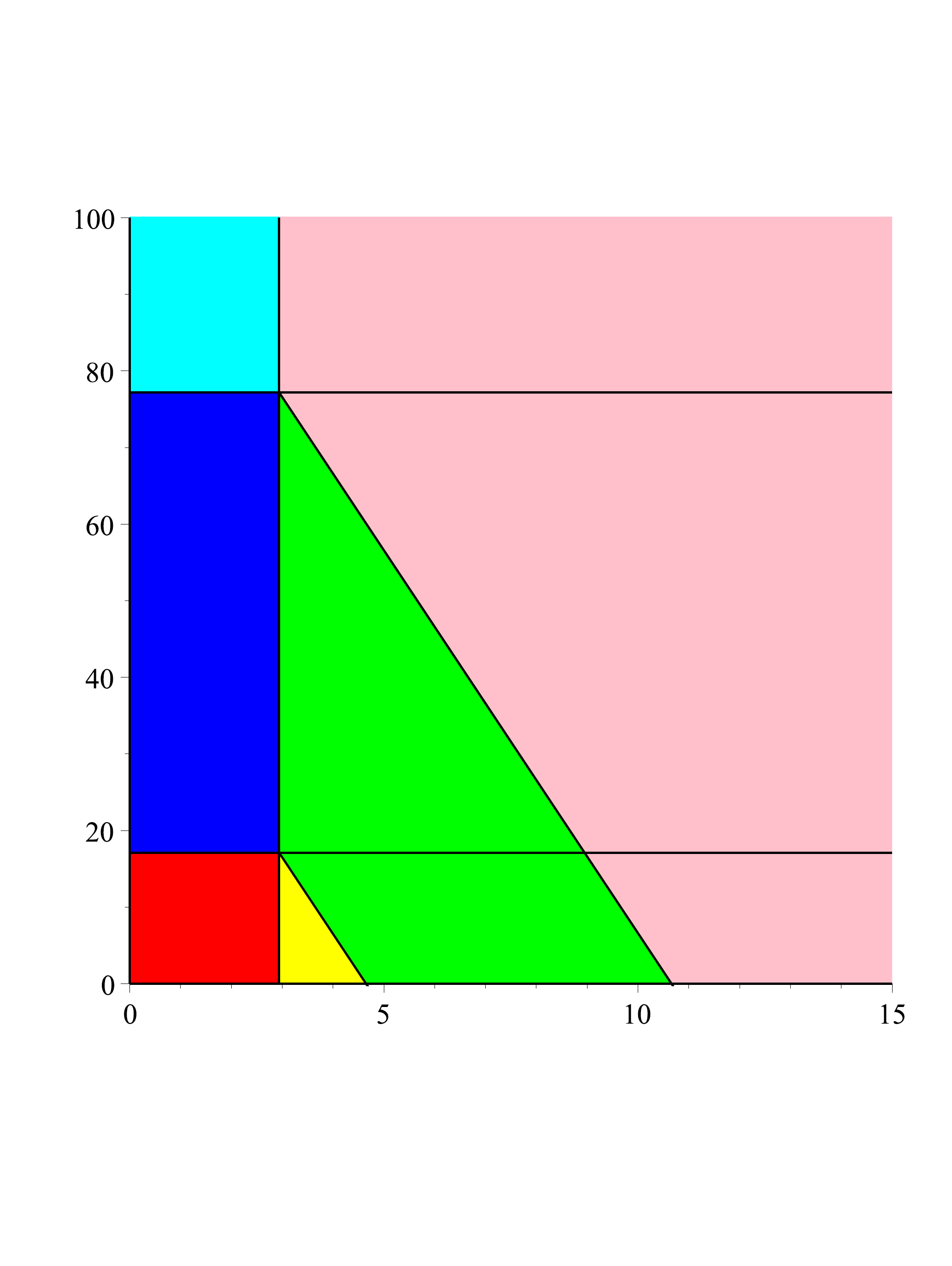}}}
\put(5,6.5){{\includegraphics[scale=0.34]{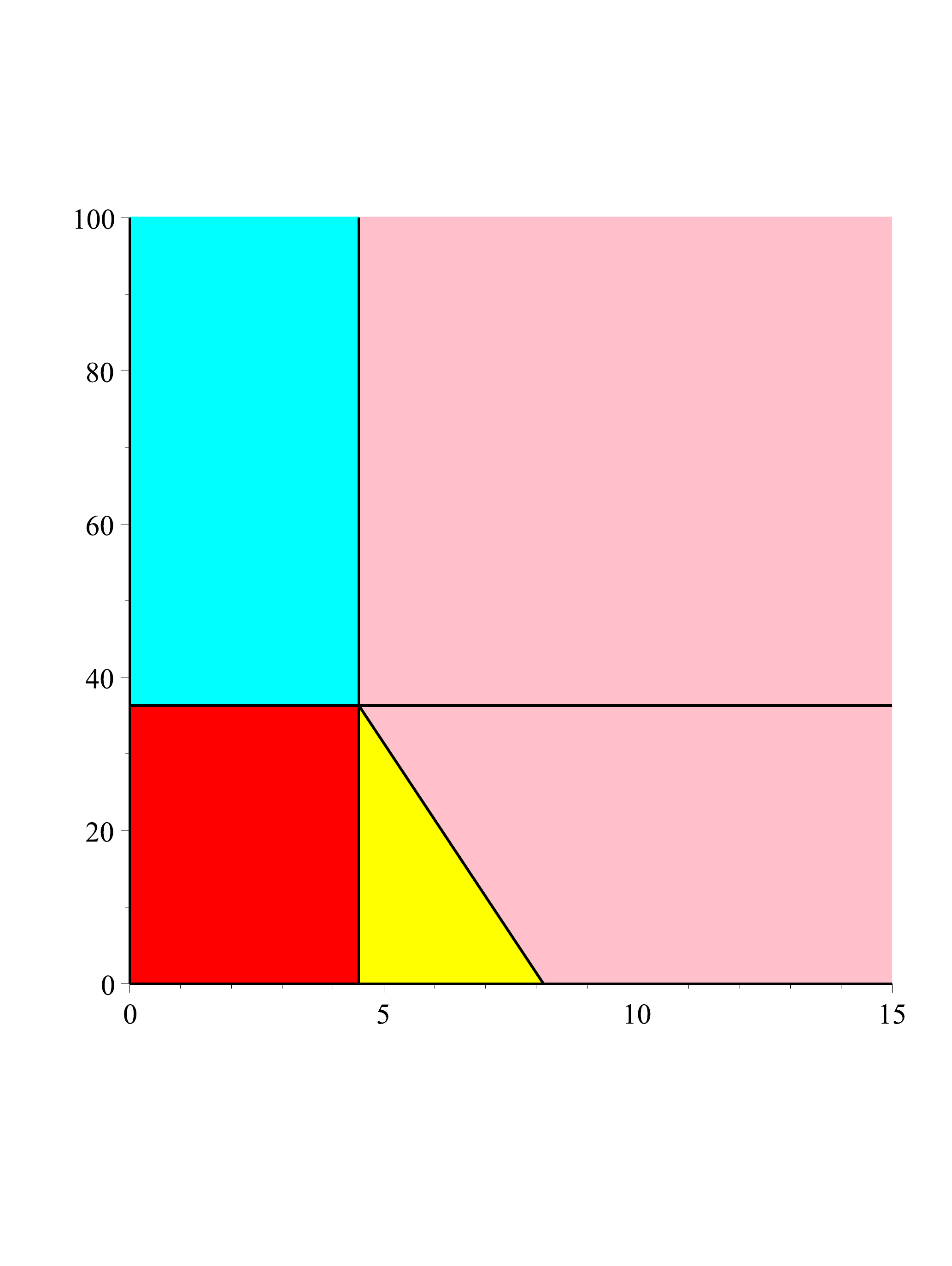}}}
\put(-3,-0.5){{\includegraphics[scale=0.34]{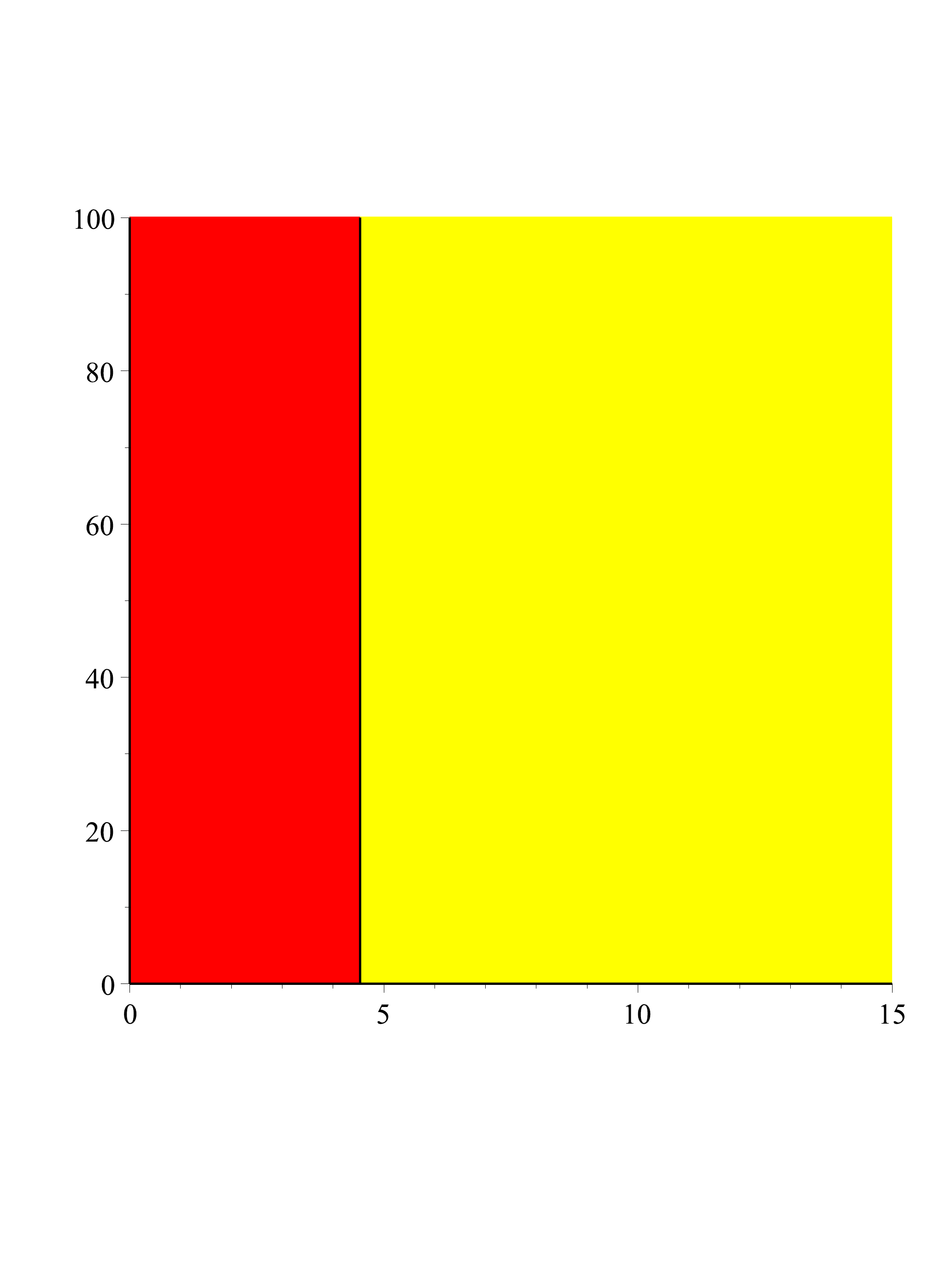}}}
\put(5,-0.5){{\includegraphics[scale=0.34]{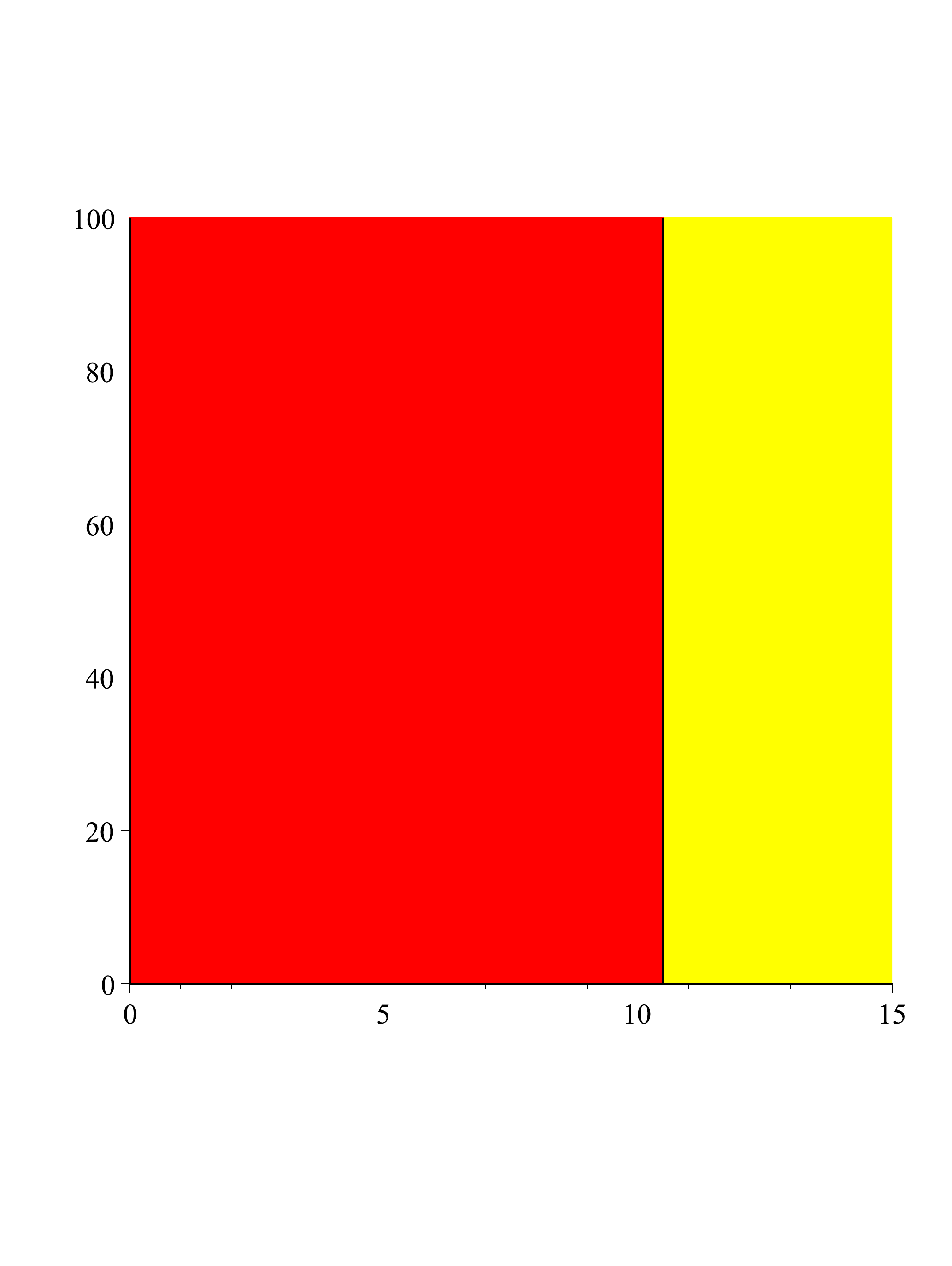}}}
%==============================================
\put(-3,8.5){(a)}
\put(5,8.5){(b)}	
\put(-3,1.5){(c)}
\put(5,1.5){(d)}		
\put(-2.9,11.4){$S_{2{\rm in}}$}
\put(5.2,11.4){$S_{2{\rm in}}$}	
\put(-2.9,4.4){$S_{2{\rm in}}$}
\put(5.2,4.4){$S_{2{\rm in}}$}	
\put(0.8,8.1){$S_{1{\rm in}}$}
\put(8.8,8.1){$S_{1{\rm in}}$}	
\put(0.8,1.1){$S_{1{\rm in}}$}
\put(8.8,1.1){$S_{1{\rm in}}$}	
\put(-1.7,9){$\mathcal{I}_0$}
\put(6.6,9.5){$\mathcal{I}_0$}
\put(-1.4,4.5){$\mathcal{I}_0$}
\put(8,4.5){$\mathcal{I}_0$}
\put(-1.7,11){$\mathcal{I}_1$}
\put(-1.7,13.5){$\mathcal{I}_2$}
\put(6.6,12.5){$\mathcal{I}_2$}	
\put(-0.9,8.8){$\mathcal{I}_3$}
\put(7.8,9.5){$\mathcal{I}_3$}
\put(1.5,4.5){$\mathcal{I}_3$}
\put(10.5,4.5){$\mathcal{I}_3$}
\put(0.4,9){$\mathcal{I}_4$}
\put(2.5,9){$\mathcal{I}_5$}
\put(9.5,9.5){$\mathcal{I}_5$}
\put(-0.5,11){$\mathcal{I}_6$}
\put(1.8,11){$\mathcal{I}_7$}
\put(1,13.5){$\mathcal{I}_8$}
\put(9.5,12.5){$\mathcal{I}_8$}
\put(-1.1,14.5){$\Gamma_1$}
\put(7.5,14.5){$\Gamma_1$}
\put(-0.5,7.5){$\Gamma_1$}
\put(9.8,7.5){$\Gamma_1$}
\put(3.8,9.5){$\Gamma_2$}
\put(3.8,13){$\Gamma_3$}
\put(10,10.9){{$\Gamma_2\!\approx\!\Gamma_3$}}
\put(11.7,10.6){$S_2^M$}
\put(-0.6,9.2){$\Gamma_4$}
\put(8.3,8.8){{$\Gamma_4\!\approx\!\Gamma_5$}}
\put(0.5,10.2){$\Gamma_5$}
\end{picture}
\end{center}
\vspace{-0.9cm}
\caption{The 2-dimensional operating diagram 
$\left(S_{1{\rm in}},S_{2{\rm in}}\right)$ obtained by cuts at  $D$ constant of the 3-dimensional operating diagram shown in Fig.~\ref{figOD3D}. 
(a): $D=0.7$; 
(b): $D=0.818557<D_2$; 
(c): $D=0.82>D_2$; 
(d): $D=1<D_1$. Here $D_1=1.2$,  
$D_2\approx 0.818557467$ and $S_2^M\approx 36.332$.}\label{S1inS2inBenyahia6}
\end{figure}
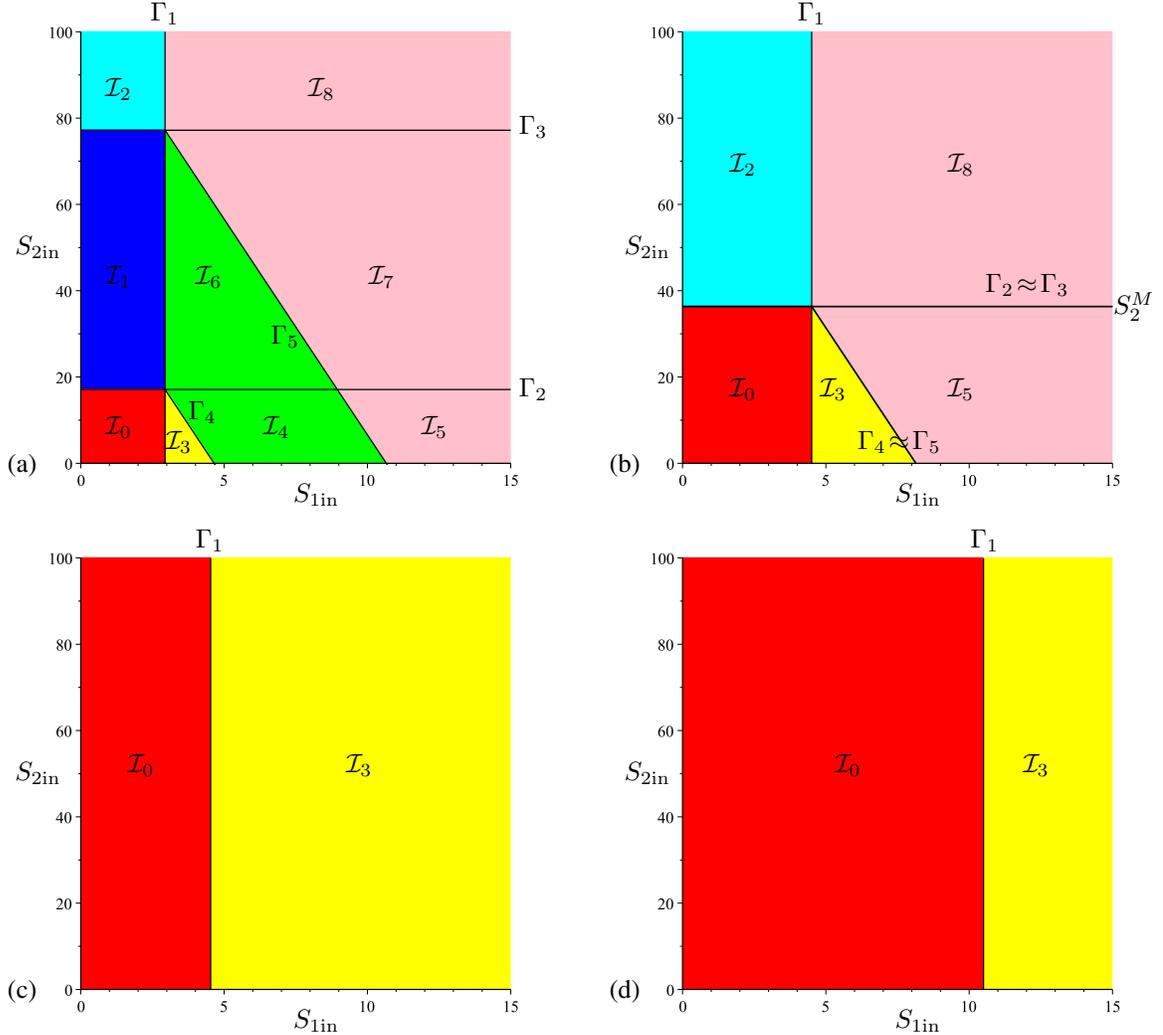

%==================================
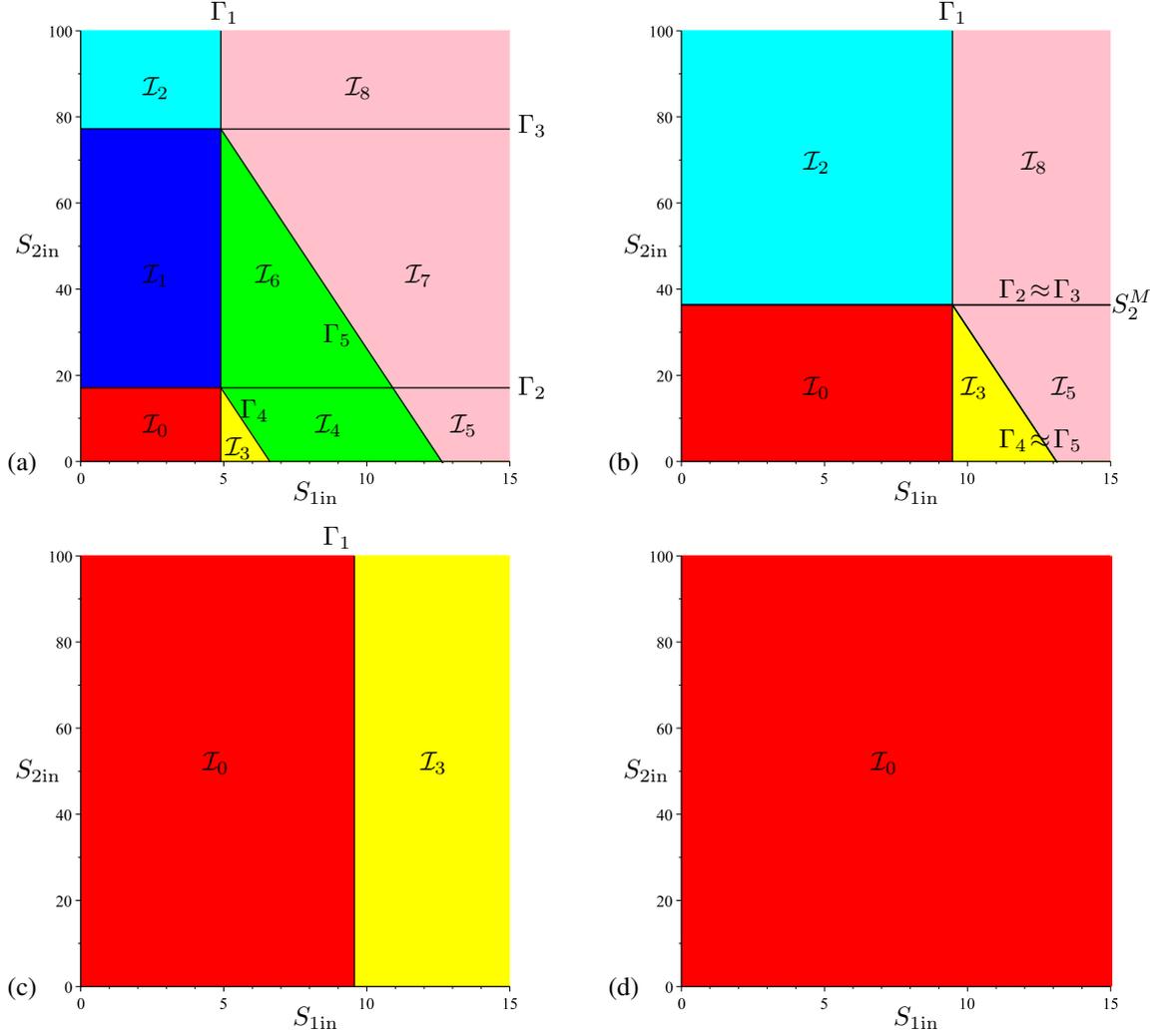
\begin{figure}[th]
\setlength{\unitlength}{1.0cm}
\begin{center}
\begin{picture}(8.5,14.6)(0,0.3)		
\put(-3,6.5){{\includegraphics[scale=0.34]{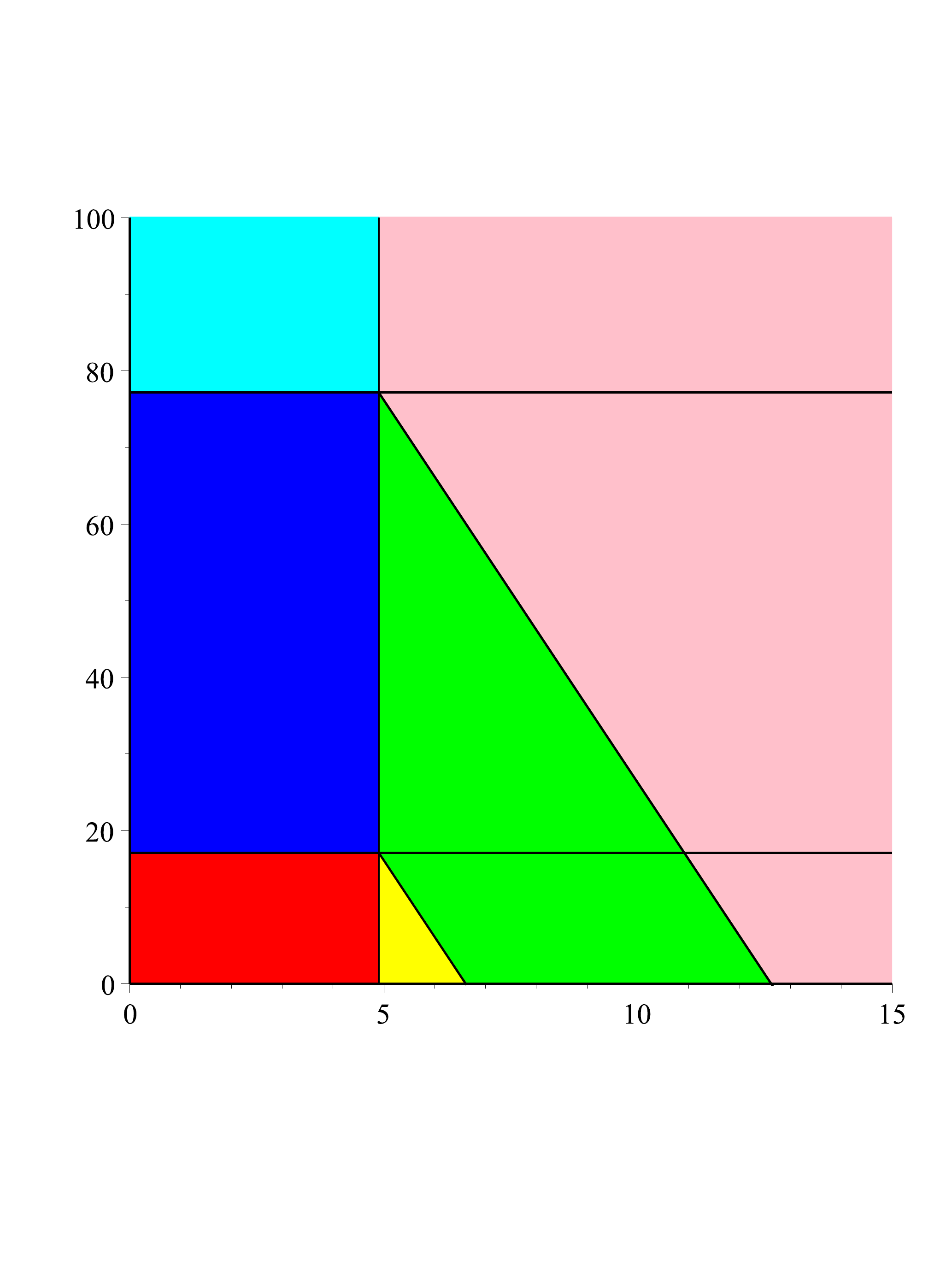}}}
\put(5,6.5){{\includegraphics[scale=0.34]{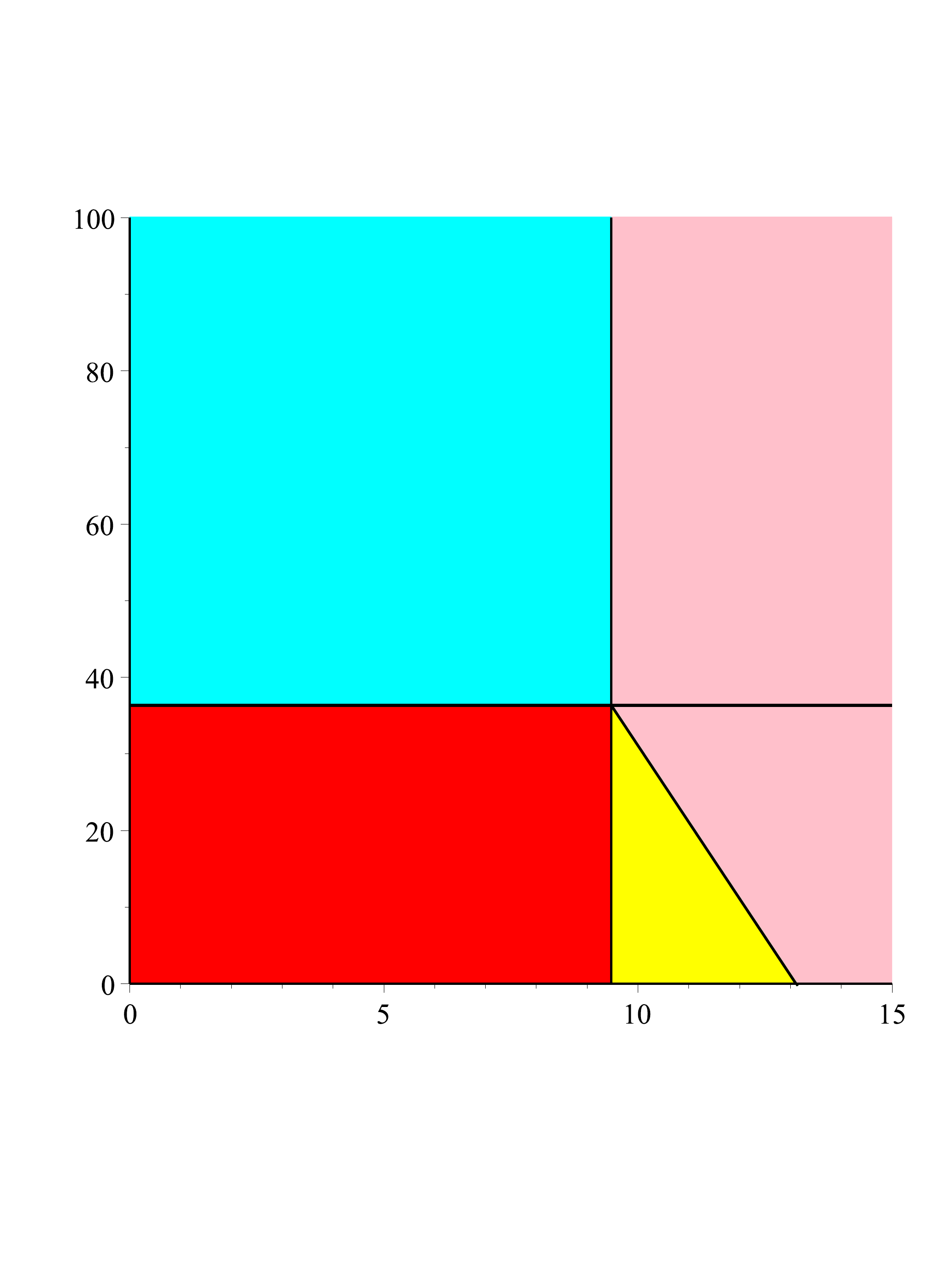}}}
\put(-3,-0.5){{\includegraphics[scale=0.34]{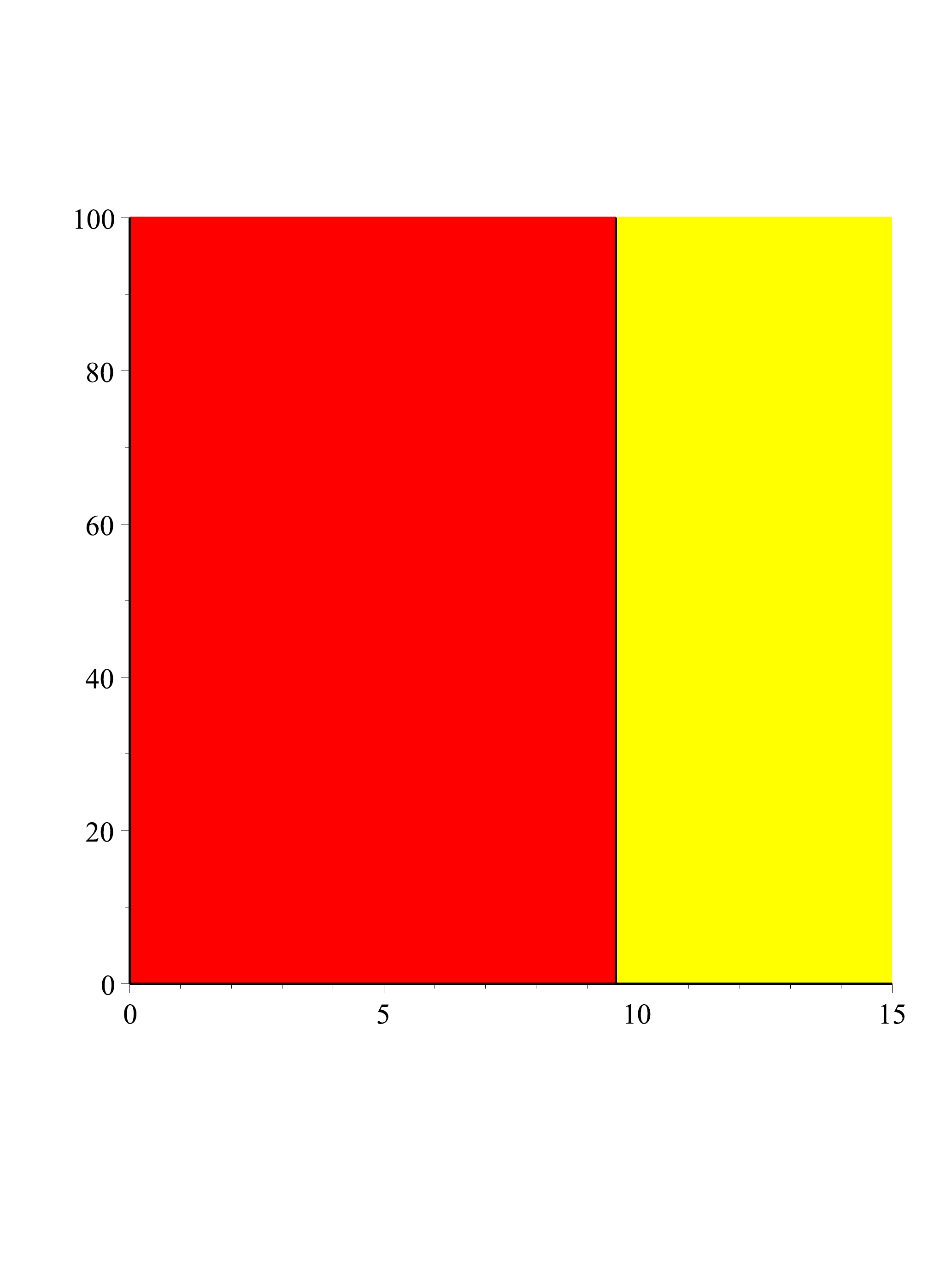}}}
\put(5,-0.5){{\includegraphics[scale=0.34]{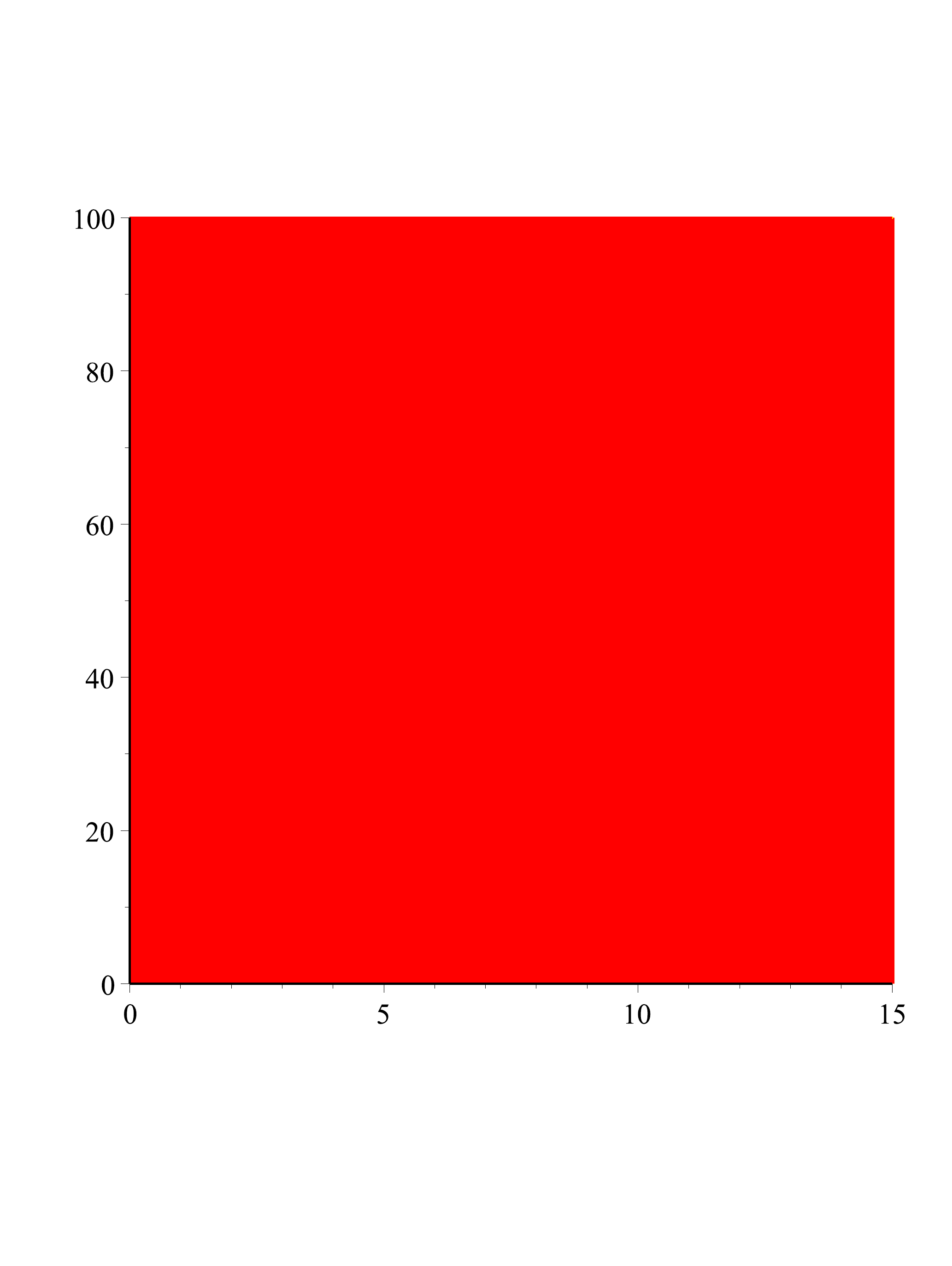}}}
%==============================================
\put(-3,8.5){(a)}
\put(5,8.5){(b)}	
\put(-3,1.5){(c)}
\put(5,1.5){(d)}		
\put(-2.9,11.4){$S_{2{\rm in}}$}
\put(5.2,11.4){$S_{2{\rm in}}$}	
\put(-2.9,4.4){$S_{2{\rm in}}$}
\put(5.2,4.4){$S_{2{\rm in}}$}	
\put(0.8,8.1){$S_{1{\rm in}}$}
\put(8.8,8.1){$S_{1{\rm in}}$}	
\put(0.8,1.1){$S_{1{\rm in}}$}
\put(8.8,1.1){$S_{1{\rm in}}$}	
\put(-1.2,9){$\mathcal{I}_0$}
\put(7.6,9.5){$\mathcal{I}_0$}
\put(-0.4,4.5){$\mathcal{I}_0$}
\put(8.5,4.5){$\mathcal{I}_0$}
\put(-1.2,11){$\mathcal{I}_1$}
\put(-1.2,13.5){$\mathcal{I}_2$}
\put(7.6,12.5){$\mathcal{I}_2$}	
\put(-0.1,8.7){$\mathcal{I}_3$}
\put(9.7,9.5){$\mathcal{I}_3$}
\put(2.5,4.5){$\mathcal{I}_3$}
\put(1.1,9){$\mathcal{I}_4$}
\put(2.9,9){$\mathcal{I}_5$}
\put(10.9,9.5){$\mathcal{I}_5$}
\put(0.3,11){$\mathcal{I}_6$}
\put(2.3,11){$\mathcal{I}_7$}
\put(1.5,13.5){$\mathcal{I}_8$}
\put(10.5,12.5){$\mathcal{I}_8$}
\put(-0.3,14.5){$\Gamma_1$}
\put(9.4,14.5){$\Gamma_1$}
\put(1.2,7.5){$\Gamma_1$}
\put(3.8,9.5){$\Gamma_2$}
\put(3.8,13){$\Gamma_3$}
\put(10.2,10.8){{$\Gamma_2\!\approx\!\Gamma_3$}}
\put(11.7,10.6){$S_2^M$}
\put(0.1,9.2){$\Gamma_4$}
\put(10.2,8.8){{$\Gamma_4\!\approx\!\Gamma_5$}}
\put(1.2,10.2){$\Gamma_5$}
\end{picture}
\end{center}
\vspace{-0.9cm}
\caption{The 2-dimensional operating diagram 
$\left(S_{1{\rm in}},S_{2{\rm in}}\right)$ with $D$ constant, corresponding to Fig.~\ref{figHi}(b). 
(a): $D=0.7$; 
(b): $D=0.818557<D_2$; 
(c): $D=0.82>D_2$; 
(d): $D=1\geq D_1$.
Here $D_1=1$, $D_2\approx 0.818557467$ and $S_2^M\approx 36.332$.
}\label{S1inS2inBenyahia5}
\end{figure}

\section{Operating diagram in $\left(S_{1{\rm in}},S_{2{\rm in}}\right)$ where $D$ is kept constant}
\label{ODDconstant}
The intersections of the surfaces 
$\Gamma_i$, $i=1\cdots5$ with a plane where $D$ is kept constant are straight lines: vertical line for $\Gamma_1$, 
horizontal lines for $\Gamma_2$ and 
$\Gamma_3$ and oblique lines for 
$\Gamma_4$ and $\Gamma_5$, see Table \ref{IntersectionD}. 
%%%%%%%%%%%%%%%%%%%%%%%%%%%%%
%%%%%%%%%%%%%%%%%%%%%%%%%%%%%% 
These straight lines separate the operating parameter plane
$\left(S_{1{\rm in}},S_{2{\rm in}}\right)$ in up to nine regions
$\mathcal{I}_k$, $k=0\cdots8$.
Since the curves are straight lines, the regions of the operating diagram are very easy to picture. We begin by considering the case where $D_2<D_1$ corresponding to Figs. \ref{figHi}(a) and \ref{figHi}(b).

\subsection{Operating diagram when $D_2<D_1$}
The cuts at $D$ constant of the 3-dimensional operating diagram shown in Fig.~\ref{figOD3D} and corresponding to Fig.~\ref{figHi}(a), are shown in Fig. \ref{S1inS2inBenyahia6}. The regions are colored according to the colors in 
Table~~\ref{Table9cases}.  
For the clarity of the picture all straight lines $\Gamma_i$ are plotted in black.
Fig.~\ref{S1inS2inBenyahia6} shows the following features.

%%%%%%%%%%%%%%%
%\begin{enumerate}
%%%%%%%%%%%%%%%%
%\item 
For $0<D<D_2$ all regions exist, see 
Fig.~\ref{S1inS2inBenyahia6}(a). 
For increasing $D$, 
the vertical line 
$\Gamma_1$ defined by 
$S_{1{\rm in}}=S_1^*(D)$ moves to the right and tends towards the vertical line defined by
$S_{1{\rm in}}=S_1^*(D_2)$. At the same time, the horizontal lines 
$\Gamma_2$ and $\Gamma_3$, defined by 
$S_{2{\rm in}}=S_2^{1*}(D)$ and 
$S_{2{\rm in}}=S_2^{2*}(D)$, respectively,
move towards each other and tend toward the horizontal line defined by 
$S_{2{\rm in}}=S_2^M$, so that the regions $\mathcal{I}_1$, $\mathcal{I}_4$ and 
$\mathcal{I}_5$ shrink an disappear, 
see Fig.~\ref{S1inS2inBenyahia6}(b). 

%%%%%%%%%%%%%%%%
%\item 
For $D=D_2$ the operating diagram changes dramatically, since regions
$\mathcal{I}_1$, 
$\mathcal{I}_2$,
$\mathcal{I}_4$,
$\mathcal{I}_5$, 
$\mathcal{I}_6$, 
$\mathcal{I}_7$ and $\mathcal{I}_8$ disappear and regions 
$\mathcal{I}_0$, $\mathcal{I}_3$ invade the whole operating plan. See
Figs.~\ref{S1inS2inBenyahia6}(b) and \ref{S1inS2inBenyahia6}(c) obtained for 
$D=0.818557<D_2$ and  
$D=0.82>D_2$ respectively, where  
$D_2\approx 0.818557467$.

%%%%%%%%%%%%%%%%%%%%%%
%\item 
For $D_2<D<D_1$ only regions $\mathcal{I}_0$ and $\mathcal{I}_3$ appear, see 
Figs.~\ref{S1inS2inBenyahia6}(c)
and \ref{S1inS2inBenyahia6}(d). 
For increasing $D$, the vertical line 
$\Gamma_1$ defined by 
$S_{1{\rm in}}=S_1^*(D)$ moves to the right and tends towards infinity.
%%%%%%%%%%%%%%%%%%%%%
%\item 
For $D\geq D_1$ only region $\mathcal{I}_0$ appears.
%\end{enumerate}

The cuts $D$ constant of the 3-dimensional operating diagram corresponding to Fig.~\ref{figHi}(b), are shown in 
Fig.~\ref{S1inS2inBenyahia5}. This figure has the same qualitative characteristics as  
Fig.~\ref{S1inS2inBenyahia6}: presence of all regions when $0<D<D_2$ as shown in 
Fig.~\ref{S1inS2inBenyahia5}(a); disappearance of all regions except 
regions $\mathcal{I}_0$ and 
$\mathcal{I}_3$, when $D=D_2$, as shown in 
the transition from 
Fig. \ref{S1inS2inBenyahia5}(b) to 
Fig.~\ref{S1inS2inBenyahia5}(c); disappearance of region $\mathcal{I}_3$, when $D\geq D_1$, as shown in \ref{S1inS2inBenyahia5}(d).

\subsection{Operating diagram when $D_1<D_2$}

%==================================
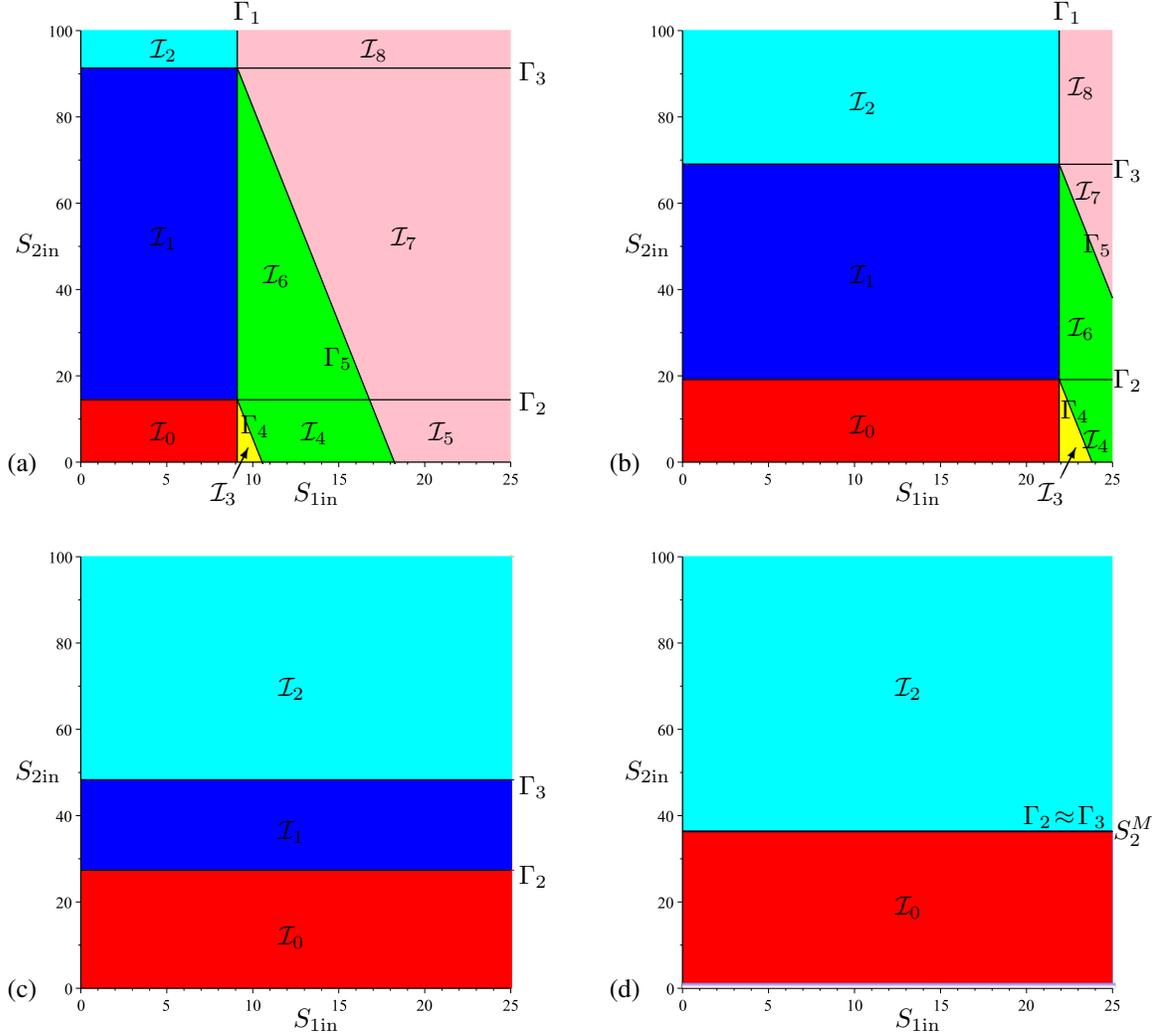
\begin{figure}[th]
\setlength{\unitlength}{1.0cm}
\begin{center}
\begin{picture}(8.5,14.6)(0,0.3)		
\put(-3,6.5){{\includegraphics[scale=0.34]{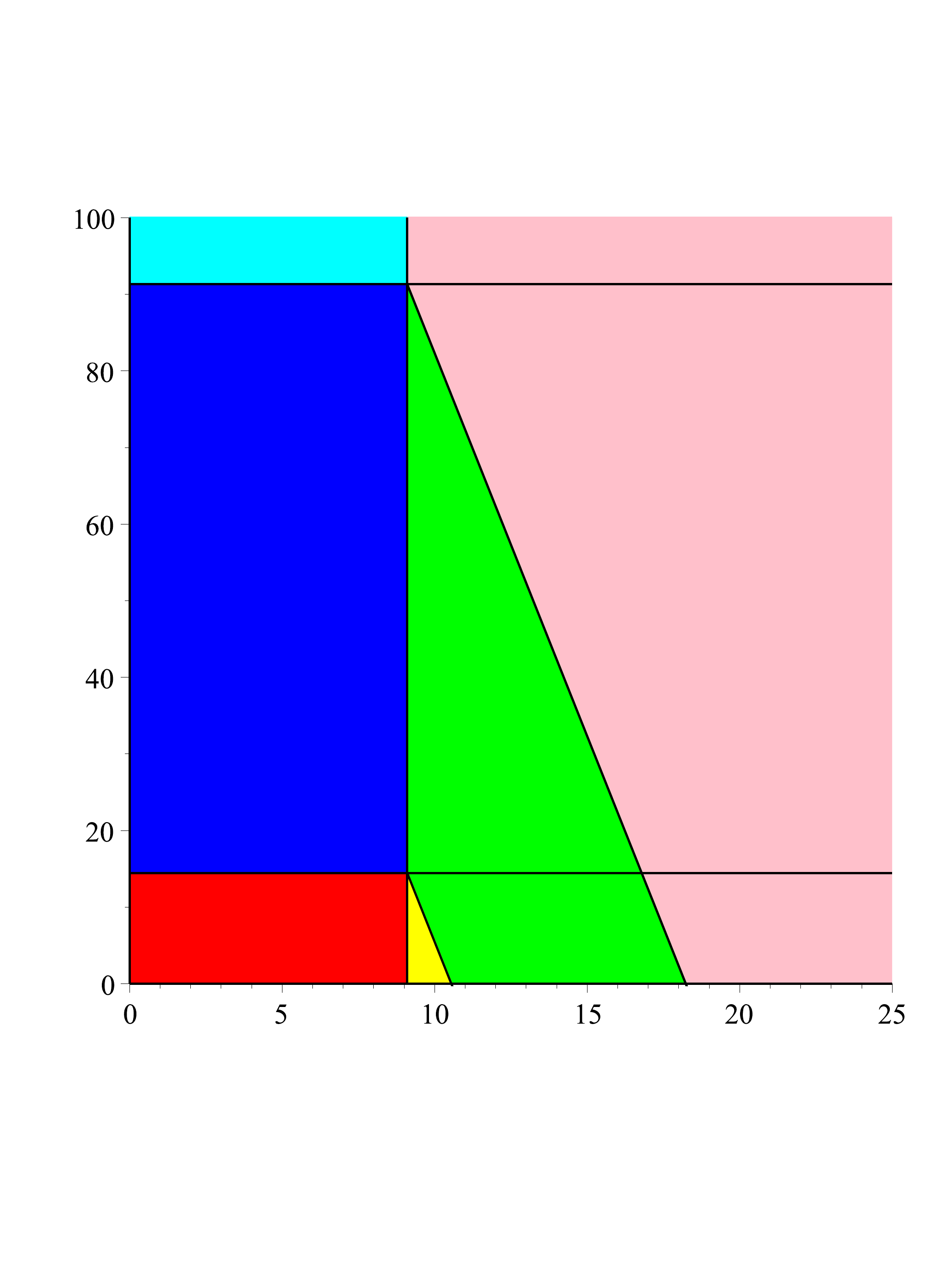}}}
\put(5,6.5){{\includegraphics[scale=0.34]{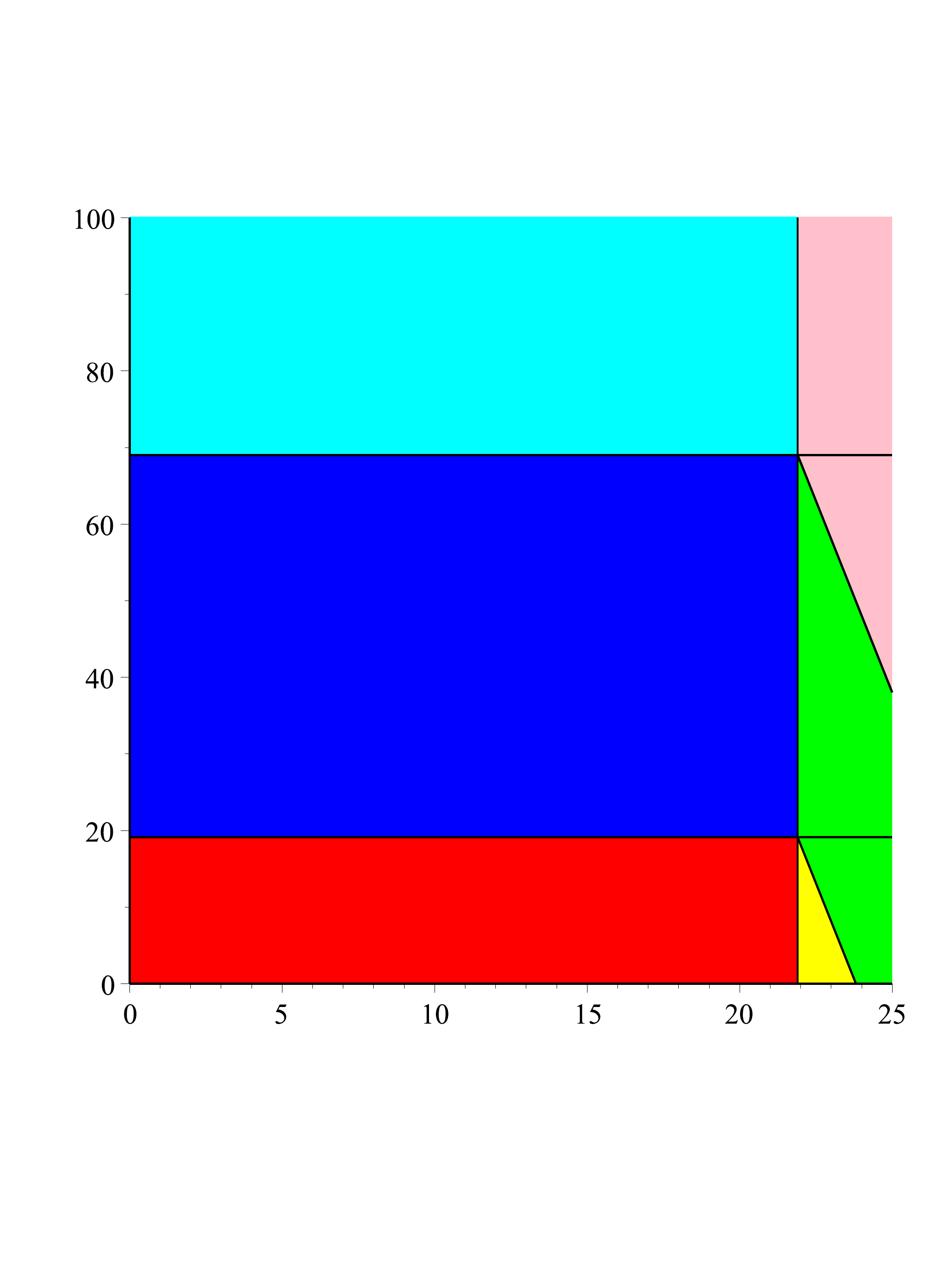}}}
\put(-3,-0.5){{\includegraphics[scale=0.34]{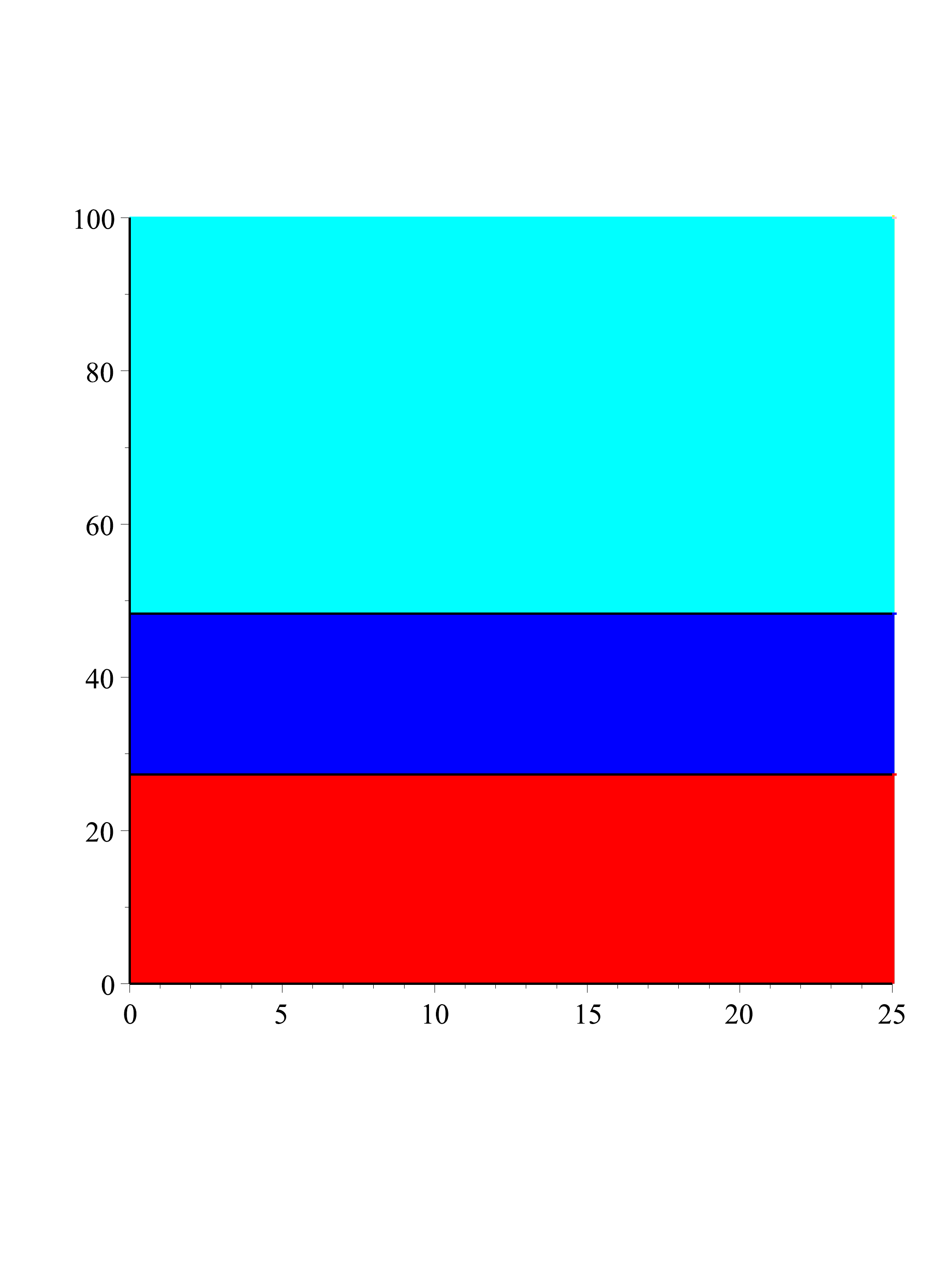}}}
\put(5,-0.5){{\includegraphics[scale=0.34]{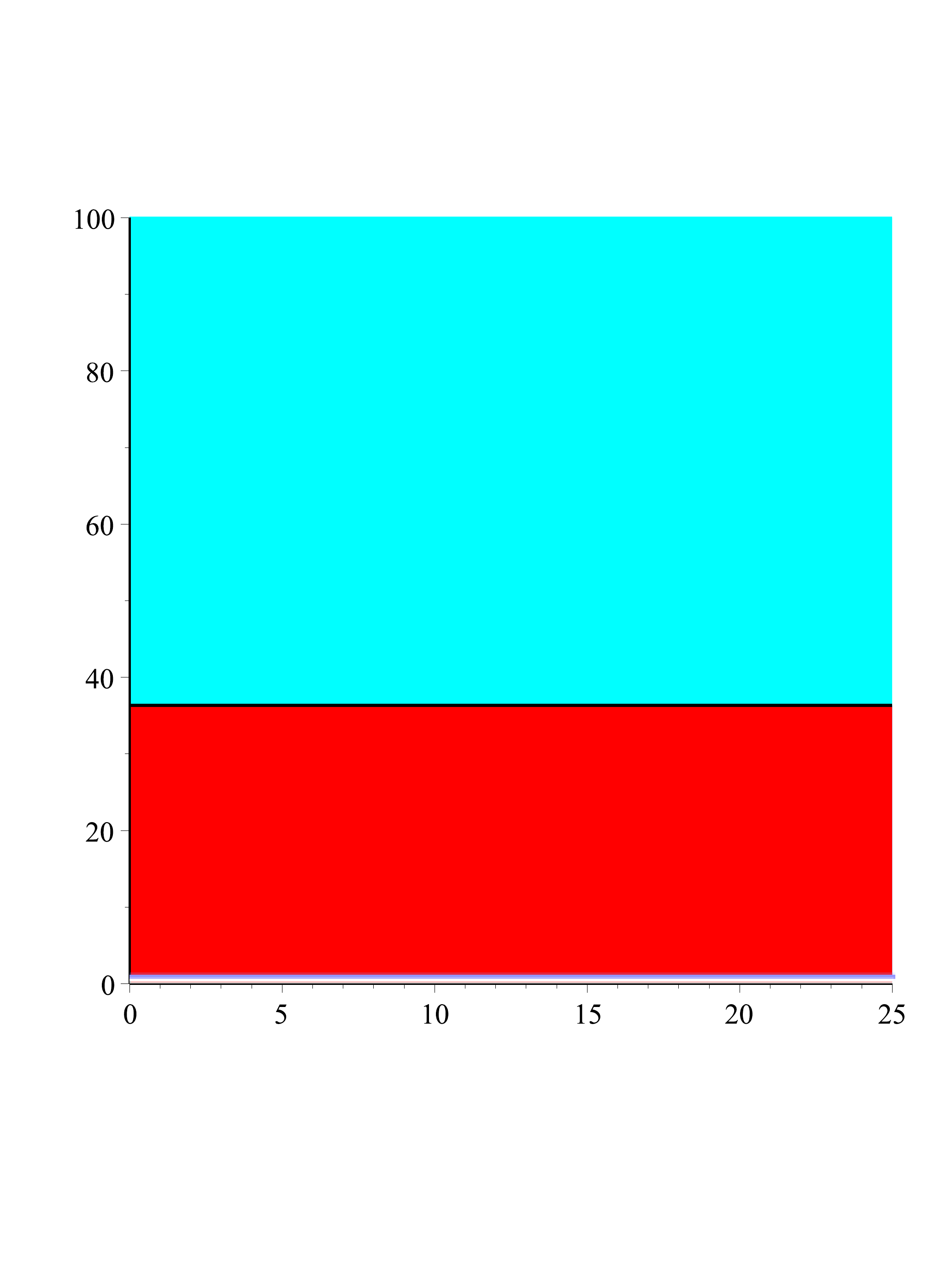}}}
%==============================================
\put(-3,8.5){(a)}
\put(5,8.5){(b)}	
\put(-3,1.5){(c)}
\put(5,1.5){(d)}		
\put(-2.9,11.4){$S_{2{\rm in}}$}
\put(5.2,11.4){$S_{2{\rm in}}$}	
\put(-2.9,4.4){$S_{2{\rm in}}$}
\put(5.2,4.4){$S_{2{\rm in}}$}	
\put(0.8,8.1){$S_{1{\rm in}}$}
\put(8.8,8.1){$S_{1{\rm in}}$}	
\put(0.8,1.1){$S_{1{\rm in}}$}
\put(8.8,1.1){$S_{1{\rm in}}$}	
\put(-1.1,8.9){$\mathcal{I}_0$}
\put(8.2,9){$\mathcal{I}_0$}
\put(0.6,2.2){$\mathcal{I}_0$}
\put(8.8,2.6){$\mathcal{I}_0$}
\put(-1.1,11.5){$\mathcal{I}_1$}
\put(8.2,11){$\mathcal{I}_1$}
\put(0.6,3.6){$\mathcal{I}_1$}
\put(-1.1,14){$\mathcal{I}_2$}
\put(8.2,13.3){$\mathcal{I}_2$}	
\put(0.6,5.5){$\mathcal{I}_2$}
\put(8.8,5.5){$\mathcal{I}_2$}
\put(-0.3,8.1){$\mathcal{I}_3$}
\put(0,8.4){\vector(1,2){0.2}}
\put(10.7,8.1){$\mathcal{I}_3$}
\put(11,8.4){\vector(1,2){0.2}}
\put(0.9,8.9){$\mathcal{I}_4$}
\put(11.3,8.8){$\mathcal{I}_4$}
\put(2.6,8.9){$\mathcal{I}_5$}
\put(0.4,11){$\mathcal{I}_6$}
\put(11.1,10.3){$\mathcal{I}_6$}
\put(2.1,11.5){$\mathcal{I}_7$}
\put(11.2,12.1){$\mathcal{I}_7$}
\put(1.7,14){$\mathcal{I}_8$}
\put(11.1,13.5){$\mathcal{I}_8$}
\put(0,14.5){$\Gamma_1$}
\put(10.9,14.5){$\Gamma_1$}
\put(3.8,9.3){$\Gamma_2$}
\put(11.7,9.6){$\Gamma_2$}
\put(3.8,3){$\Gamma_2$}
\put(3.8,13.7){$\Gamma_3$}
\put(11.7,12.4){$\Gamma_3$}
\put(3.8,4.2){$\Gamma_3$}
\put(10.5,3.8){{$\Gamma_2\!\approx\!\Gamma_3$}}
\put(11.7,3.6){$S_2^M$}
\put(0.1,9){$\Gamma_4$}
\put(11,9.2){$\Gamma_4$}
\put(1.2,9.9){$\Gamma_5$}
\put(11.3,11.4){$\Gamma_5$}
\end{picture}
\end{center}
\vspace{-0.9cm}
\caption{The 2-dimensional operating diagram 
$\left(S_{1{\rm in}},S_{2{\rm in}}\right)$ with $D$ constant, corresponding to Fig.~\ref{figHi}(c). 
(a): $D=0.65$;
(b): $D=0.73$; 
(c): $D=D1=0.8$;
(d): $D=0.818557<D_2$. Here 
$D_2\approx 0.818557467$ and 
$S_M^2\approx 36.332$.}\label{S1inS2inBenyahia4}
\end{figure}

The cuts $D$ constant of the 3-dimensional operating corresponding to 
Fig.~\ref{figHi}(c), are shown in 
Fig.~\ref{S1inS2inBenyahia4}. The regions are colored according to the colors in 
Table \ref{Table9cases}.  
Fig.~\ref{S1inS2inBenyahia4} shows the following features.
%%%%%%%%%%%%%%%

For $0<D<D_1$ all regions appear, 
see Fig.~\ref{S1inS2inBenyahia4}(a). 
For increasing $D$, 
the vertical line 
$\Gamma_1$ defined by 
$S_{1{\rm in}}=S_1^*(D)$ moves to the right and tends towards infinity. 
At the same time, the horizontal lines 
$\Gamma_2$ and $\Gamma_3$, defined by 
$S_{2{\rm in}}=S_2^{1*}(D)$ and 
$S_{2{\rm in}}=S_2^{2*}(D)$, respectively,
move towards each other, as depicted in 
Fig.~\ref{S1inS2inBenyahia4}(b), and tend towards the horizontal lines defined by 
$S_{2{\rm in}}=S_2^{1*}(D_1)$ and 
$S_{2{\rm in}}=S_2^{2*}(D_1)$, respectively, as depicted in Fig.~\ref{S1inS2inBenyahia4}(c).
 
For $D=D_1$, the operating diagram changes dramatically:  all regions 
$\mathcal{I}_3$, $\mathcal{I}_4$ and 
$\mathcal{I}_5$, $\mathcal{I}_6$, 
$\mathcal{I}_7$ and 
$\mathcal{I}_8$ have 
disappeared since they are located to the right of the vertical $\Gamma_1$ which tends toward infinity, when $D$ tends to $D_1$, 
as depicted in 
Fig.~\ref{S1inS2inBenyahia4}(c).
%%%%%%%%%%%%%%%%%%%%%%%

For $D_1\leq D<D_2$ only regions
$\mathcal{I}_0$, 
$\mathcal{I}_1$,
and $\mathcal{I}_2$ appear.
For increasing $D$, 
the horizontal lines 
$\Gamma_2$ and $\Gamma_3$, defined by 
$S_{2{\rm in}}=S_2^{1*}(D)$ and 
$S_{2{\rm in}}=S_2^{2*}(D)$, respectively,
move towards each other and tend toward the horizontal line defined by 
$S_{2{\rm in}}=S_2^M$, so that the regions $\mathcal{I}_1$ shrinks an disappear, see 
Fig.~\ref{S1inS2inBenyahia4}(d). 
%%%%%%%%%%%%%%%%

For $D=D_2$ the operating diagram changes dramatically, since regions
$\mathcal{I}_1$ and 
$\mathcal{I}_2$ disappear and region 
$\mathcal{I}_0$ invades the whole operating plan.
%%%%%%%%%%%%%%%%%%%%%
For $D\geq D_2$ only region $\mathcal{I}_0$ appears.

%%%%%%%%%%%%%%%%%%%%%%%%%%
%%%%%%%%%%%%%%%%%%%%%
\begin{figure}[th]
\setlength{\unitlength}{1.0cm}
\begin{center}
\begin{picture}(8.5,14.7)(0,0.2)		
\put(-3,6.5){{\includegraphics[scale=0.34]{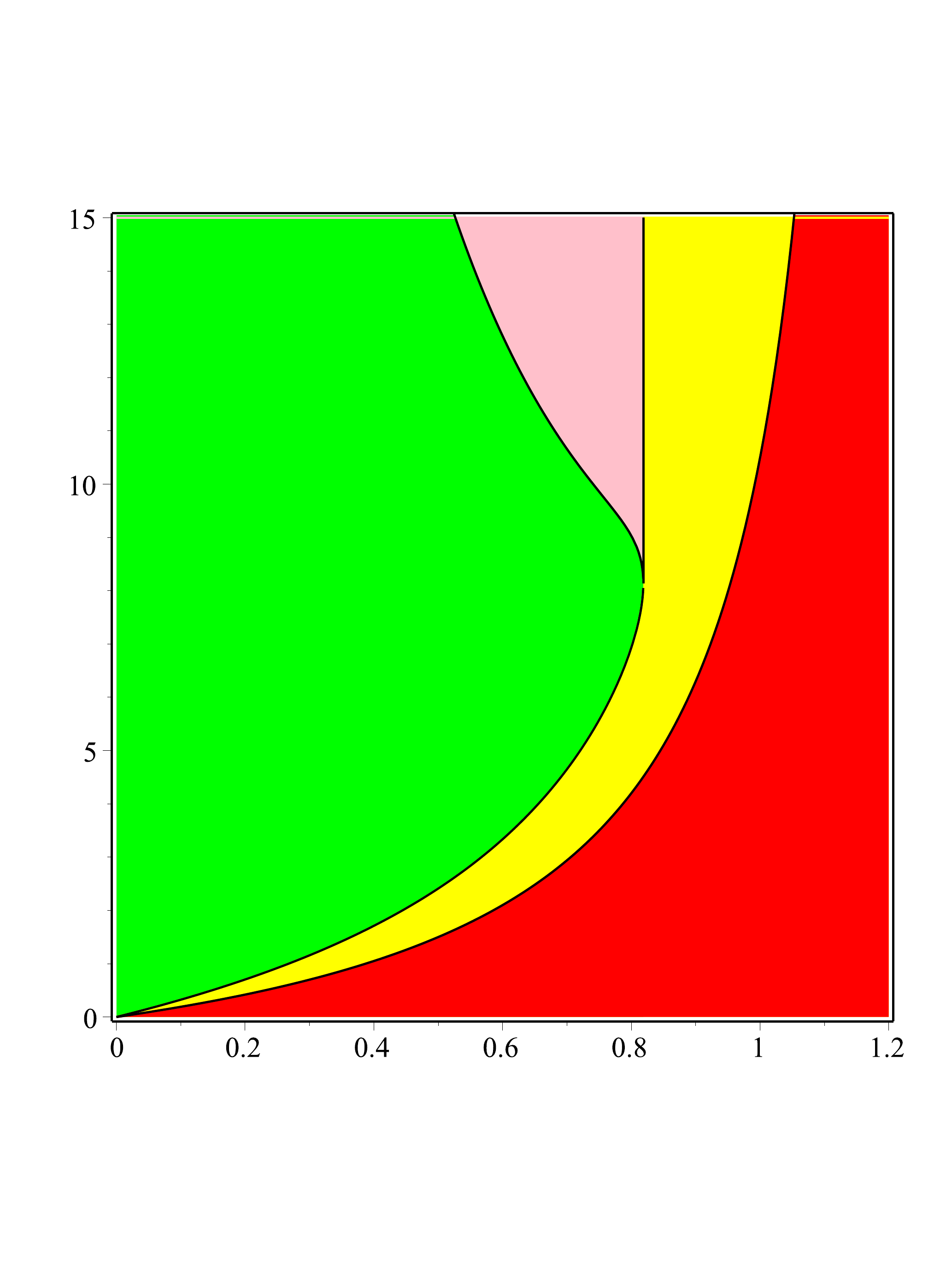}}}
\put(5,6.5){{\includegraphics[scale=0.34]{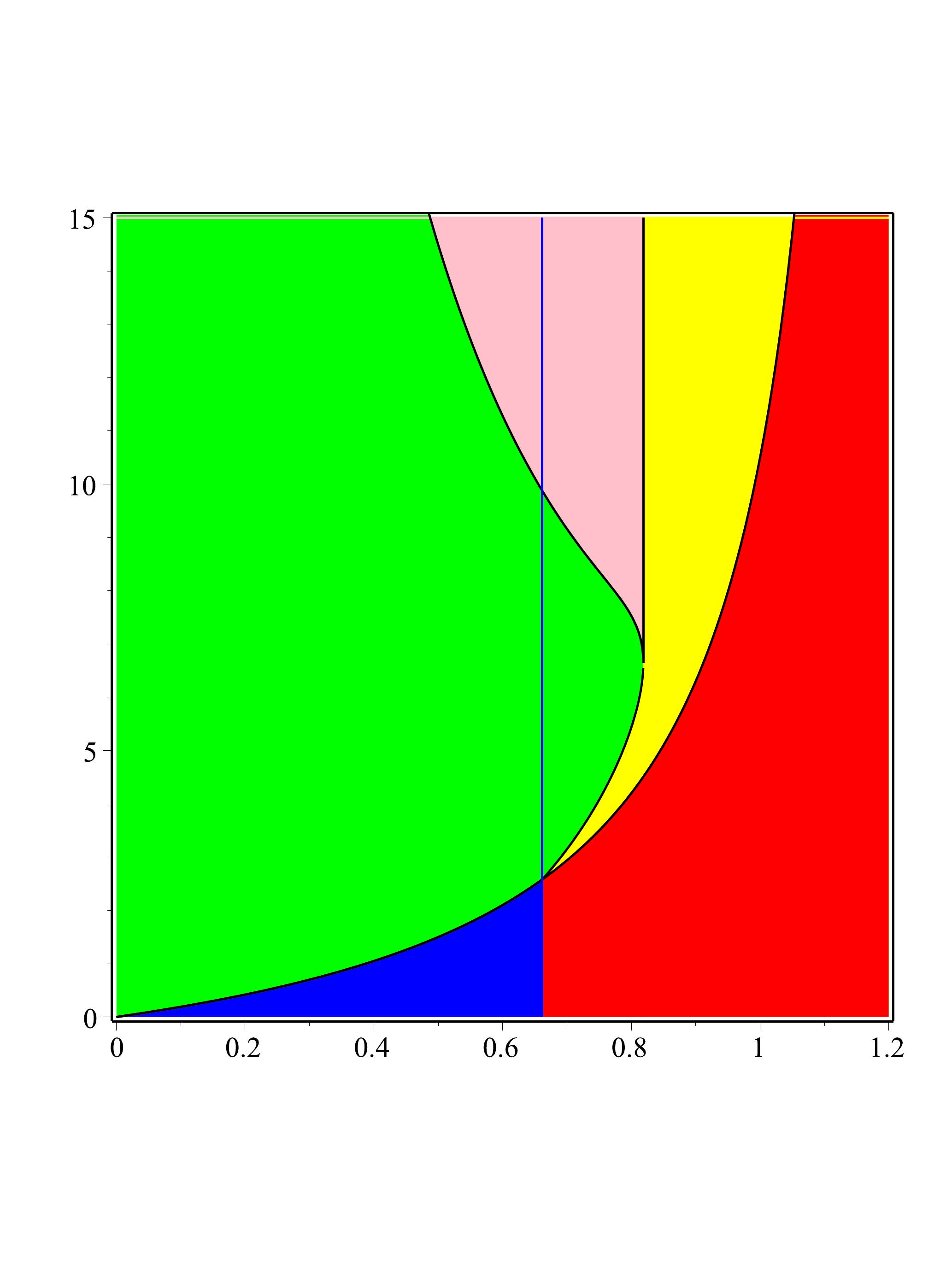}}}
\put(-3,-0.5){{\includegraphics[scale=0.34]{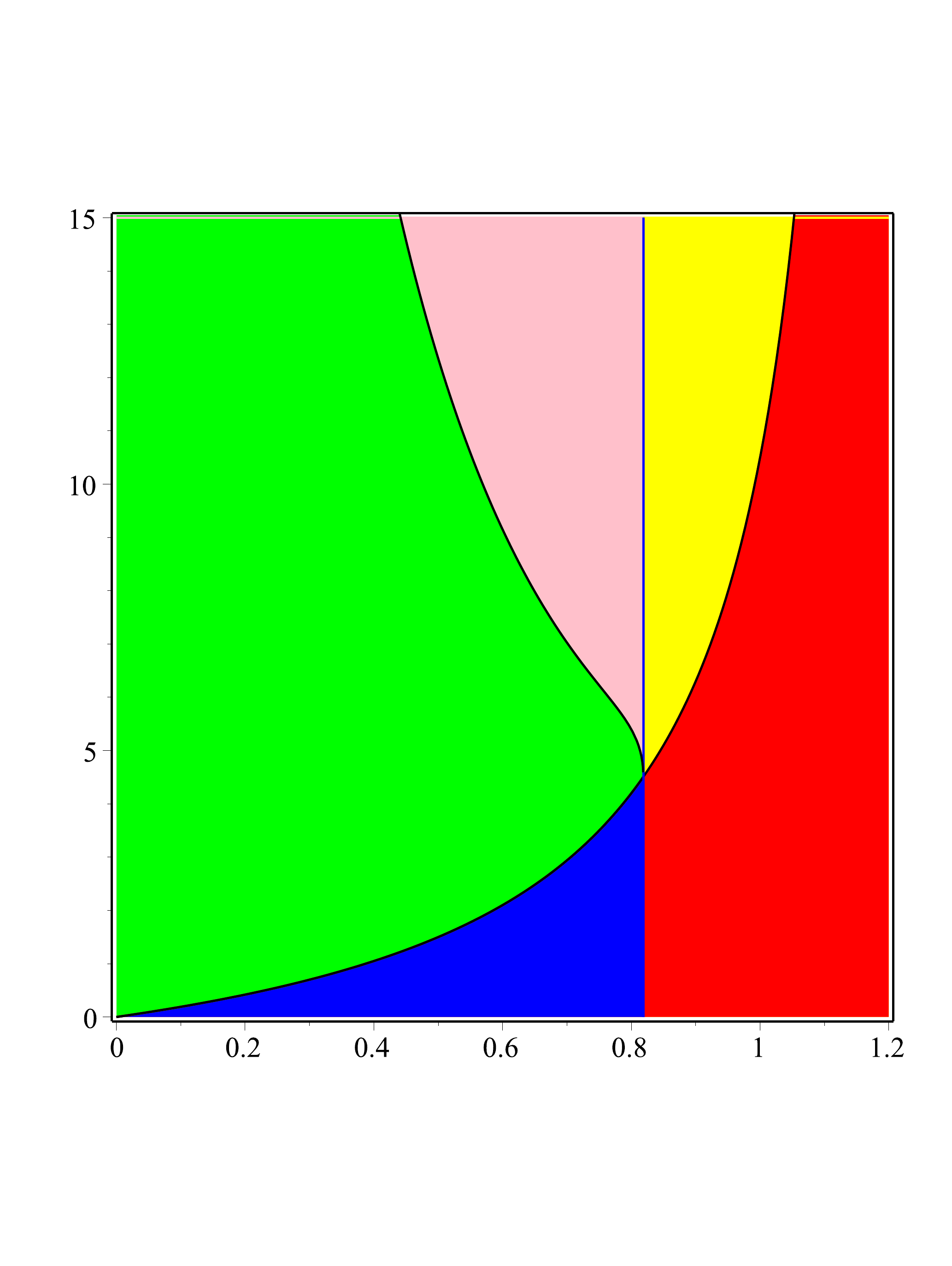}}}
\put(5,-0.5){{\includegraphics[scale=0.34]{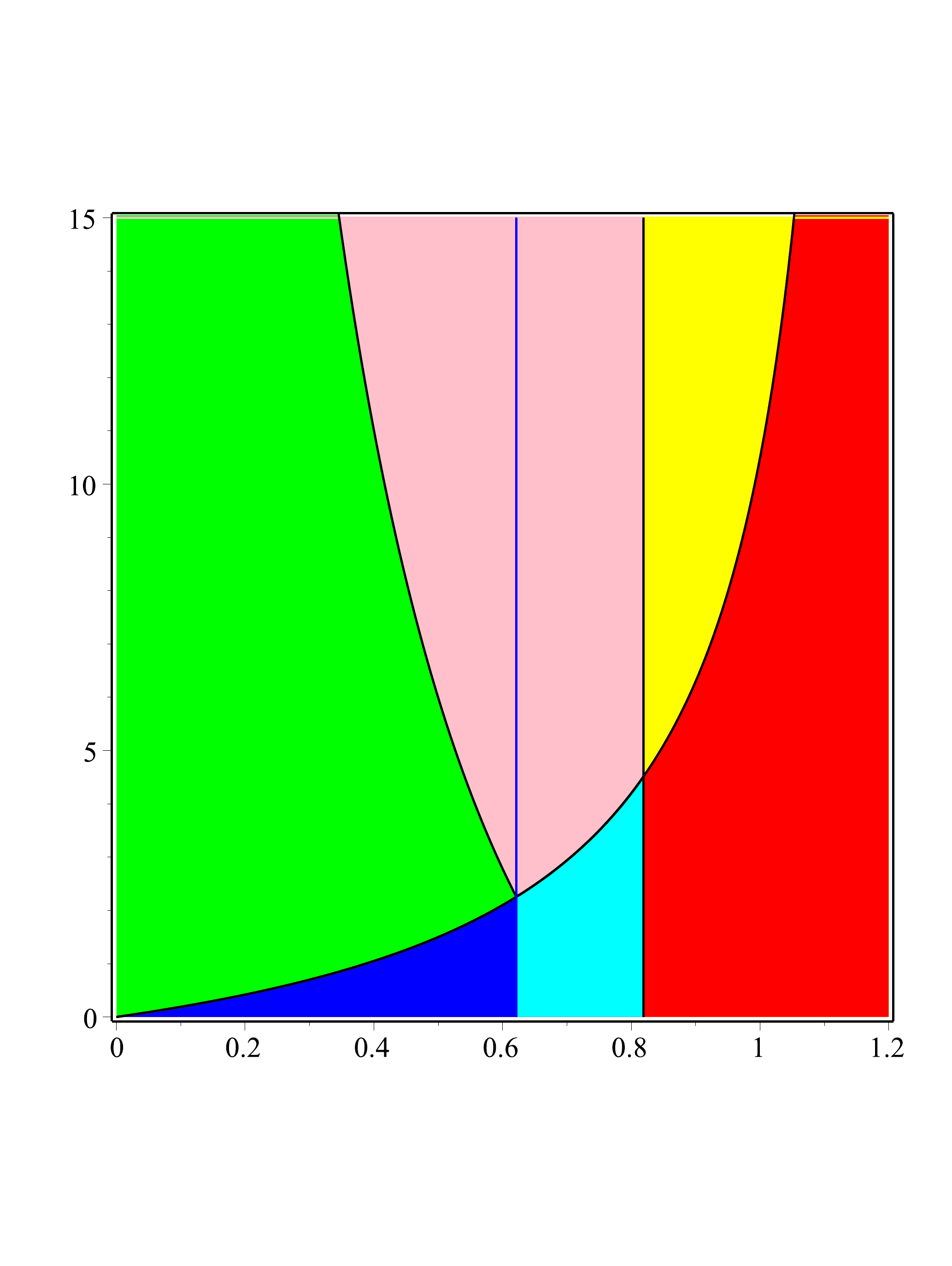}}}
%==============================================
\put(-3,8.5){(a)}
\put(5,8.5){(b)}	
\put(-3,1.5){(c)}
\put(5,1.5){(d)}		
\put(-2.9,11.4){$S_{1{\rm in}}$}
\put(5.1,11.4){$S_{1{\rm in}}$}	
\put(-2.9,4.4){$S_{1{\rm in}}$}
\put(5.1,4.4){$S_{1{\rm in}}$}	
\put(0.9,7.9){$D$}
\put(8.9,7.9){$D$}	
\put(0.9,0.9){$D$}
\put(8.9,0.9){$D$}
\put(2.7,10){$\mathcal{I}_0$}
\put(10.7,10){$\mathcal{I}_0$}
\put(2.7,3){$\mathcal{I}_0$}
\put(10.7,3){$\mathcal{I}_0$}
\put(8,8.5){$\mathcal{I}_1$}
\put(0.7,1.6){$\mathcal{I}_1$}
\put(8.2,1.6){$\mathcal{I}_1$}
\put(9.2,1.6){$\mathcal{I}_2$}	
\put(2.1,13.5){$\mathcal{I}_3$}
\put(10.1,13.5){$\mathcal{I}_3$}
\put(2.1,6){$\mathcal{I}_3$}
\put(10.1,6){$\mathcal{I}_3$}
\put(0,10.7){$\mathcal{I}_4$}
\put(9.2,10.7){$\mathcal{I}_4$}
\put(1.2,13.5){$\mathcal{I}_5$}
\put(9.2,13.5){$\mathcal{I}_5$}
\put(7,10.7){$\mathcal{I}_6$}
\put(-0.5,3){$\mathcal{I}_6$}
\put(7,3){$\mathcal{I}_6$}
\put(8.6,13.5){$\mathcal{I}_7$}
\put(1,6){$\mathcal{I}_7$}
\put(8.2,6){$\mathcal{I}_7$}
\put(9.1,6){$\mathcal{I}_8$}
\put(2.7,14.5){$\Gamma_1$}
\put(10.7,14.5){$\Gamma_1$}
\put(2.7,7.5){$\Gamma_1$}
\put(10.7,7.5){$\Gamma_1$}
\put(8.9,14.5){$\Gamma_2$}
\put(8.7,7.5){$\Gamma_3$}
\put(0.55,9.6){$\Gamma_4$}
\put(9.3,10){$\Gamma_4$}
\put(0.1,14.5){$\Gamma_5$}
\put(8,14.5){$\Gamma_5$}
\put(-0.3,7.5){$\Gamma_5$}
\put(7.4,7.5){$\Gamma_5$}
\put(1.5,14.5){$\Gamma_6$}
\put(9.6,14.5){$\Gamma_6$}
\put(1.5,7.5){$\Gamma_6$}
\put(9.6,7.5){$\Gamma_6$}
\end{picture}
\end{center}
\vspace{-0.9cm}
\caption{The 2-dimensional operating diagram 
$\left(D,S_{1{\rm in}}\right)$ obtained by cuts at $S_{2{\rm in}}$ constant of the 3-dimensional operating diagram shown in Fig.~\ref{figOD3D} and corresponding to  Fig.~\ref{figHi}(a). (a):
$S_{2{\rm in}}=0$, 
(b): $S_{2{\rm in}}=15$, 
(c): $S_{2{\rm in}}=S_2^M\simeq 36.332$ and (d): $S_{2{\rm in}}=100$.}\label{DS1inBenyahia6}
\end{figure}
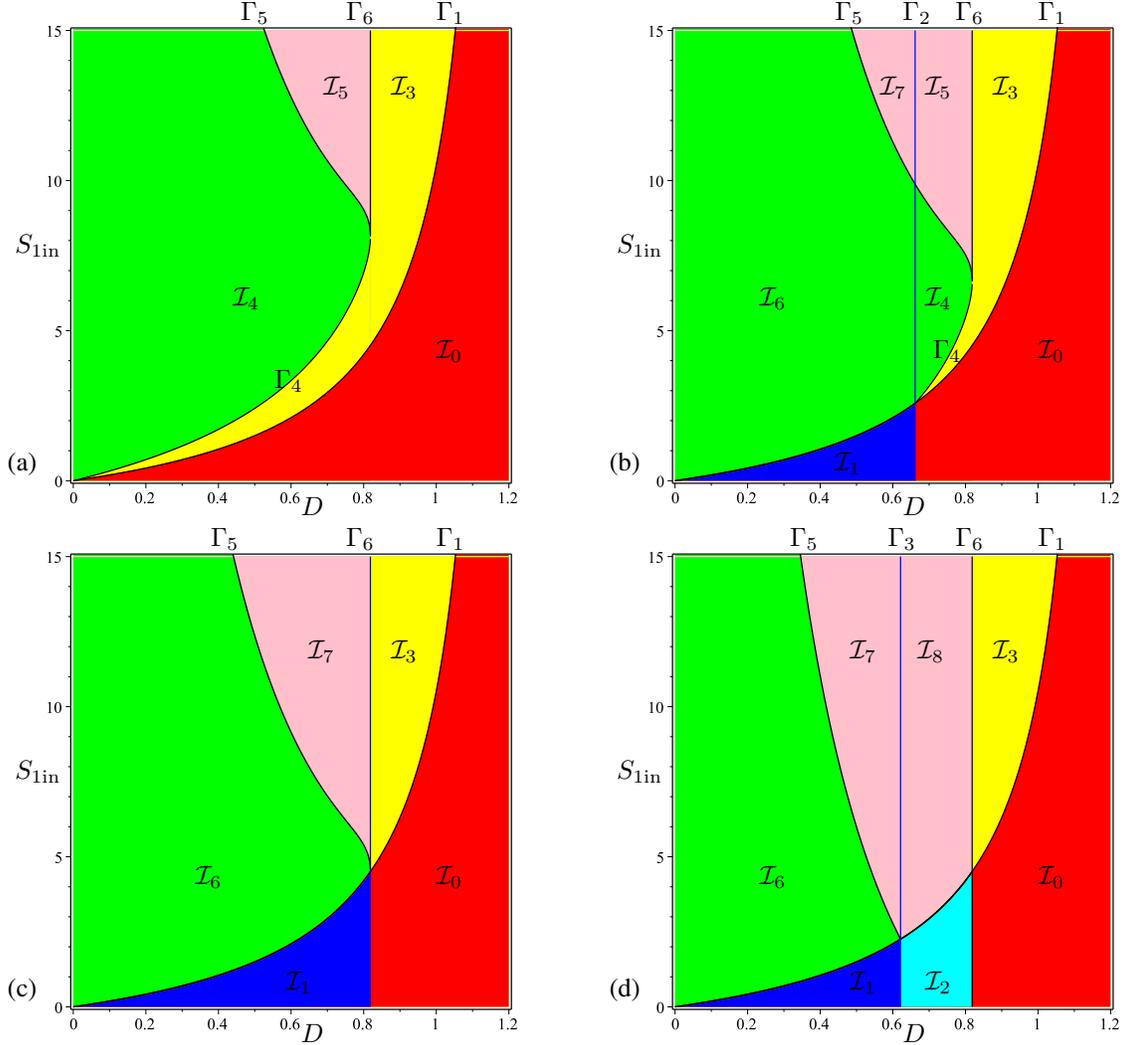
%%%%%%%%%%%%%%%%%%%%%
%%%%%%%%%%%%%%%%%%%%%
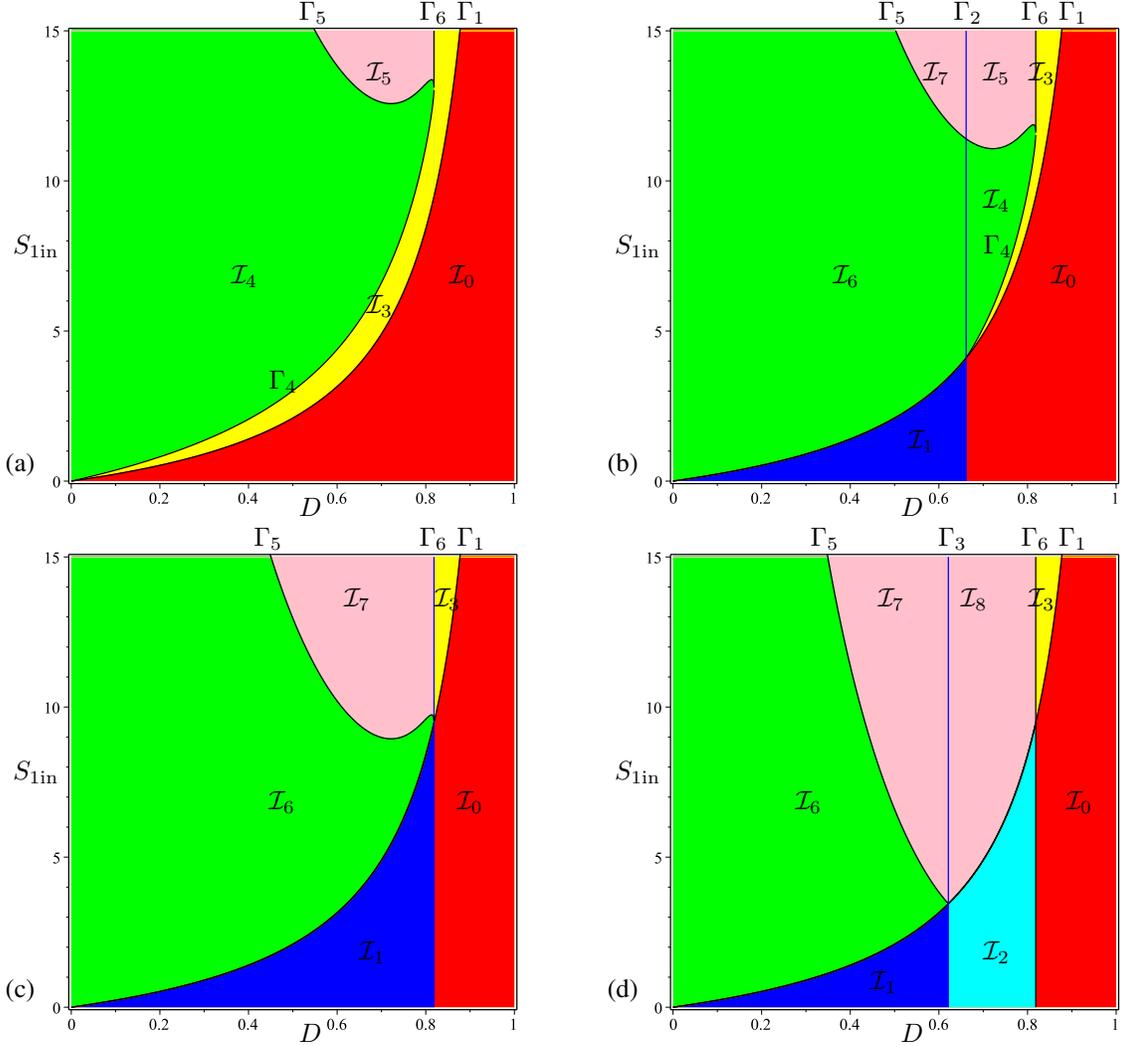
\begin{figure}[th]
\setlength{\unitlength}{1.0cm}
\begin{center}
\begin{picture}(8.5,14.7)(0,0.2)		
\put(-3,6.5){{\includegraphics[scale=0.34]{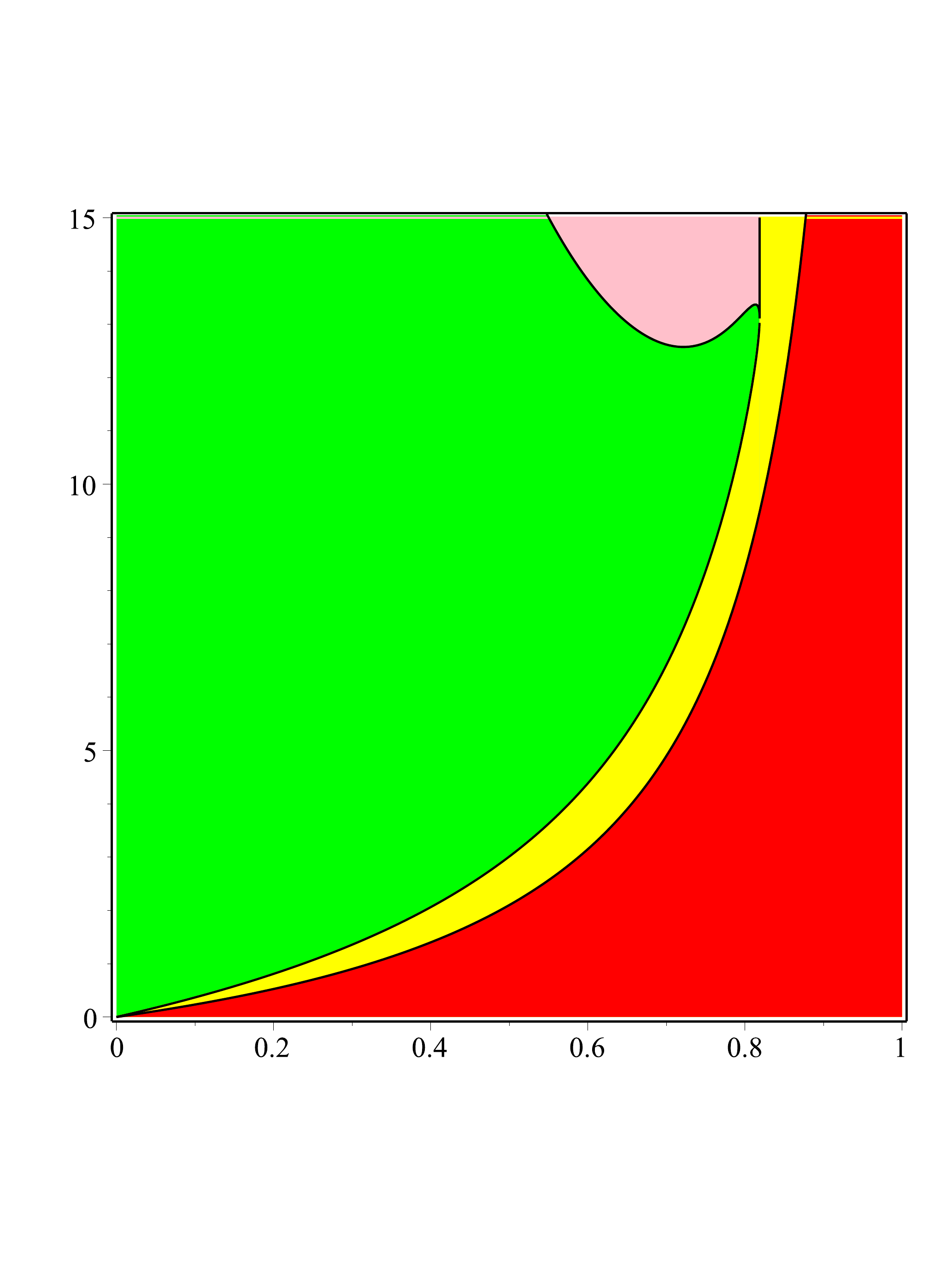}}}
\put(5,6.5){{\includegraphics[scale=0.34]{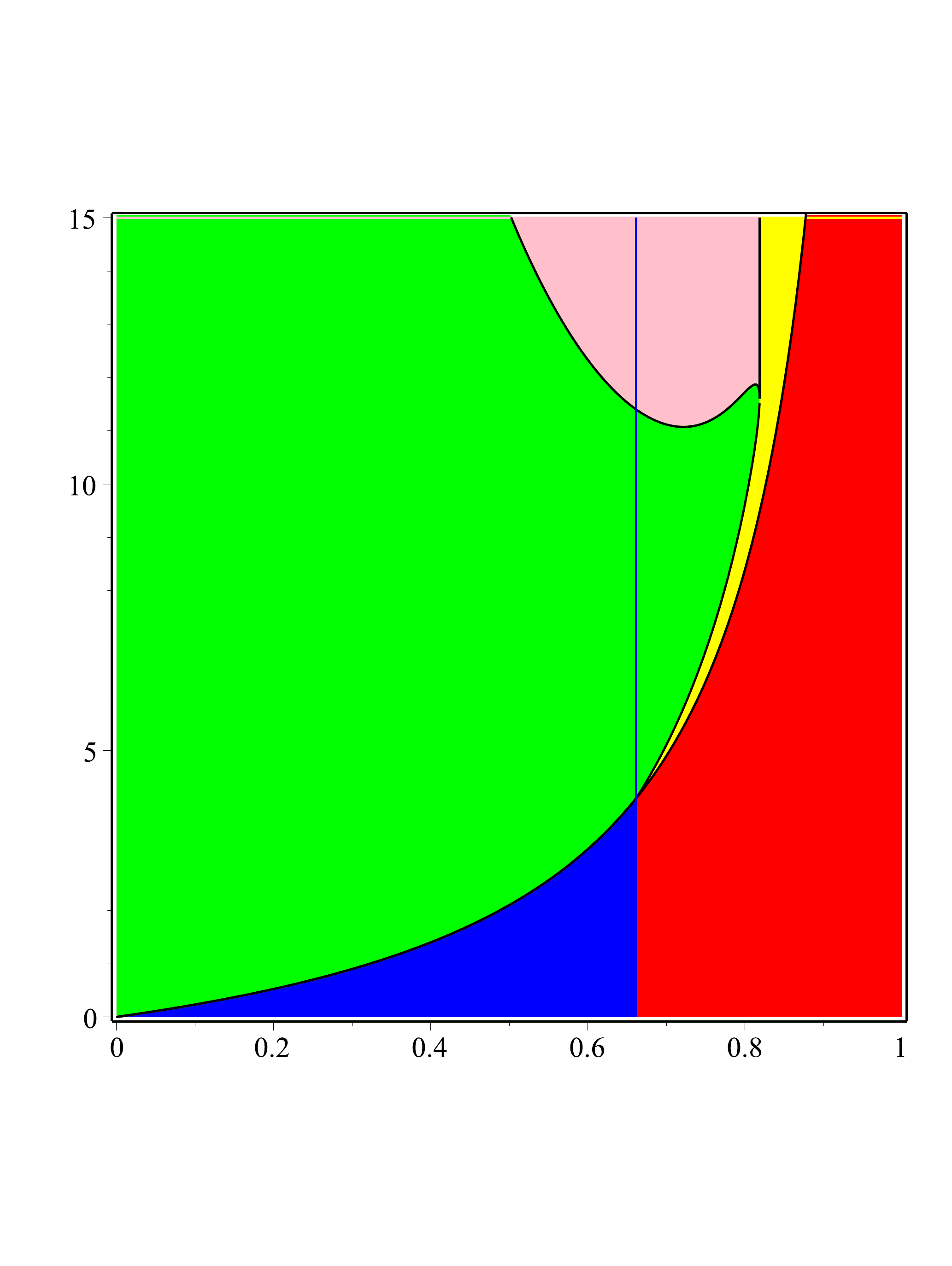}}}
\put(-3,-0.5){{\includegraphics[scale=0.34]{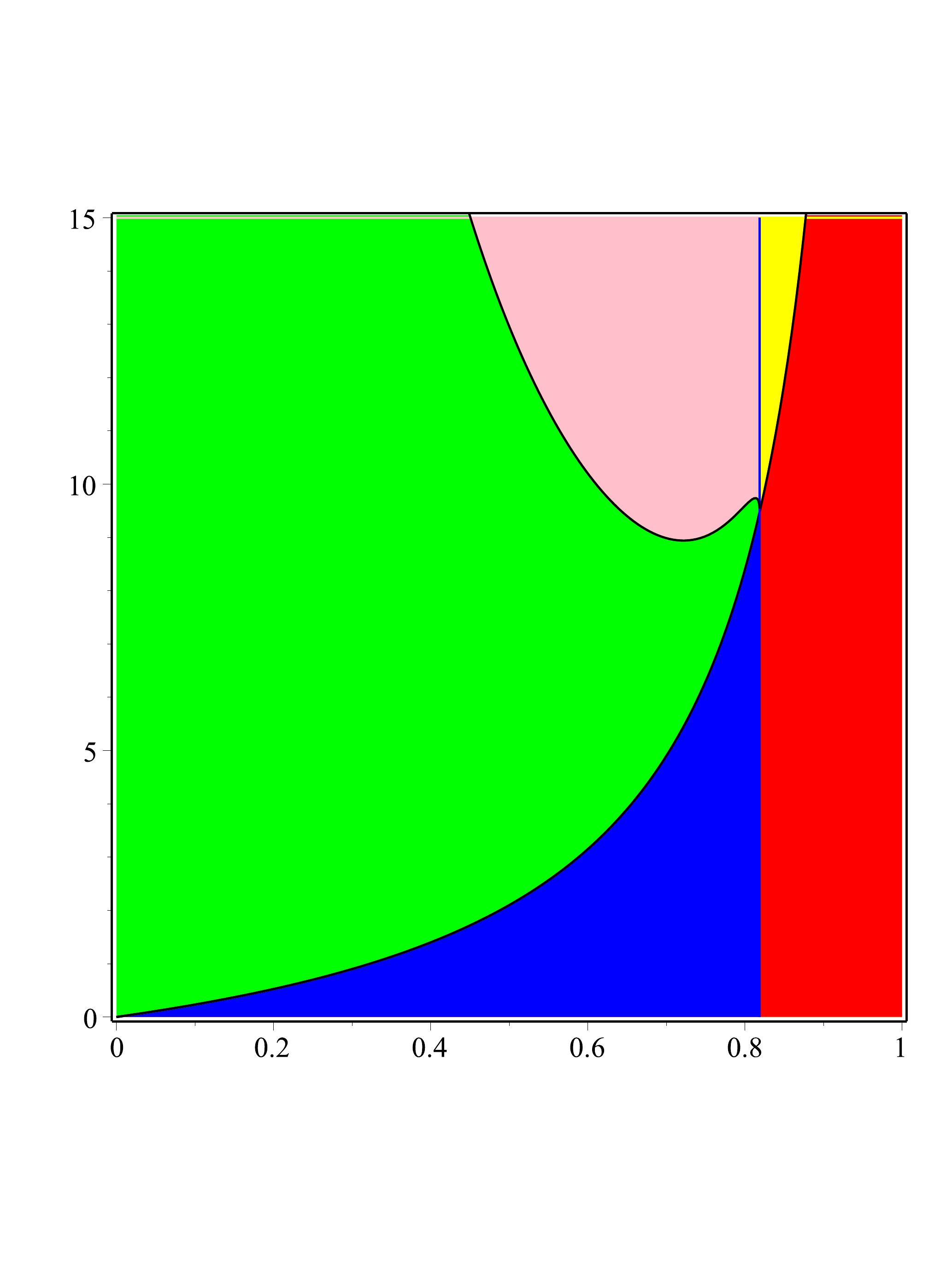}}}
\put(5,-0.5){{\includegraphics[scale=0.34]{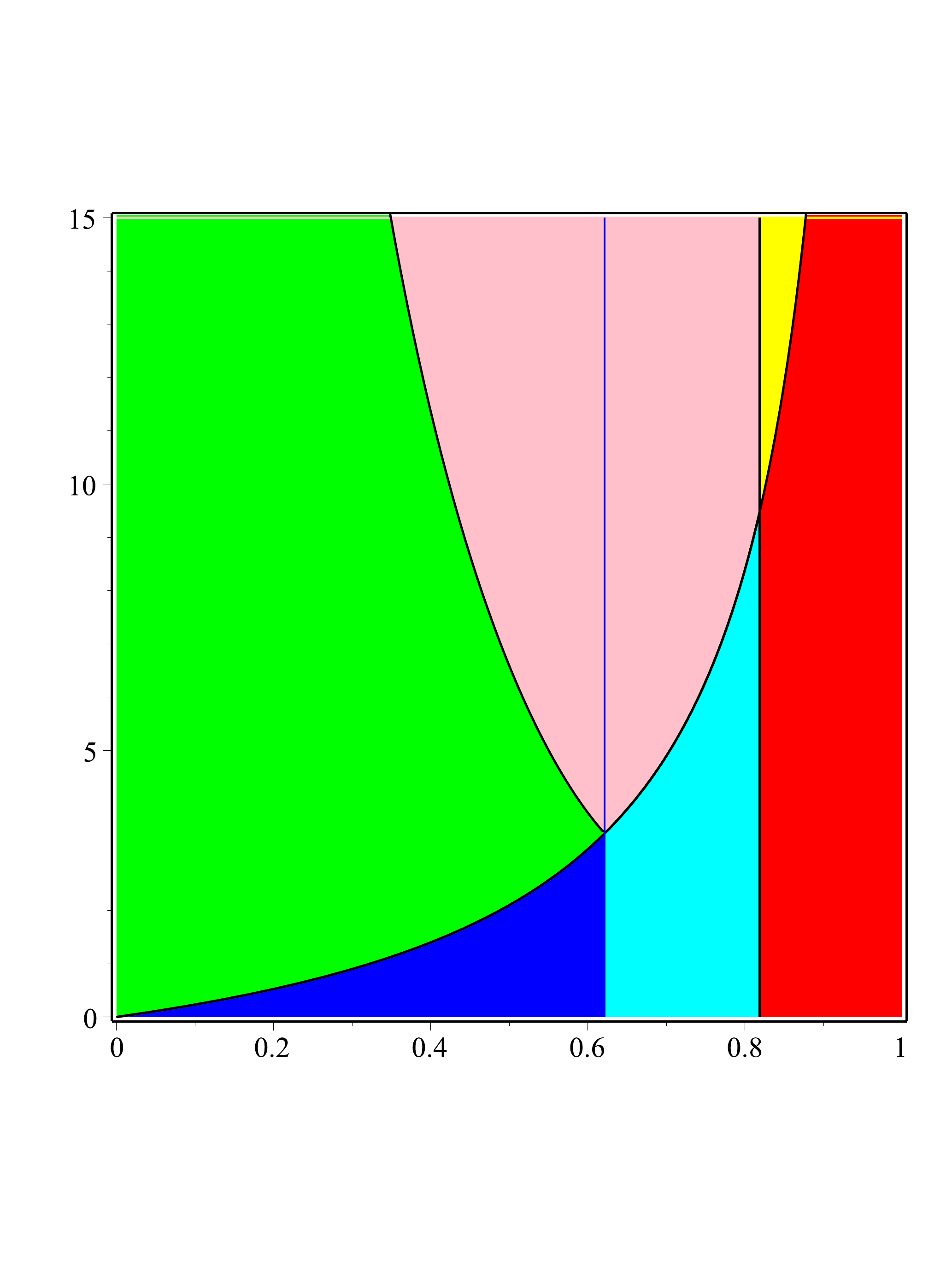}}}
%==============================================
\put(-3,8.5){(a)}
\put(5,8.5){(b)}	
\put(-3,1.5){(c)}
\put(5,1.5){(d)}		
\put(-2.9,11.4){$S_{1{\rm in}}$}
\put(5.1,11.4){$S_{1{\rm in}}$}	
\put(-2.9,4.4){$S_{1{\rm in}}$}
\put(5.1,4.4){$S_{1{\rm in}}$}	
\put(0.9,7.9){$D$}
\put(8.9,7.9){$D$}	
\put(0.9,0.9){$D$}
\put(8.9,0.9){$D$}
\put(2.9,11){$\mathcal{I}_0$}
\put(10.9,11){$\mathcal{I}_0$}
\put(3,4){$\mathcal{I}_0$}
\put(11.1,4){$\mathcal{I}_0$}
\put(9,8.8){$\mathcal{I}_1$}
\put(1.7,2){$\mathcal{I}_1$}
\put(8.5,1.6){$\mathcal{I}_1$}
\put(10,2){$\mathcal{I}_2$}	
\put(1.8,10.6){$\mathcal{I}_3$}
\put(10.6,13.7){$\mathcal{I}_3$}
\put(2.7,6.7){$\mathcal{I}_3$}
\put(10.6,6.7){$\mathcal{I}_3$}
\put(0,11){$\mathcal{I}_4$}
\put(10,12){$\mathcal{I}_4$}
\put(1.8,13.7){$\mathcal{I}_5$}
\put(10,13.7){$\mathcal{I}_5$}
\put(8,11){$\mathcal{I}_6$}
\put(0.5,4){$\mathcal{I}_6$}
\put(7.5,4){$\mathcal{I}_6$}
\put(9.2,13.7){$\mathcal{I}_7$}
\put(1.5,6.7){$\mathcal{I}_7$}
\put(8.6,6.7){$\mathcal{I}_7$}
\put(9.7,6.7){$\mathcal{I}_8$}
\put(3,14.5){$\Gamma_1$}
\put(11,14.5){$\Gamma_1$}
\put(3,7.5){$\Gamma_1$}
\put(11,7.5){$\Gamma_1$}
\put(9.6,14.5){$\Gamma_2$}
\put(9.4,7.5){$\Gamma_3$}
\put(0.5,9.6){$\Gamma_4$}
\put(10,11.4){$\Gamma_4$}
\put(0.9,14.5){$\Gamma_5$}
\put(8.6,14.5){$\Gamma_5$}
\put(0.3,7.5){$\Gamma_5$}
\put(7.7,7.5){$\Gamma_5$}
\put(2.5,14.5){$\Gamma_6$}
\put(10.5,14.5){$\Gamma_6$}
\put(2.5,7.5){$\Gamma_6$}
\put(10.5,7.5){$\Gamma_6$}
\end{picture}
\end{center}
\vspace{-0.9cm}
\caption{The 2-dimensional operating diagram 
$\left(D,S_{1{\rm in}}\right)$ obtained by cuts at $S_{2{\rm in}}$ constant of the 3-dimensional operating diagram corresponding to Fig.~\ref{figHi}(b). (a):
$S_{2{\rm in}}=0$, 
(b): $S_{2{\rm in}}=15$, 
(c): $S_{2{\rm in}}=S_2^M\approx 36.332$ and (d): $S_{2{\rm in}}=100$.}\label{DS1inBenyahia5}
\end{figure}

%=================================
\section{Operating diagram in $\left(D,S_{1{\rm in}}\right)$ where $S_{2{\rm in}}$ is kept constant}
\label{ODS2inconstant}

The intersections of  
$\Gamma_2$ and $\Gamma_3$ and $\Gamma_6$ surfaces with a plane where $S_{2{\rm in}}$ is kept constant are vertical lines, and the intersections of  
$\Gamma_1$, $\Gamma_4$ and $\Gamma_5$ surface with this plane are curves of funtions of $D$,  as shown in Table \ref{IntersectionS2in}. 
Curves 
$\Gamma_1$ and $\Gamma_6$ do not depend 
on $S_{2{\rm in}}$ while curves
$\Gamma_2$, $\Gamma_3$, $\Gamma_4$ and $\Gamma_5$ depend on $S_{2{\rm in}}$. 
Note that curves $\Gamma_4$ and $\Gamma_5$ 
simply consist of translating downwards the  $H_1$ and $H_2$ function curves, shown 
in Fig.~\ref{figHi}, and multiplying by $k_1/k_2$. 
The curves $\Gamma_k$, $k=1\cdots 6$, separate the operating parameter plane
$\left(D,S_{1{\rm in}}\right)$ in up to nine regions
$\mathcal{I}_k$, $k=0\cdots8$. We begin by considering the case where $D_2<D_1$ corresponding to  
Figs.~\ref{figHi}(a) and \ref{figHi}(b).

\subsection{Operating diagram when $D_2<D_1$}

The cuts at $S_{2{\rm in}}$ constant of the 3-dimensional operating diagram shown in 
Fig.~\ref{figOD3D} and corresponding to Fig.~\ref{figHi}(a), are shown in 
Fig.~\ref{DS1inBenyahia6}. The regions are colored according to the colors in 
Table~\ref{Table9cases}.  
Fig.~\ref{DS1inBenyahia6} shows the following features.

For 
$S_{2{\rm in}}=0$, only the regions
$\mathcal{I}_0$, $\mathcal{I}_3$, 
$\mathcal{I}_4$ and $\mathcal{I}_5$ exist, 
see Fig.~\ref{DS1inBenyahia6}(a). 
For $0<S_{2{\rm in}}<S_2^M$, $\Gamma_2$ curve appears, giving birth to $\mathcal{I}_1$, 
$\mathcal{I}_6$
and 
$\mathcal{I}_7$ regions, see 
Fig.~\ref{DS1inBenyahia6}(b). 
For increasing $S_{2{\rm in}}$,
$\Gamma_4$ and $\Gamma_5$ curves are translated downwards, while 
the vertical line 
$\Gamma_2$ moves to the right and tends towards the vertical line $\Gamma_6$, as
$S_{2{\rm in}}$ tends to $S_2^M$. 

For 
$S_{2{\rm in}}=S_2^M$, $\Gamma_4$ curve disappears, while 
$\Gamma_2$ becomes equal to $\Gamma_6$, so that $\mathcal{I}_4$ and 
$\mathcal{I}_5$ regions have disappeared, see Fig.~\ref{DS1inBenyahia6}(c). 
For $S_{2{\rm in}}>S_2^M$, $\Gamma_3$ curve appears, giving birth to $\mathcal{I}_2$
and $\mathcal{I}_8$ regions, see 
Fig.~\ref{DS1inBenyahia6}(d). 
For increasing $S_{2{\rm in}}$, 
the vertical line $\Gamma_3$ moves to the left, while $\Gamma_5$ curve is translated downwards.

The cuts $S_{2{\rm in}}$ constant of the 3-dimensional operating diagram corresponding to  
Fig.~\ref{figHi}(b), are shown in 
Fig.~\ref{DS1inBenyahia5}. This figure has the same qualitative characteristics as  
Fig.~\ref{DS1inBenyahia6}: presence of only
$\mathcal{I}_0$, $\mathcal{I}_3$, 
$\mathcal{I}_4$ and $\mathcal{I}_5$ regions when $S_{2{\rm in}}=0$, see 
Fig.~\ref{DS1inBenyahia5}(a); appearance of
$\mathcal{I}_1$, $\mathcal{I}_6$ and 
$\mathcal{I}_7$ regions when 
$0<S_{2{\rm in}}<S_2^M$, see 
Fig.~\ref{DS1inBenyahia5}(b); disappearance of $\mathcal{I}_4$ and $\mathcal{I}_5$ regions when $S_{2{\rm in}}= S_2^M$, see 
Fig.~\ref{DS1inBenyahia5}(c);
appearance of
$\mathcal{I}_2$ and $\mathcal{I}_8$ regions when $S_{2{\rm in}}>S_2^M$, see 
Fig.~\ref{DS1inBenyahia5}(d). 

It should be noticed that 
in Fig.~\ref{DS1inBenyahia5}, 
the region of global asymptotic stability of the positive steady state $E_2^1$ (the Green region 
$\mathcal{I}_4\cup\mathcal{I}_6$)  presents the very surprising property that there exists a range of values for the operating parameters $S_{1{\rm in}}$ and 
$S_{2{\rm in}}$ such that the system can go from the bistability region 
(the Pink region 
$\mathcal{I}_5\cup\mathcal{I}_7$), 
to the global asymptotic stability region, when the dilution rate $D$ increases. Indeed, the boundary $\Gamma_5$ of Green and Pink regions has an increasing part, with respect to parameter $D$. Therefore, near this part of $\Gamma_5$, as  
$S_{1{\rm in}}$ is kept constant and $D$ increases the system goes from  
$\mathcal{I}_5$ to $\mathcal{I}_4$, see 
Fig.~\ref{DS1inBenyahia5}(a) and \ref{DS1inBenyahia5}(b), or goes from 
$\mathcal{I}_7$ to $\mathcal{I}_6$, see 
Fig.~\ref{DS1inBenyahia5}(c). 

This possibility of globally stabilizing the system, which presents bistability, is surprising since the global stability of the positive steady state is more likely obtained by decreasing $D$ rather than increasing it. This unespecated behavior was first observed in a slightly different two-step model, where the first kinetics is of Contois type \cite{Hanaki}. This behavior is investigated in \cite{HDR}.

It is worth-noting that this unexpected  behavior can occur only for suitable values of the biological parameters. For instance, in Fig.~\ref{DS1inBenyahia6}, where all biological parameters are the same as in 
Fig.~\ref{DS1inBenyahia5}, excepted that $m_1$ is changed from $m_1=0.5$ to $m_1=0.6$, the behavior does not occur and a transition from Pink region to Green region is possible only by decreasing $D$.

%%%%%%%%%%%%%%%%%%%%%
\begin{figure}[th]
\setlength{\unitlength}{1.0cm}
\begin{center}
\begin{picture}(8.5,14.7)(0,0.2)		
\put(-3,6.5){{\includegraphics[scale=0.34]{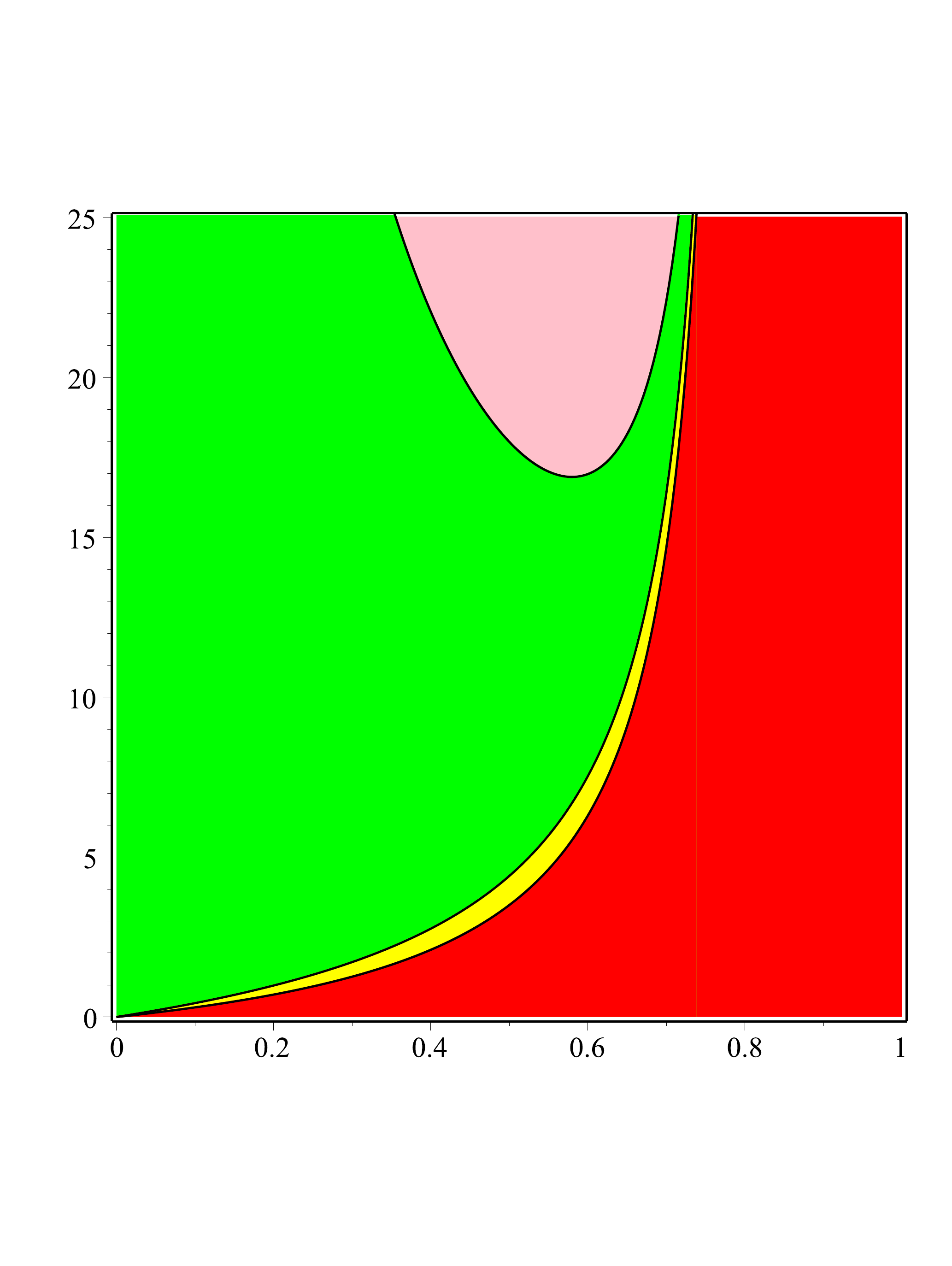}}}
\put(5,6.5){{\includegraphics[scale=0.34]{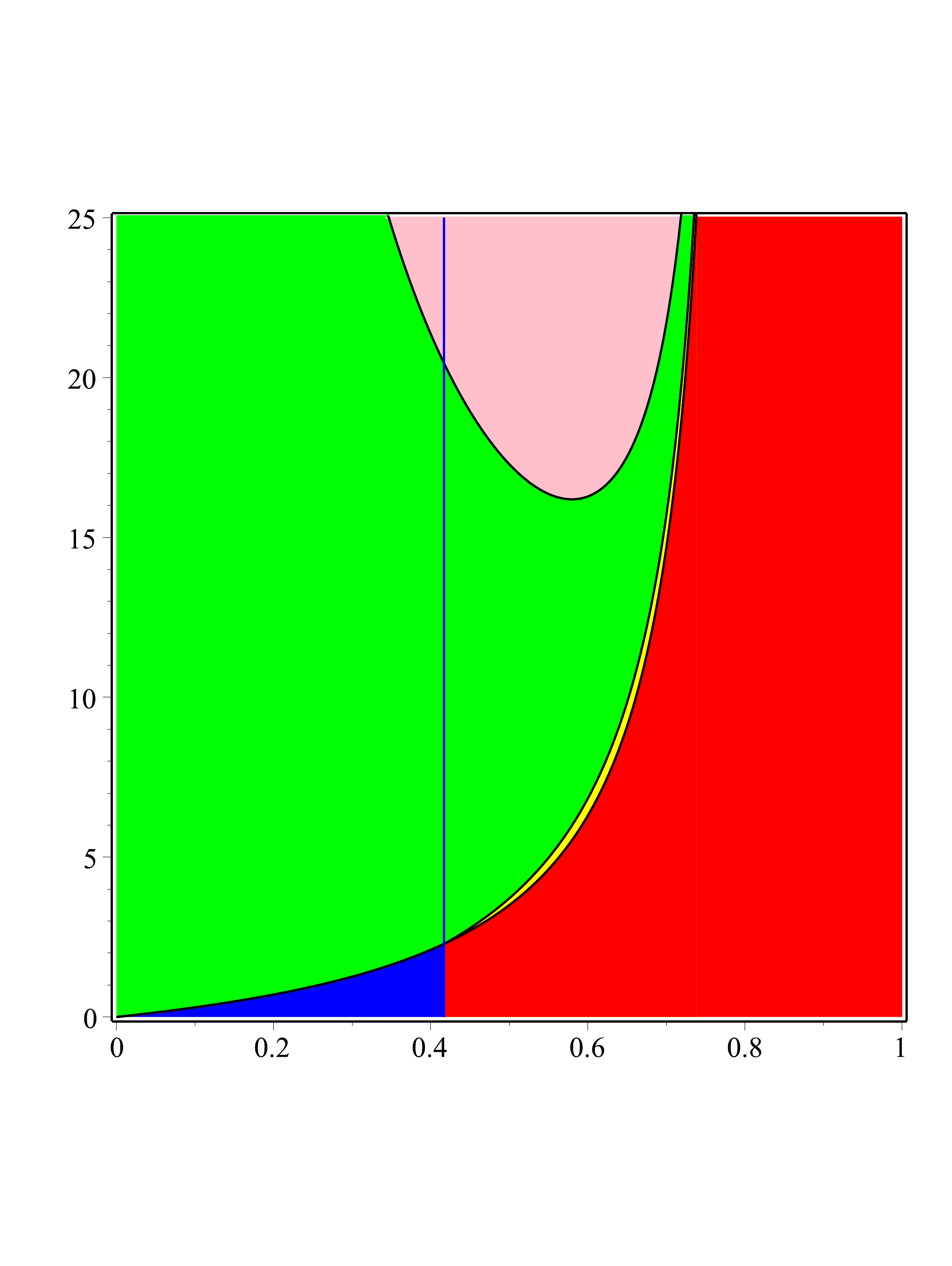}}}
\put(-3,-0.5){{\includegraphics[scale=0.34]{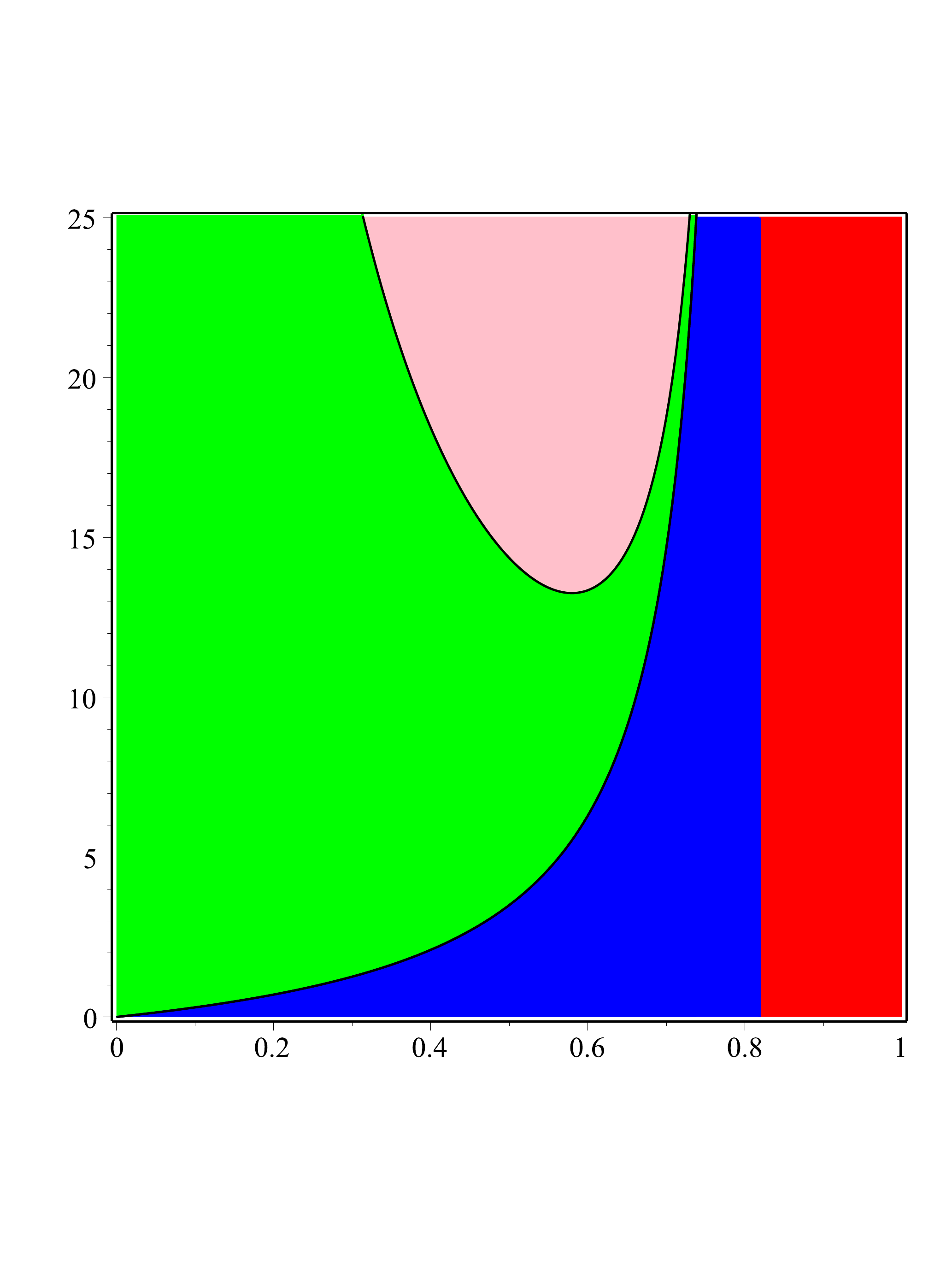}}}
\put(5,-0.5){{\includegraphics[scale=0.34]{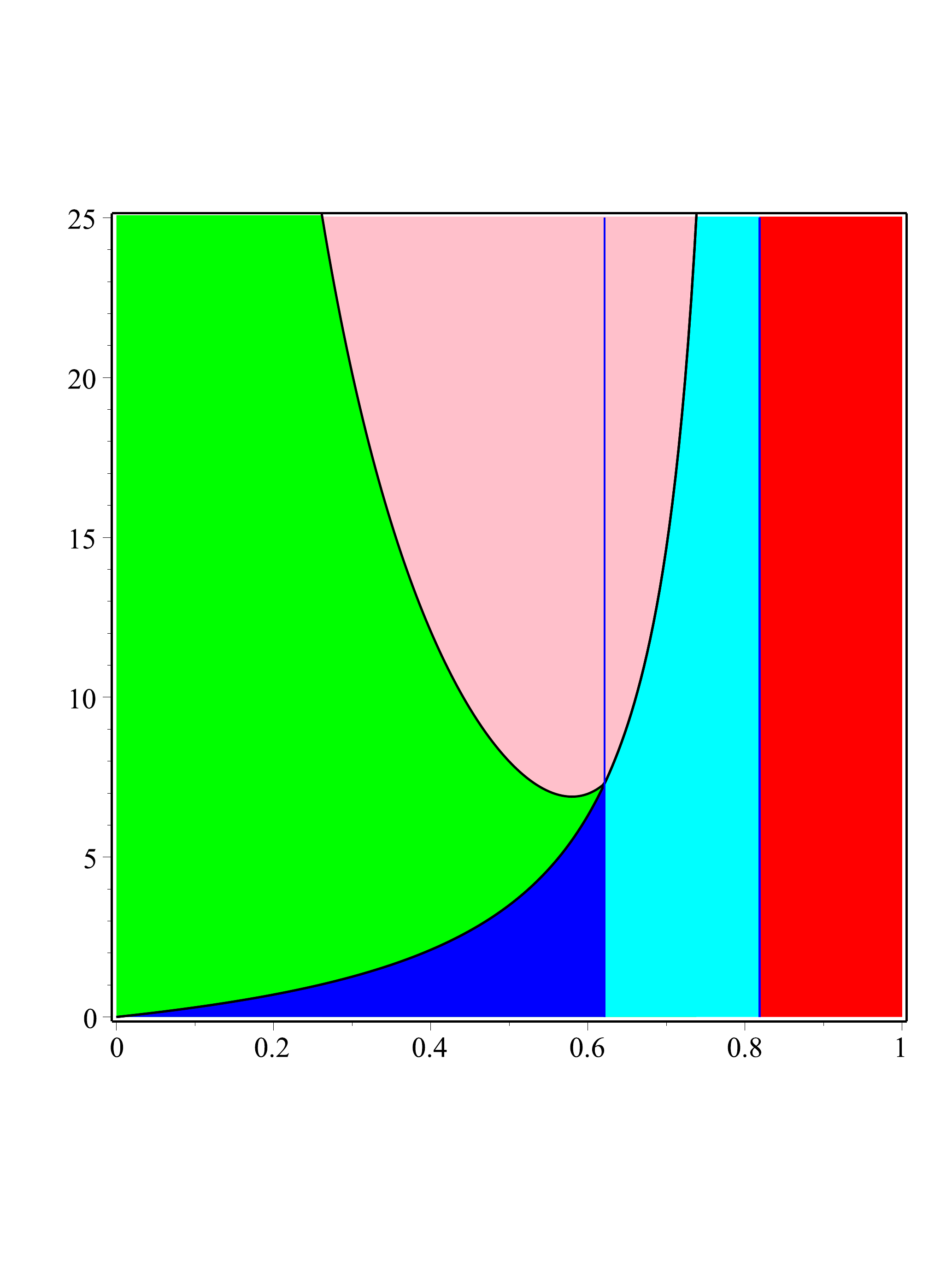}}}
%==============================================
\put(-3,8.5){(a)}
\put(5,8.5){(b)}	
\put(-3,1.5){(c)}
\put(5,1.5){(d)}		
\put(-2.9,11.4){$S_{1{\rm in}}$}
\put(5.1,11.4){$S_{1{\rm in}}$}	
\put(-2.9,4.4){$S_{1{\rm in}}$}
\put(5.1,4.4){$S_{1{\rm in}}$}	
\put(0.9,7.9){$D$}
\put(8.9,7.9){$D$}	
\put(0.9,0.9){$D$}
\put(8.9,0.9){$D$}
\put(2.7,10){$\mathcal{I}_0$}
\put(10.7,10){$\mathcal{I}_0$}
\put(3,4){$\mathcal{I}_0$}
\put(11,4){$\mathcal{I}_0$}
\put(7.9,8.45){$\mathcal{I}_1$}
\put(1.5,2){$\mathcal{I}_1$}
\put(8.7,1.6){$\mathcal{I}_1$}
\put(10,3){$\mathcal{I}_2$}	
\put(1.1,9){\vector(-1,2){0.2}}
\put(1.1,8.8){$\mathcal{I}_3$}
\put(9.46,9.3){\vector(-1,2){0.2}}
\put(9.45,9.1){$\mathcal{I}_3$}
\put(0,10){$\mathcal{I}_4$}
\put(8.7,10){$\mathcal{I}_4$}
\put(1,13.5){$\mathcal{I}_5$}
\put(9,13.5){$\mathcal{I}_5$}
\put(7,10){$\mathcal{I}_6$}
\put(-0.5,3){$\mathcal{I}_6$}
\put(7,3){$\mathcal{I}_6$}
\put(8,13.6){\vector(1,2){0.2}}
\put(7.7,13.4){$\mathcal{I}_7$}
\put(1,6){$\mathcal{I}_7$}
\put(8.6,6){$\mathcal{I}_7$}
\put(9.6,6){$\mathcal{I}_8$}
\put(2.3,11){\vector(-1,0){0.4}}
\put(2.3,10.85){$\Gamma_1$}
\put(10.3,11){\vector(-1,0){0.4}}
\put(10.3,10.85){$\Gamma_1$}
\put(2.3,4){\vector(-1,0){0.4}}
\put(2.3,3.85){$\Gamma_1$}
\put(10.3,4){\vector(-1,0){0.4}}
\put(10.3,3.85){$\Gamma_1$}
\put(8.2,14.5){$\Gamma_2$}
\put(8.7,7.5){$\Gamma_3$}
\put(1.25,11){\vector(1,0){0.4}}
\put(0.9,10.85){$\Gamma_4$}
\put(9.3,11){\vector(1,0){0.4}}
\put(8.95,10.85){$\Gamma_4$}
\put(-0.1,14.5){$\Gamma_5$}
\put(7.7,14.5){$\Gamma_5$}
\put(-0.3,7.5){$\Gamma_5$}
\put(7.4,7.5){$\Gamma_5$}
\put(2.6,7.5){$\Gamma_6$}
\put(10.6,7.5){$\Gamma_6$}
\end{picture}
\end{center}
\vspace{-0.9cm}
\caption{The 2-dimensional operating diagram 
$\left(D,S_{1{\rm in}}\right)$ obtained by cuts at $S_{2{\rm in}}$ constant of the 3-dimensional operating diagram corresponding to Fig.~\ref{figHi}(c). (a):
$S_{2{\rm in}}=0$, 
(b): $S_{2{\rm in}}=7$, 
(c): $S_{2{\rm in}}=S_2^M\simeq 36.3$ and (d): $S_{2{\rm in}}=100$.}\label{DS1inBenyahia4}
\end{figure}
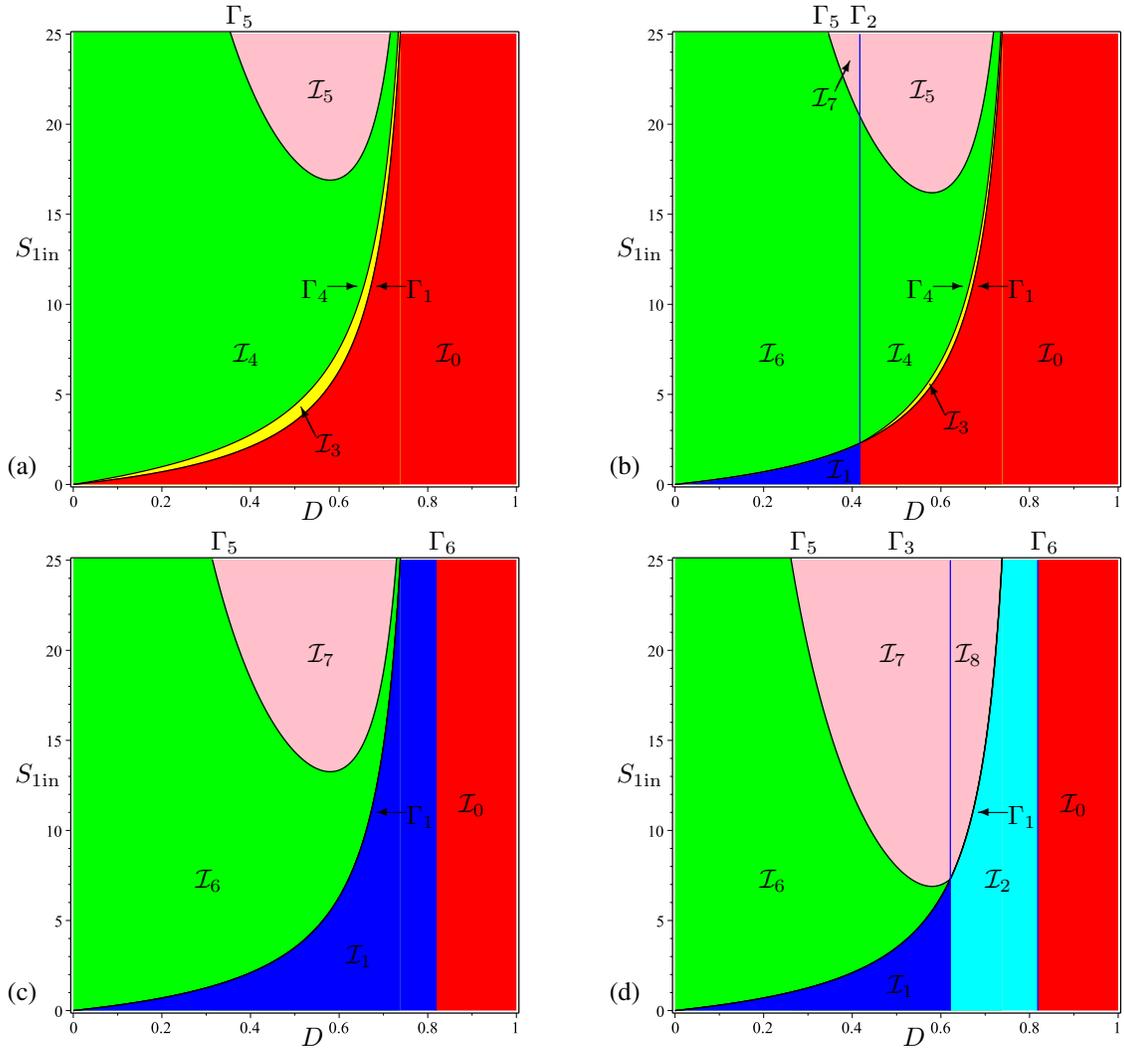

\subsection{Operating diagram when $D_1<D_2$}

The cuts at $S_{2{\rm in}}$ constant of the 3-dimensional operating diagram corresponding to Fig.~\ref{figHi}(c), are shown in Fig.~\ref{DS1inBenyahia4}. The regions are colored according to the colors in 
Table~\ref{Table9cases}. 
Since $D_1<D_2$ there exists a value 
$S_2^0<S_2^M$ such that $\mu_2(S_2^0)=\alpha D_1$. 

Fig.~\ref{DS1inBenyahia4} shows the following features.
For 
$S_{2{\rm in}}=0$, only regions
$\mathcal{I}_0$, $\mathcal{I}_3$, 
$\mathcal{I}_4$ and $\mathcal{I}_5$ appear, 
see Fig.~\ref{DS1inBenyahia4}(a). 
For $0<S_{2{\rm in}}<S_2^0$, 
$\Gamma_2$ curve appears, giving birth to $\mathcal{I}_1$, $\mathcal{I}_6$, 
$\mathcal{I}_7$ regions, see 
Fig.~\ref{DS1inBenyahia4}(b). 
For increasing $S_{2{\rm in}}$,
$\Gamma_4$ and 
$\Gamma_5$ curves are translated downwards, 
while the vertical line 
$\Gamma_2$ moves to the right and
 tends towards the common vertical asymptote $D=D_1$ for curves $\Gamma_1$, $\Gamma_4$ and 
$\Gamma_5$, as $S_{2{\rm in}}$ tends to $S_2^0$. In the limit 
$S_{2{\rm in}}=S_2^0$, the very tiny region
$\mathcal{I}_3$ (in Yellow on the figure) located between 
curves $\Gamma_1$ and 
$\Gamma_4$, together with $\mathcal{I}_4$ and 
$\mathcal{I}_5$ regions have disappeared.  

For $S_2^0<S_{2{\rm in}}\leq S_2^M$, only regions
$\mathcal{I}_0$, $\mathcal{I}_1$, 
$\mathcal{I}_6$ and $\mathcal{I}_7$ exist. For increasing $S_{2{\rm in}}$, 
the vertical line 
$\Gamma_2$ moves to the right and tends towards $\Gamma_6$ as $S_{2{\rm in}}$ tends to $S_2^M$, see 
Fig.~\ref{DS1inBenyahia4}(c).
For $S_{2{\rm in}}>S_2^M$, $\Gamma_3$ curve appears, giving birth to $\mathcal{I}_2$
and $\mathcal{I}_8$ regions, see 
Fig.~\ref{DS1inBenyahia4}(d). 
For increasing $S_{2{\rm in}}$, 
the vertical line $\Gamma_3$ moves to the left while $\Gamma_5$ curve is translated downwards.

It should be noticed that as the case (B), it is seen in  
Fig.~\ref{DS1inBenyahia4} that 
the region of global asymptotic stability of the positive steady state $E_2^1$ (the Green region 
$\mathcal{I}_4\cup\mathcal{I}_6$)  presents the property that there exists a range of values for the operating parameters $S_{1{\rm in}}$ and 
$S_{2{\rm in}}$ such that the system can go from the bistability region 
(the Pink region 
$\mathcal{I}_5\cup\mathcal{I}_7$), 
to the global asymptotic stability region, when the dilution rate $D$ increases.

\section{Bifurcations}\label{sec4}

The surfaces $\Gamma_k$, $k=1\cdots 6$, are the borders of the regions in the operating parameters space 
$(D,S_{1{\rm in}},S_{2{\rm in}})$ on which bifurcations occur, while the steady states change their stability. In codimension-one bifurcations, only transcritical and saddle node bifurcations can be encountered, as stated in the following result. 

%%%%%%%%%%%%%%%%%%%
\begin{propo}                                           \label{proBifurcations}

The bifurcations of the steady states of (\ref{AM2})  arising on the boundaries of regions  $\mathcal{I}_k$, $k=0\cdots 8$, are listed in Table \ref{Bifurcation}.
\end{propo}
%=====================================================================
\begin{proof}
The proof is given in Appendix \ref{ProofproBifurcations}.	
\end{proof}
%%%%%%%%%%%%%%%%%%%%%%%%

\begin{table}[ht]
\caption{Codimension-one bifurcations along subsets of surfaces $\Gamma_k$ and the corresponding cases in \cite{Benyahia}: Transcritical bifurcations (TB) and Saddle Node bifurcations (SNB) occur.}\label{Bifurcation}
\begin{center}
\begin{tabular}{c|c|l|l}
$\Gamma_k$
&
Subset of $\Gamma_k$
& 
Bifurcation
&
{Case} of \cite{Benyahia}
\\
\hline
  \multirow{3}{*}{{$\Gamma_1$}}
  &   
$\Gamma_1\cap\left\{0\leq S_{2{\rm in}}<S_2^{1*}(D)\right\}$
&
TB: $E_1^0=E_2^0$
\\
&
$\Gamma_1\cap\left\{S_2^{1*}(D)<S_{2{\rm in}}<S_2^{2*}(D)\right\}$
&
TB: $E_1^i=E_2^i$, $i=0,1$
\\
&
$\Gamma_1\cap\left\{S_{2{\rm in}}>S_2^{2*}(D)\right\}$
&
TB: $E_1^i=E_2^i$, $i=0,1,2$ 
  \\
 \hline
$\Gamma_2$
  &   
$\Gamma_2$
&
TB: $E_1^0=E_1^1$
&
{\bf 1.4}, {\bf 2.8}, {\bf 2.9}
  \\
% \hline
 $\Gamma_3$
  &   
$\Gamma_3$
&
TB: $E_1^0=E_1^2$
&
{\bf 1.5}, {\bf 2.13}
  \\
% \hline
$\Gamma_4$
&
$\Gamma_4$
&
TB: $E_2^0=E_2^1$
&
{\bf 2.7}
\\
%\hline
$\Gamma_5$
&
$\Gamma_5$
&
TB: $E_2^0=E_2^2$
&
{\bf 2.12}, {\bf 2.15}
  \\
  \hline
\multirow{3}{*}{
\begin{tabular}{c}
$\Gamma_6$\\
$D_2<D_1$
\end{tabular}
}
&
\begin{tabular}{l}
$\Gamma_6\cap\left\{0\leq S_{2{\rm in}}<S_2^M\mbox{ and }\right.$\\
$\qquad\left. S_{1{\rm in}}>S_1^*(D_2)+\frac{k_1}{k_2}\left(S_2^M-S_{2{\rm in}}\right)
\right\}$
\end{tabular}
&
SNB: $E_2^1=E_2^2$
&
{\bf 2.11}
\\
&
\begin{tabular}{l}
$\Gamma_6\cap\left\{S_{2{\rm in}}>S_2^M\mbox{ and }
S_{1{\rm in}}>S_1^*(D_2)\right\}$
\end{tabular}
&
SNB: $E_j^1=E_j^2$, $j=1,2$
&
{\bf 2.14}
\\
&
\begin{tabular}{l}
$\Gamma_6\cap\left\{S_{2{\rm in}}>S_2^M\mbox{ and }
S_{1{\rm in}}<S_1^*(D_2)
\right\}$
\end{tabular}
&
SNB: $E_1^1=E_1^2$
&
{\bf 1.6}
  \\
 \hline
\begin{tabular}{c}
$\Gamma_6$\\
$D_1<D_2$
\end{tabular}
&
$\Gamma_6\cap\left\{S_{2{\rm in}}>S_2^M\mbox{ and } S_{1{\rm in}}>0
\right\}$
&
SNB: $E_1^1=E_1^2$
&
{\bf 1.6}
\end{tabular}  
\end{center}
  \end{table}

\begin{rem}
The last column of Table \ref{Bifurcation} shows the corresponding cases with {\em non hyperbolic} steady states given in Theorem 1  of \cite{Benyahia}. The case labeled 
{\bf 2.10} in this theorem, where 
$E_1^0=E_1^1$ and $E_2^0=E_2^2$,
 does not appear in Table \ref{Bifurcation}, since it is a codimension-two bifurcation arising along 
 $\Gamma_2\cap\Gamma_5$. The bifurcations along 
 $\Gamma_1$, corresponding to the condition $S_{1{\rm in}}=S_1^*(D)$ were not analyzed in \cite{Benyahia}. In Theorem 1  of \cite{Benyahia} only the cases $S_{1{\rm in}}<S_1^*(D)$ and $S_{1{\rm in}}>S_1^*(D)$ were considered.
\end{rem}

To have a better understanding of the nature of the bifurcations of steady states, let us consider the dilution rate $D$ as the bifurcation parameter. Throughout this section, we assume that biological parameters are fixed as 
in Fig. \ref{DS1inBenyahia5}(a), corresponding to case (b) of Fig. \ref{figHi} and 
$S_{2{\rm in}}=0$.
We now fix the operating parameter 
$S_{1{\rm in}}$ at various typical values, 
as depicted in the horizontal lines shown in Fig. \ref{figBifS1inBenyahia5S2in0}, and plot one-parameter bifurcation diagrams in 
$D$, with $X_i$, $i=1,2$, on the $y$-axis, see Fig.~\ref{figBifSin14}, \ref{figBifSin13} and \ref{figBifSin13.3}.

%======================================
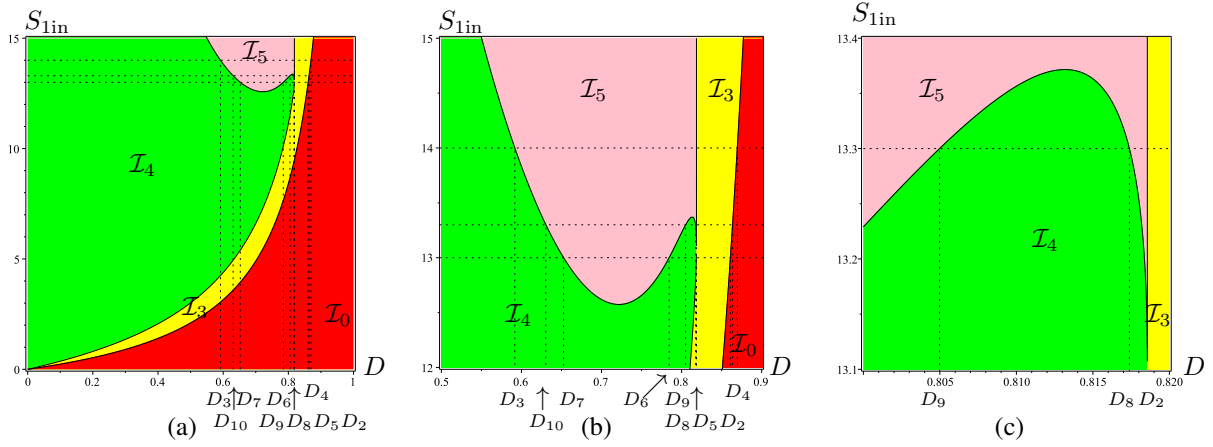
\begin{figure}[ht]
\setlength{\unitlength}{1.0cm}
\begin{center}
\begin{picture}(8.5,6.5)(0,-0.2)		
\put(-4,0){{\includegraphics[scale=0.25]{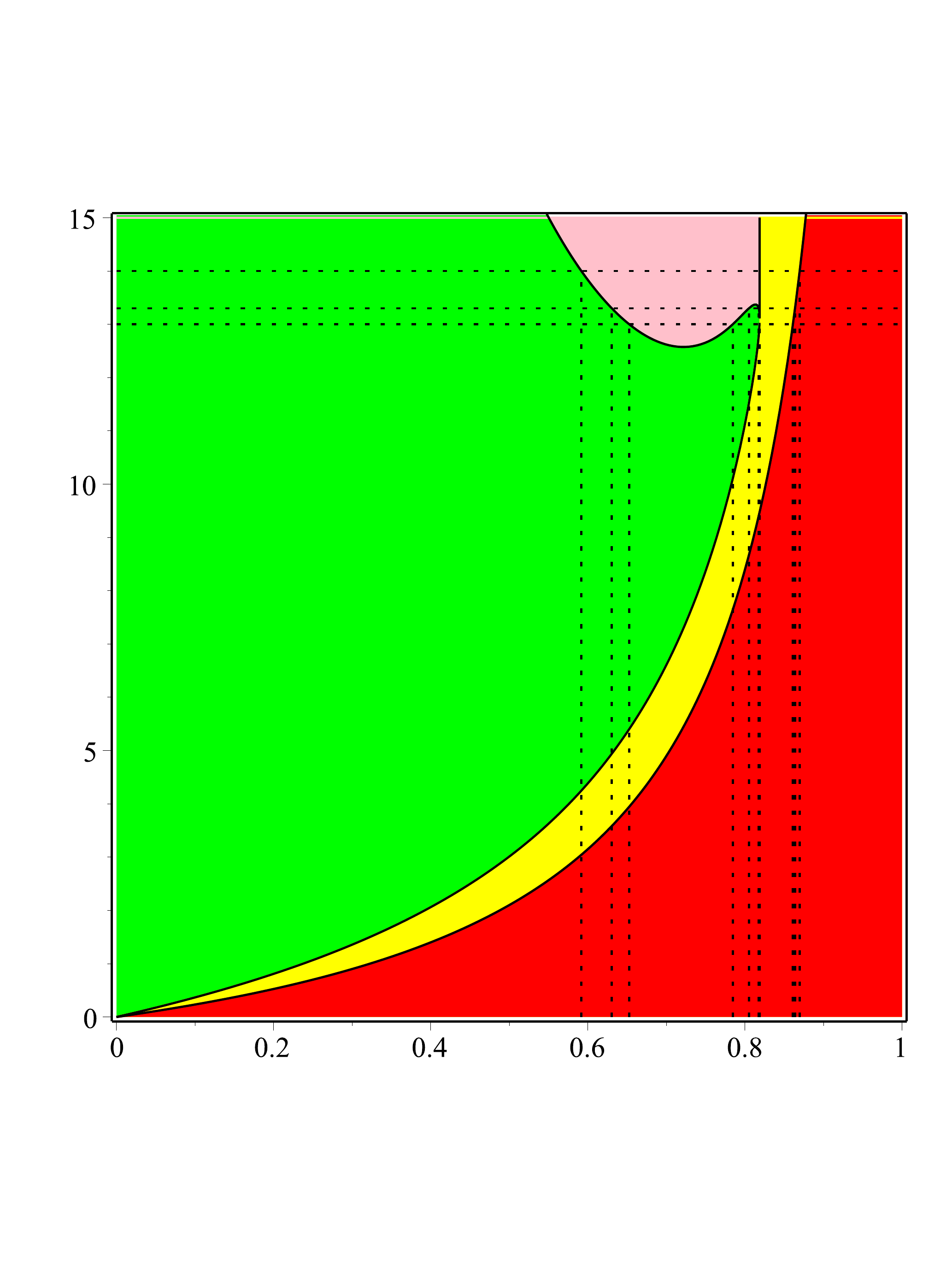}}}
\put(1.5,0){{\includegraphics[scale=0.25]{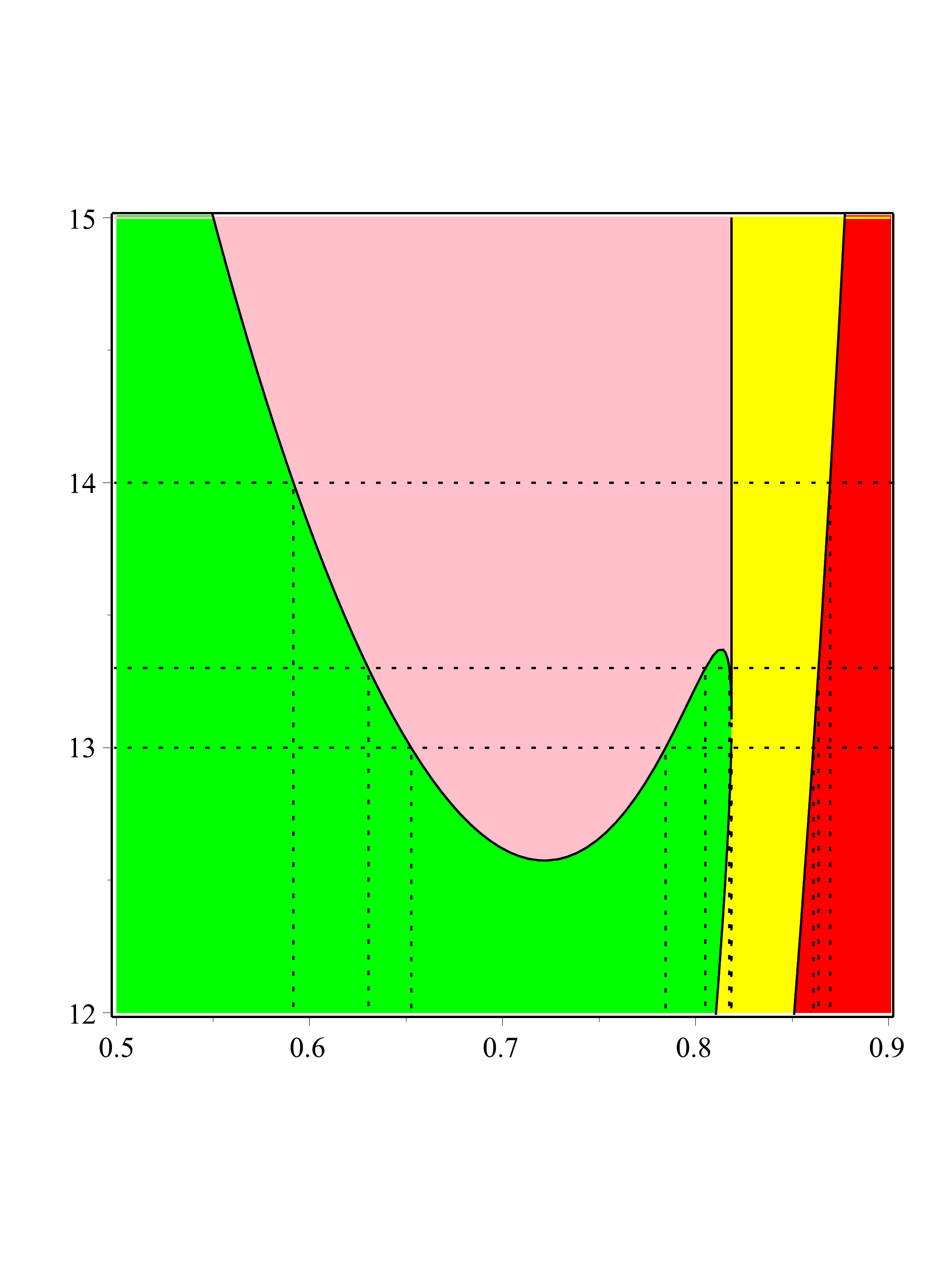}}}
\put(7,0){{\includegraphics[scale=0.25]{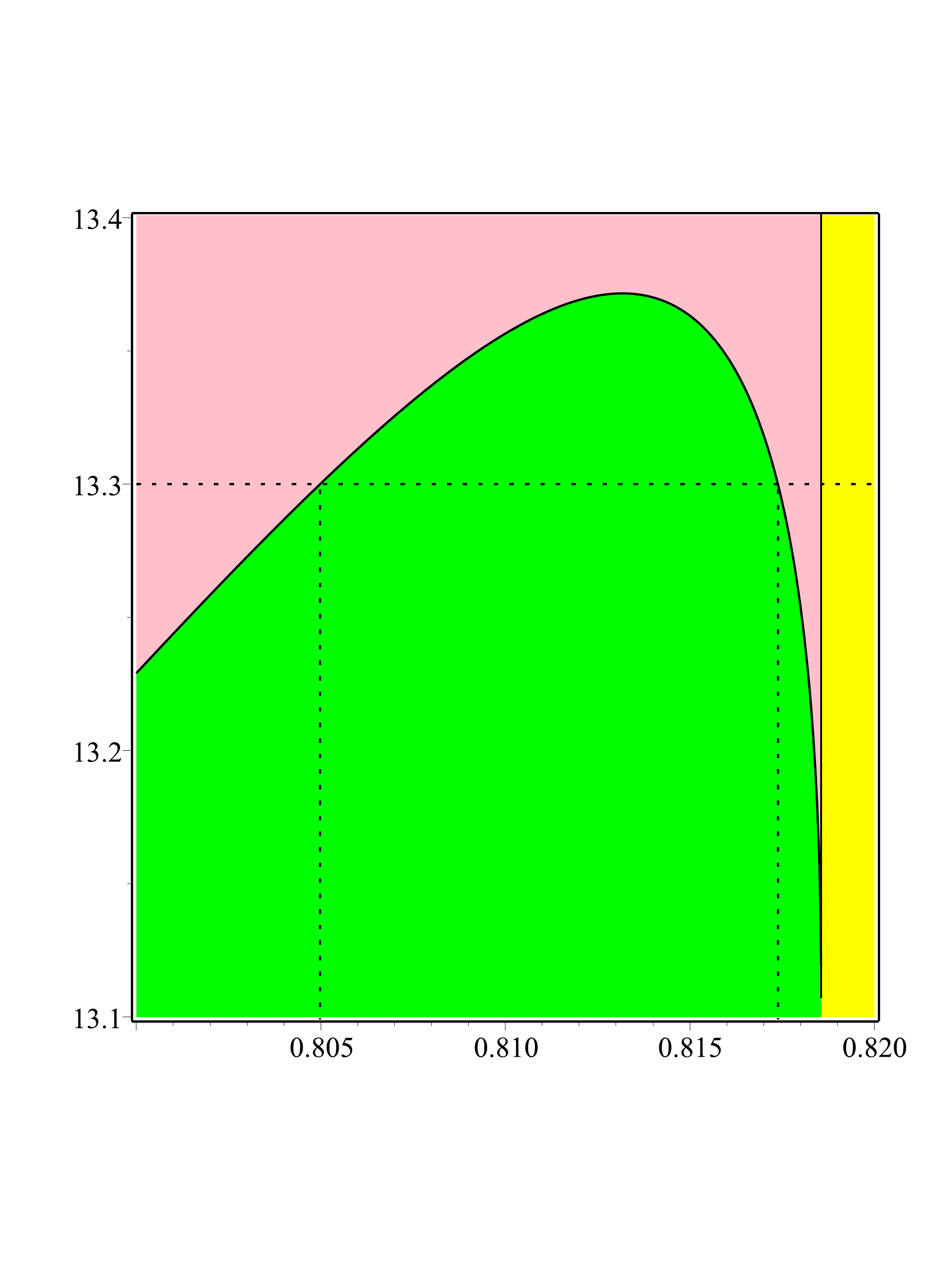}}}
%==============================================
\put(-1.5,0.5){(a)}
\put(4,0.5){(b)}	
\put(9.5,0.5){(c)}	
%%%%%%%%%%%%%%%%%%
\put(-1,0.9){\scriptsize$D_3$}
\put(-0.6,0.9){\scriptsize$D_7$}
\put(-0.7,0.9){$\uparrow$}
\put(-0.9,0.6){\scriptsize$D_{10}$}
\put(-0.2,0.9){\scriptsize$D_6$}
\put(0.1,0.9){$\uparrow$}
\put(-0.3,0.6){\scriptsize$D_{9}D_8D_5D_2$}
\put(0.3,1){\scriptsize$D_4$}
%%%%%%%%%%%%%%%%%%%%%%%
\put(2.9,0.9){\scriptsize$D_3$}
\put(3.7,0.9){\scriptsize$D_7$}
\put(3.4,0.9){$\uparrow$}
\put(3.3,0.6){\scriptsize$D_{10}$}
\put(4.8,1.05){$\nearrow$}
\put(4.55,0.9){\scriptsize$D_6$}
\put(5.1,0.9){\scriptsize$D_9$}
\put(5.45,0.9){$\uparrow$}
\put(5.1,0.6){\scriptsize$D_8D_5D_2$}
\put(5.9,1){\scriptsize$D_4$}
%%%%%%%%%%%%%%
\put(8.5,0.9){\scriptsize$D_9$}
\put(11,0.9){\scriptsize$D_8$}
\put(11.4,0.9){\scriptsize$D_2$}
%%%%%%%%%%%%%%%%
\put(1.1,1.3){$D$}
\put(6.5,1.3){$D$}	
\put(12,1.3){$D$}	
\put(-3.4,5.9){$S_{1{\rm in}}$}
\put(2.1,5.9){$S_{1{\rm in}}$}	
\put(7.6,6){$S_{1{\rm in}}$}
%%%%%%%%%%%%%%%%%%%%%
\put(0.6,2){{$\mathcal{I}_0$}}
\put(-1.3,2.1){{$\mathcal{I}_3$}}
\put(-2,4){{$\mathcal{I}_4$}}
\put(-0.5,5.5){{$\mathcal{I}_5$}}
%%%%%%%%%%%%%%%%%%%%%%%%%%%%%%
\put(3,2){{$\mathcal{I}_4$}}
\put(5.7,5){{$\mathcal{I}_3$}}
\put(4,5){{$\mathcal{I}_5$}}
\put(6,1.6){{$\mathcal{I}_0$}}
%%%%%%%%%%%%%%%%%%%%%%%%
\put(10,3){{$\mathcal{I}_4$}}
\put(11.5,2){{$\mathcal{I}_3$}}
\put(8.5,5){{$\mathcal{I}_5$}}
\end{picture}
\end{center}
\vspace{-0.9cm}
\caption{Operating diagram where $S_{2{\rm in}}=0$ corresponding to Fig.~\ref{DS1inBenyahia5}(a).  
(a): Cuts where $S_{1{\rm in}}$ is kept constant and $D$ is the bifurcation parameter.
(b): Magnification of the operating diagram showing the bifurcation values $D_k$, defined by (\ref{Dk1}), (\ref{Dk2}), and (\ref{Dk3}). Notice that there are three different values of $D_4$ corresponding to the three different values $S_{1{\rm in}}=13$, $S_{1{\rm in}}=13.3$ and $S_{1{\rm in}}=14$. (c) : Magnification showing the values $D=D_9$, $D=D_8$ and $D=D_2$.} \label{figBifS1inBenyahia5S2in0}
\end{figure}
%=======================================

Recall that the curve $\Gamma_5$ separating the Pink and Green regions is the curve of the function 
$S_{1{\rm in}}=\frac{k_1}{k_2}H_2(D)$. 
Case (B) corresponds to a function $H_2$ which is decreasing, then increasing, then decreasing. For the considered biological parameters values, the function $H_2(D)$ attains its minimum for 
$D_{min}\simeq 0.72$ and its maximum for 
$D_{max}=0.81$ and satisfies
$H_2(D_2)=S_2^M+\frac{k_2}{k_1}S_1^*(D_2)\simeq 131.1$, where 
$D_2=\frac{1}{\alpha}\mu_2\left(S_2^M\right)\simeq 0.82$.
Therefore, the variations of $\frac{k_1}{k_2}H_2(D)$ are as shown in the following table
\begin{center}
\begin{tabular}{c|ccccccc}
$D$&0&&0.72&&0.81&&0.82\\
\hline
$\frac{k_1}{k_2}H_2(D)$
&
$+\infty$
&
$\searrow$
&
12.57
&
$\nearrow$
&
13.37
&
$\searrow$
&
13.11
\end{tabular}
\end{center}

We fix three typical values $S_{1{\rm in}}=13$, $S_{1{\rm in}}=13.3$ and $S_{1{\rm in}}=14$, corresponding to the three horizontal lines shown in Fig. \ref{figBifS1inBenyahia5S2in0}.
We begin with the case where $S_{1{\rm in}}=14$.
Since $S_{1{\rm in}}>13.37$, as it is seen in Fig. \ref{figBifS1inBenyahia5S2in0}, 
with increasing $D$, there is a transition from 
$\mathcal{I}_4$ to $\mathcal{I}_5$ for $D=D_3\approx 0.5917$, 
then from $\mathcal{I}_5$ to 
$\mathcal{I}_3$ for $D=D_2\simeq 0.8186 $, 
then from $\mathcal{I}_3$ to 
$\mathcal{I}_0$ for 
$D=D_4\simeq 0.8696$. 
The bifurcation values $D_2$, $D_3$ and $D_4$ are defined by
\begin{equation}
\label{Dk1}
D_4=\frac{\mu_1\left(S_{1{\rm in}}\right)}{\alpha},
\quad
D_2=\frac{\mu_2\left(S_2^M\right)}{\alpha}
\quad \mbox{ and }
D_3
\mbox{ is the unique solution of }
S_{1{\rm in}}=\frac{k_1}{k_2}H_2(D)
\end{equation}
The bifurcation value 
$D_4$ corresponds to a transcritical bifurcation of $E_2^0$ and $E_1^0$; 
$D_2$ corresponds to a saddle node bifurcation of $E_2^1$ and $E_2^2$ and $D_3$ corresponds to a transcritical bifurcation of $E_2^0$ and $E_2^2$. The plot of $X_1$ and $X_2$ components of all existing steady states with respect of $D$ is shown in Fig.~\ref{figBifSin14}. Solid lines and dotted lines correspond to stable and unstable steady states respectively. Since $S_{2{\rm in}}$ the steady states $E_1^1$ and $E_1^2$ cannot exist. On Fig.~\ref{figBifSin14}(a), 
for $0<D<D_3$, the $X_1$-component of $E_2^i$, $i=0,2$, is colored in Red, with Green dots, showing the stability of $E_2^1$ and the instability of $E_2^0$.
 For $D_3<D<D_2$, the $X_1$-component of $E_2^i$, $i=0,1,2$, is colored in Red and Green, with Blue dots, showing the bistability of $E_2^0$ and $E_2^1$ and the instability of $E_2^1$.
 For $D_2<D<D_4$, the $X_1$-component of $E_2^0$ is colored in Green, showing the stability of $E_2^0$. 
On Fig.~\ref{figBifSin14}(b) 
and \ref{figBifSin14}(c), 
for $0<D<D_2$, the $X_2=0$-component of $E_j^0$, $j=1,2$, is colored with Green and Black dots, showing the instability of $E_2^0$ and $E_1^0$; For $D_2<D<D_4$ it is colored in Green, with Black dots, showing the stability of $E_2^0$ and the instability of $E_1^0$. For $D>D_4$ it is colored in Black showing the stability of $E_1^0$
%======================================

%======================================
\begin{figure}[ht]
\setlength{\unitlength}{1.0cm}
\begin{center}
\begin{picture}(8.5,6.5)(0,-0.2)		
\put(-4,0){{\includegraphics[scale=0.25]{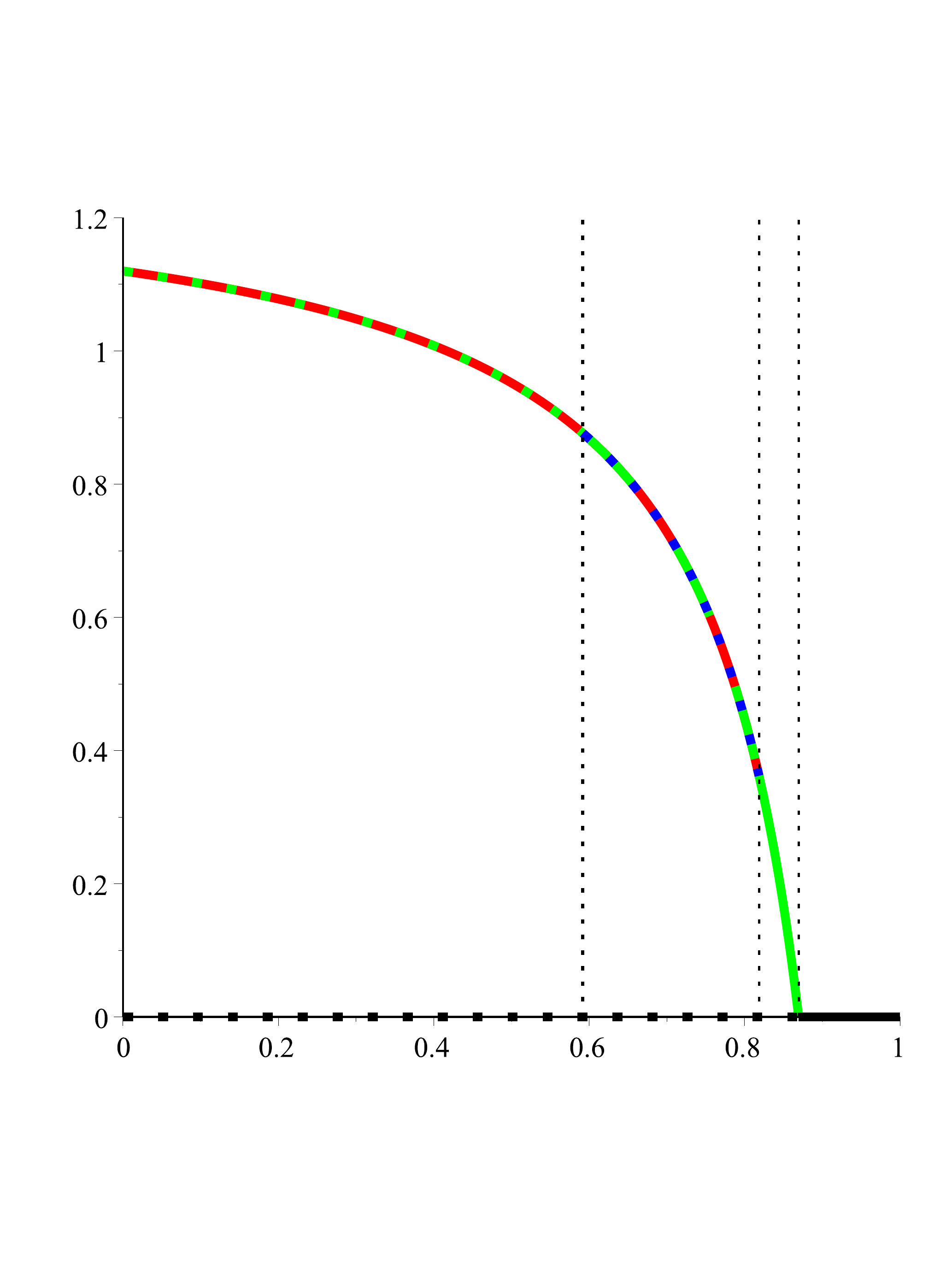}}}
\put(1.5,0){{\includegraphics[scale=0.25]{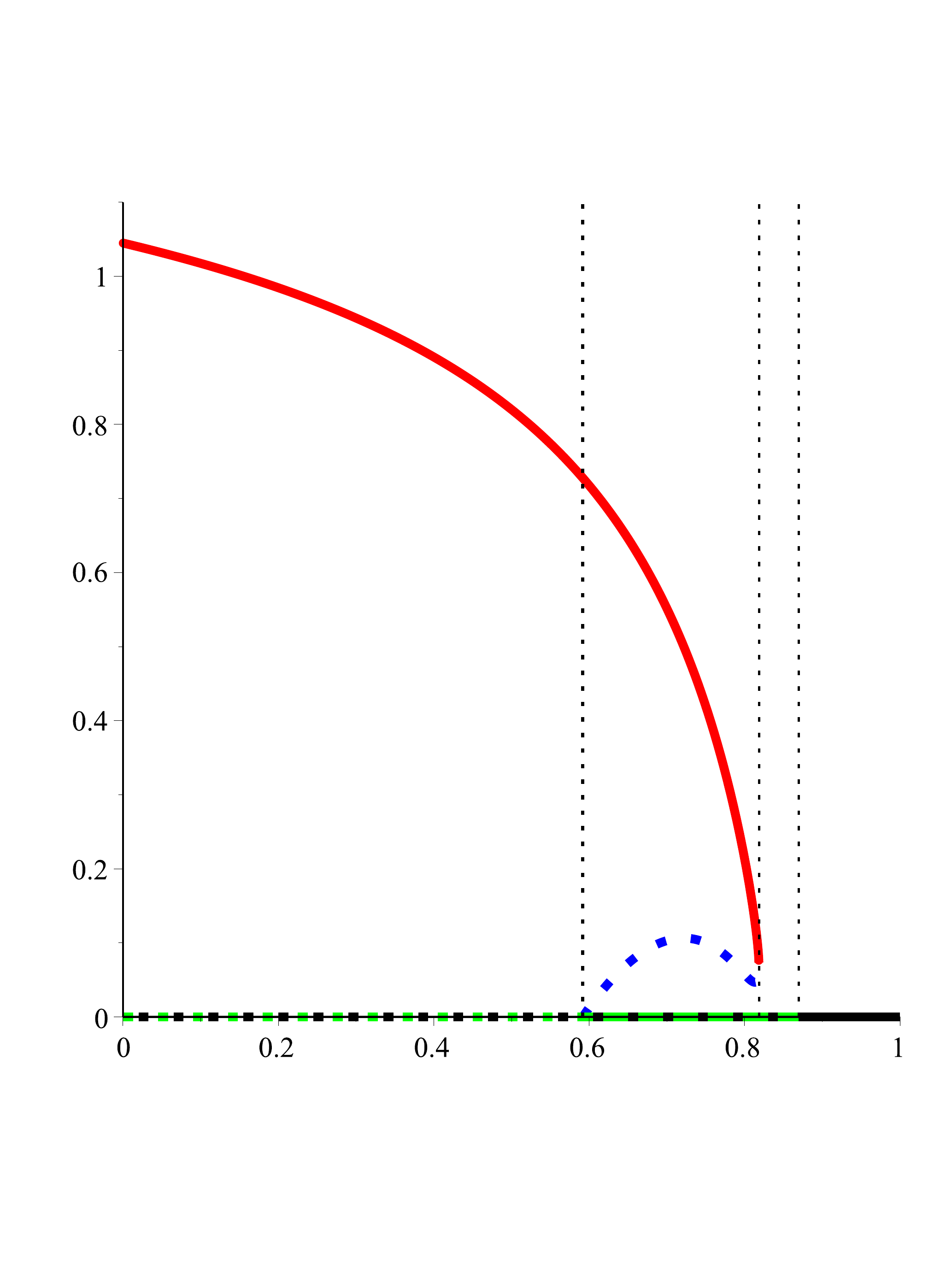}}}
\put(7,0){{\includegraphics[scale=0.25]{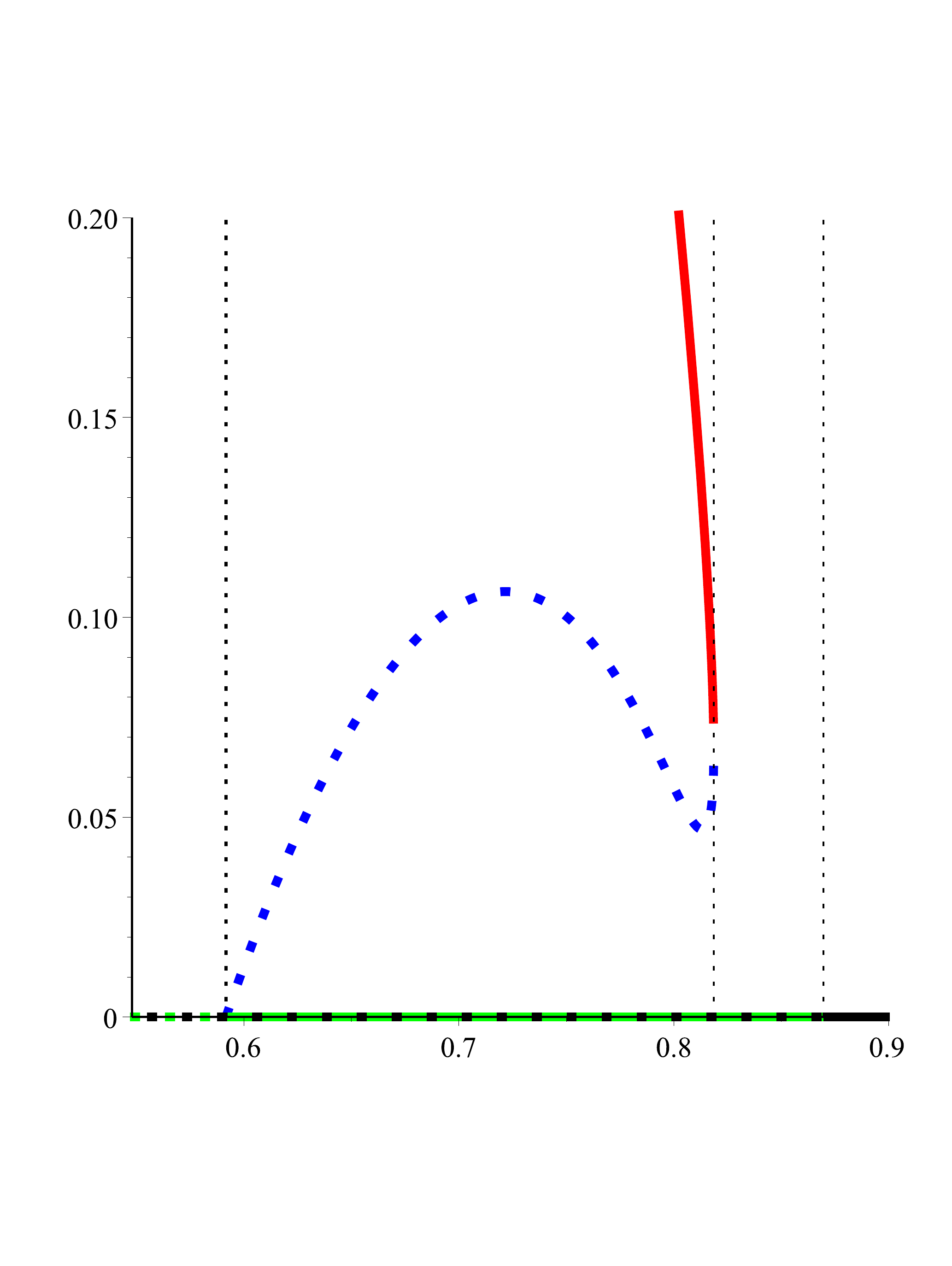}}}
%==============================================
\put(-1.5,0.5){(a)}
\put(4,0.5){(b)}	
\put(9.5,0.5){(c)}	
%%%%%%%%%%%%%%%%%%
\put(-1,0.9){\scriptsize$D_3$}
\put(0,0.9){\scriptsize$D_2$}
\put(0.4,0.9){\scriptsize$D_4$}
\put(4.5,0.9){\scriptsize$D_3$}
\put(5.5,0.9){\scriptsize$D_2$}
\put(5.9,0.9){\scriptsize$D_4$}
\put(8.1,0.9){\scriptsize$D_3$}
\put(10.8,0.9){\scriptsize$D_2$}
\put(11.4,0.9){\scriptsize$D_4$}
\put(1,1.3){$D$}
\put(6.5,1.3){$D$}	
\put(12,1.3){$D$}	
\put(-3.4,5.9){$X_1$}
\put(2.1,6){$X_2$}	
\put(7.6,6){$X_2$}
%%%%%%%%%%%%%%%%%%%%%
\put(-1.7,1.55){{$E_1^0$}}
\put(-0.4,4.25){{$E_2^2$}}
\put(-1.6,5.2){{$E_2^1$}}
\put(-0.45,3){{$E_2^0$}}
%%%%%%%%%%%%%%%%%%%%%%%%%%%%%%
\put(3.4,1.55){{$E_1^0$~$E_2^0$}}
\put(5.1,1.95){{$E_2^2$}}
\put(3.9,5.1){{$E_2^1$}}
%%%%%%%%%%%%%%%%%%%%%%%%
\put(9.5,3.85){{$E_2^2$}}
\put(10.3,5.5){{$E_2^1$}}
\put(9,1.55){{$E_1^0$ $E_2^0$}}
\end{picture}
\end{center}
\vspace{-0.9cm}
\caption{Bifurcation diagram with $D$ as the bifurcation parameter, corresponding to Fig. \ref{DS1inBenyahia5}(a) and 
$S_{1{\rm in}}=14$. 
(a): The $X_1$-components and 
(b): the $X_2$-components, of the steady states $E_1^0$ (in Black), $E_2^0$ (in Green), $E_2^1$ (in Red) and $E_2^2$ (in Blue). 
 (c): A magnification showing the bifurcation values $D_1$,  $D_2$ and $D_3$. 
Solid lines and dotted lines correspond to stable and unstable steady states respectively.  \label{figBifSin14}}
\end{figure}
%=======================================

%======================================
\begin{figure}[ht]
\setlength{\unitlength}{1.0cm}
\begin{center}
\begin{picture}(8.5,6.5)(0,-0.2)		
\put(-4,0){{\includegraphics[scale=0.25]{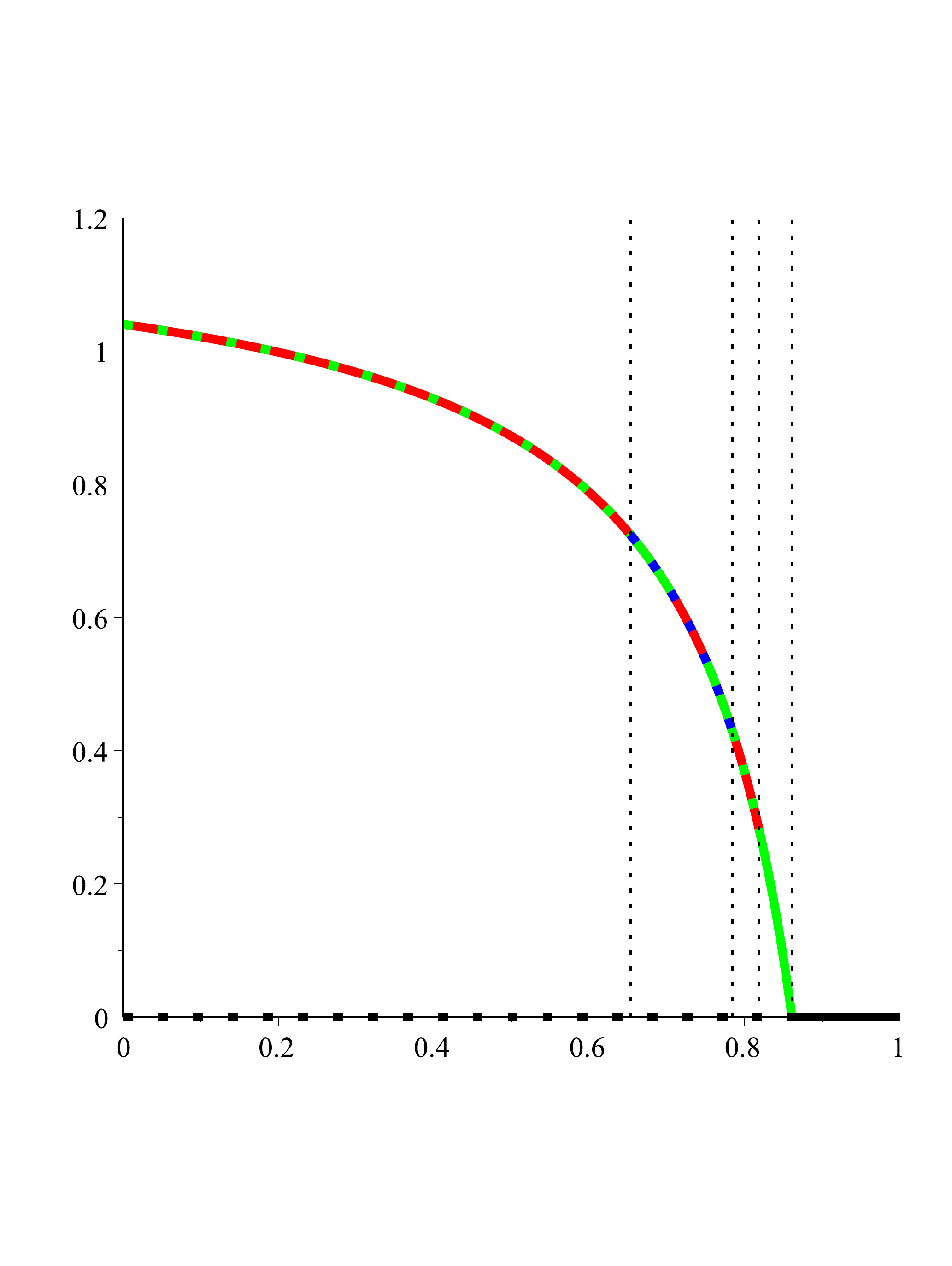}}}
\put(1.5,0){{\includegraphics[scale=0.25]{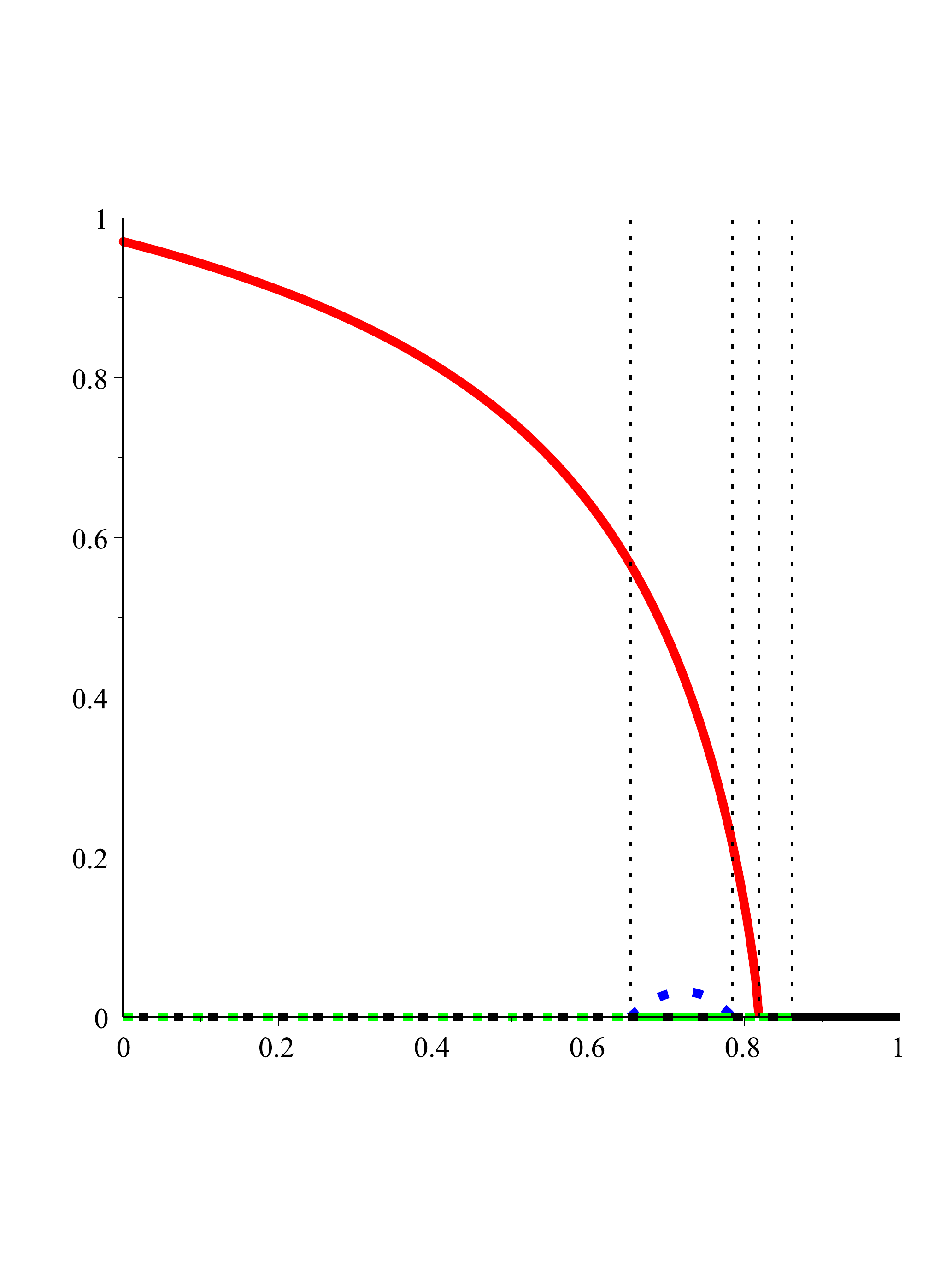}}}
\put(7,0){{\includegraphics[scale=0.25]{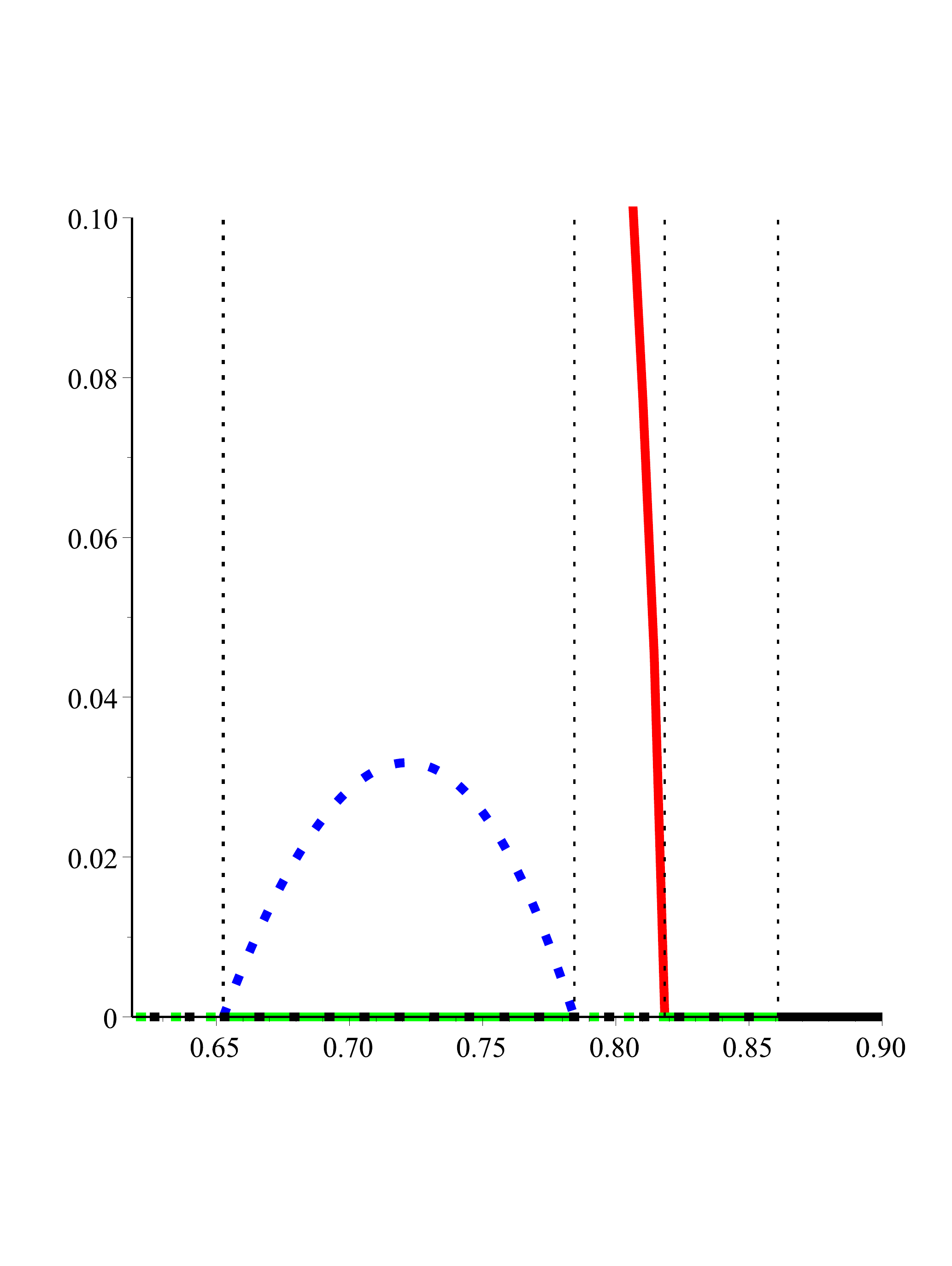}}}
%==============================================
\put(-1.5,0.5){(a)}
\put(4,0.5){(b)}	
\put(9.5,0.5){(c)}	
%%%%%%%%%%%%%%%%%%
\put(-0.7,0.9){\scriptsize$D_7$}
\put(-0.3,0.9){\scriptsize$D_6$}
\put(0.05,0.9){\scriptsize$D_5$}
\put(0.4,0.9){\scriptsize$D_4$}
\put(4.8,0.9){\scriptsize$D_7$}
\put(5.2,0.9){\scriptsize$D_6$}
\put(5.55,0.9){\scriptsize$D_5$}
\put(5.9,0.9){\scriptsize$D_4$}
\put(8.1,0.9){\scriptsize$D_7$}
\put(10,0.9){\scriptsize$D_6$}
\put(10.55,0.9){\scriptsize$D_5$}
\put(11.15,0.9){\scriptsize$D_4$}
\put(1,1.3){$D$}
\put(6.5,1.3){$D$}	
\put(12,1.3){$D$}	
\put(-3.4,5.9){$X_1$}
\put(2.1,6){$X_2$}	
\put(7.6,6){$X_2$}
%%%%%%%%%%%%%%%%%%%%%
\put(-1.7,1.55){{$E_1^0$}}
\put(-0.45,4){{$E_2^2$}}
\put(-1.6,4.85){{$E_2^1$}}
\put(-0.45,2.7){{$E_2^0$}}
%%%%%%%%%%%%%%%%%%%%%%%%%%%%%%
\put(3.4,1.55){{$E_1^0$~$E_2^0$}}
\put(5,1.7){{$E_2^2$}}
\put(3.9,5.1){{$E_2^1$}}
%%%%%%%%%%%%%%%%%%%%%%%%
\put(9.1,2.9){{$E_2^2$}}
\put(10,5.5){{$E_2^1$}}
\put(8.6,1.55){{$E_1^0$ $E_2^0$}}
\end{picture}
\end{center}
\vspace{-0.9cm}
\caption{Bifurcation diagram with $D$ as the bifurcation parameter, corresponding to Fig. \ref{DS1inBenyahia5}(a) and 
$S_{1{\rm in}}=13$. 
(a): The $X_1$-components and 
(b): the $X_2$-components, of the steady states $E_1^0$ (in Black), $E_2^0$ (in Green), $E_2^1$ (in Red) and $E_2^2$ (in Blue). 
 (c): A magnification showing the bifurcation values $D_4$, $D_5$, $D_6$ and $D_7$. 
Solid lines and dotted lines correspond to stable and unstable steady states respectively.  \label{figBifSin13}}
\end{figure}
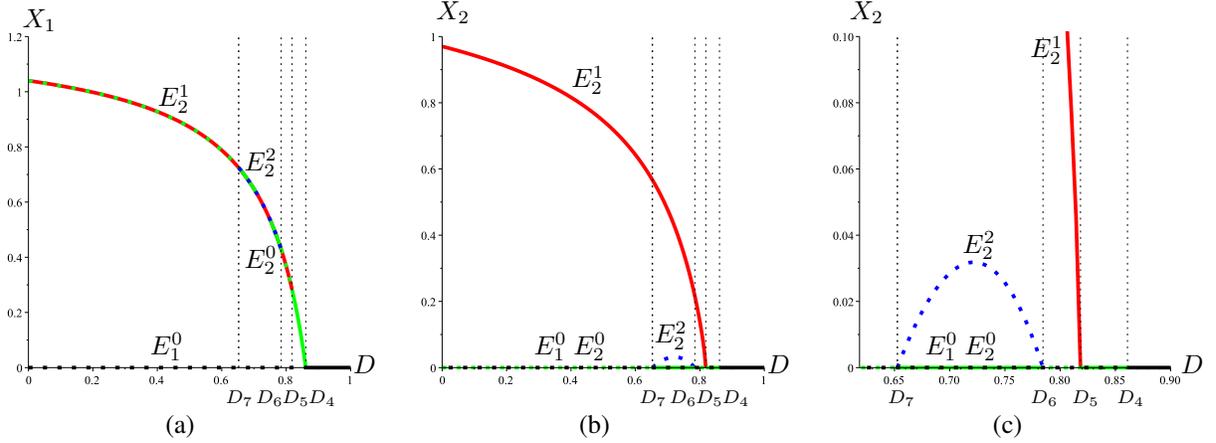
%=======================================

Consider now the case where $S_{1{\rm in}}=13$. This case corresponds to the surprising situation where we can go from the bistability region 
(colored in Pink) 
to the global asymptotic stability region (colored in Green), when the dilution rate $D$ increases. 
Since $12.57<S_{1{\rm in}}<13.11$, as it is seen in Fig. \ref{figBifS1inBenyahia5S2in0}, with increasing $D$, there is a transition from 
$\mathcal{I}_4$ to $\mathcal{I}_5$ for $D=D_7\approx 0.6526$, 
then from $\mathcal{I}_5$ to 
$\mathcal{I}_4$ for $D=D_6\simeq 0.7844$, 
then from $\mathcal{I}_4$ to 
$\mathcal{I}_3$ for 
$D=D_5\simeq 0.8184$, then from $\mathcal{I}_3$ to 
$\mathcal{I}_0$ for 
$D=D_4\simeq 0.8609$. 
The bifurcation values $D_4$, $D_5$, $D_6$ and $D_7$ are defined by
\begin{equation}
\label{Dk2}
D_4=\frac{\mu_1\left(S_{1{\rm in}}\right)}{\alpha},
\quad
D_5
\mbox{ is the solution of }
S_{1{\rm in}}=\frac{k_1}{k_2}H_1(D),
\quad 
D_6, D_7
\mbox{ are the solutions of }
S_{1{\rm in}}=\frac{k_1}{k_2}H_2(D)
\end{equation}
The bifurcation value 
$D_4$ corresponds to a transcritical bifurcation of $E_2^0$ and $E_1^0$; 
$D_5$ corresponds to a transcritical bifurcation of $E_2^1$ and $E_2^0$ and $D_6$ and $D_7$ correspond to transcritical bifurcations of $E_2^0$ and $E_2^2$. The plot of $X_1$ and $X_2$ components of all existing steady states with respect of $D$ is shown in Fig.~\ref{figBifSin13}. Solid lines and dotted lines correspond to stable and unstable steady states respectively.
On Fig.~\ref{figBifSin13}(a), 
for $0<D<D_7$ and $D_6<D<D_5$ the $X_1$-component of $E_2^i$, $i=0,2$, is colored in Red, with Green dots, showing the stability of $E_2^1$ and the instability of $E_2^0$.
 For $D_7<D<D_6$, the $X_1$-component of $E_2^i$, $i=0,1,2$, is colored in Red and Green, with Blue dots, showing the bistability of $E_2^0$ and $E_2^1$ and the instability of $E_2^1$.
 For $D_5<D<D_4$, the $X_1$-component of $E_2^0$ is colored in Green, showing the stability of $E_2^0$. 
On Fig.~\ref{figBifSin13}(b) 
and \ref{figBifSin13}(c), for $0<D<D_7$ and $D_6<D<D_5$, the $X_2=0$-component of $E_j^0$, $j=1,2$, is colored with Green and Black dots, showing the instability of $E_2^0$ and $E_1^0$; For $D_7<D<D_6$ 
 and $D_5<D<D_4$ it is colored in Green, with Black dots, showing the stability of $E_2^0$ and the instability of $E_1^0$. For $D>D_4$ it is colored in Black showing the stability of $E_1^0$.
%======================================
%======================================
\begin{figure}[ht]
\setlength{\unitlength}{1.0cm}
\begin{center}
\begin{picture}(8.5,6.5)(0,-0.2)		
\put(-4,0){{\includegraphics[scale=0.25]{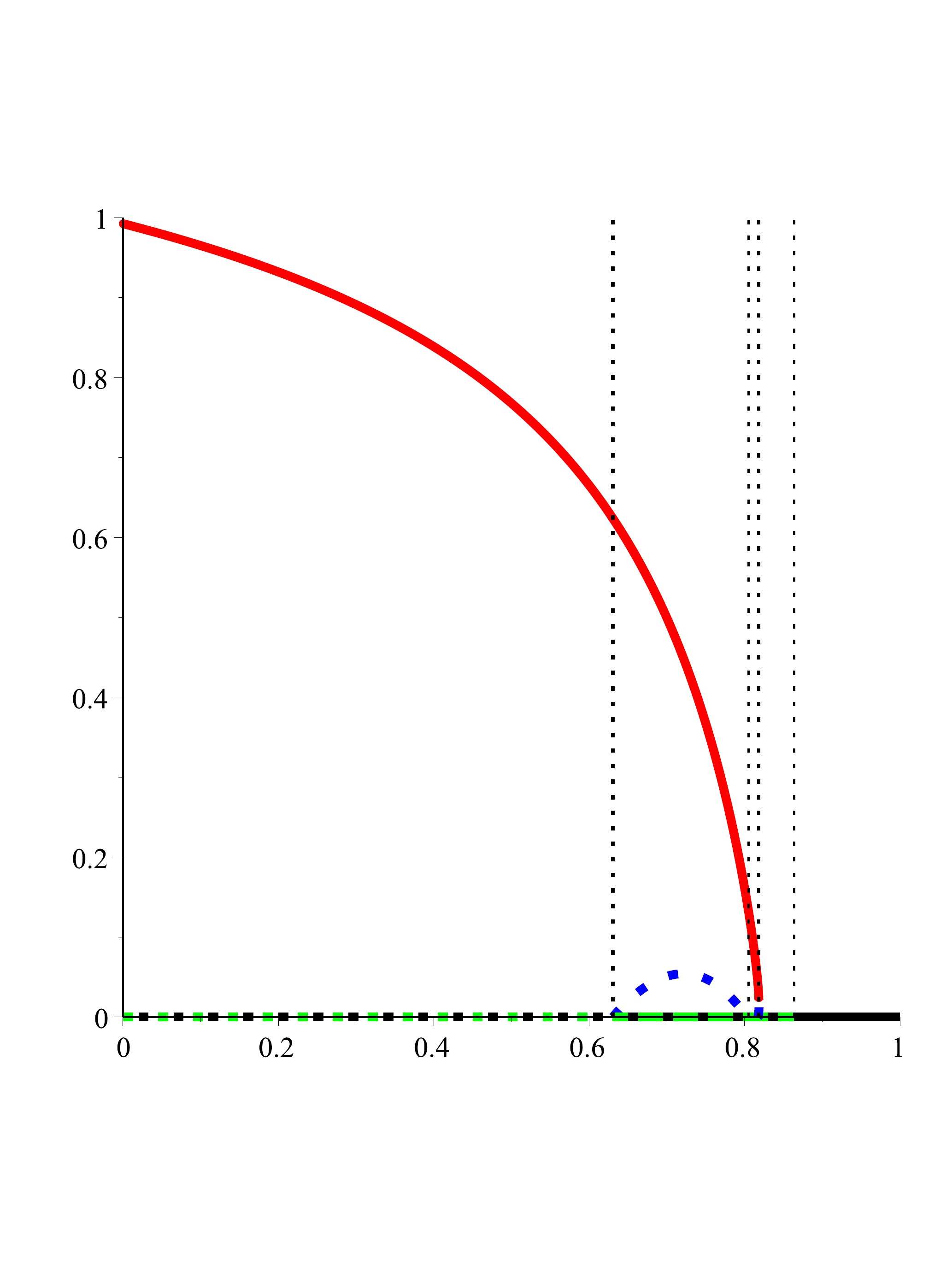}}}
\put(1.5,0){{\includegraphics[scale=0.25]{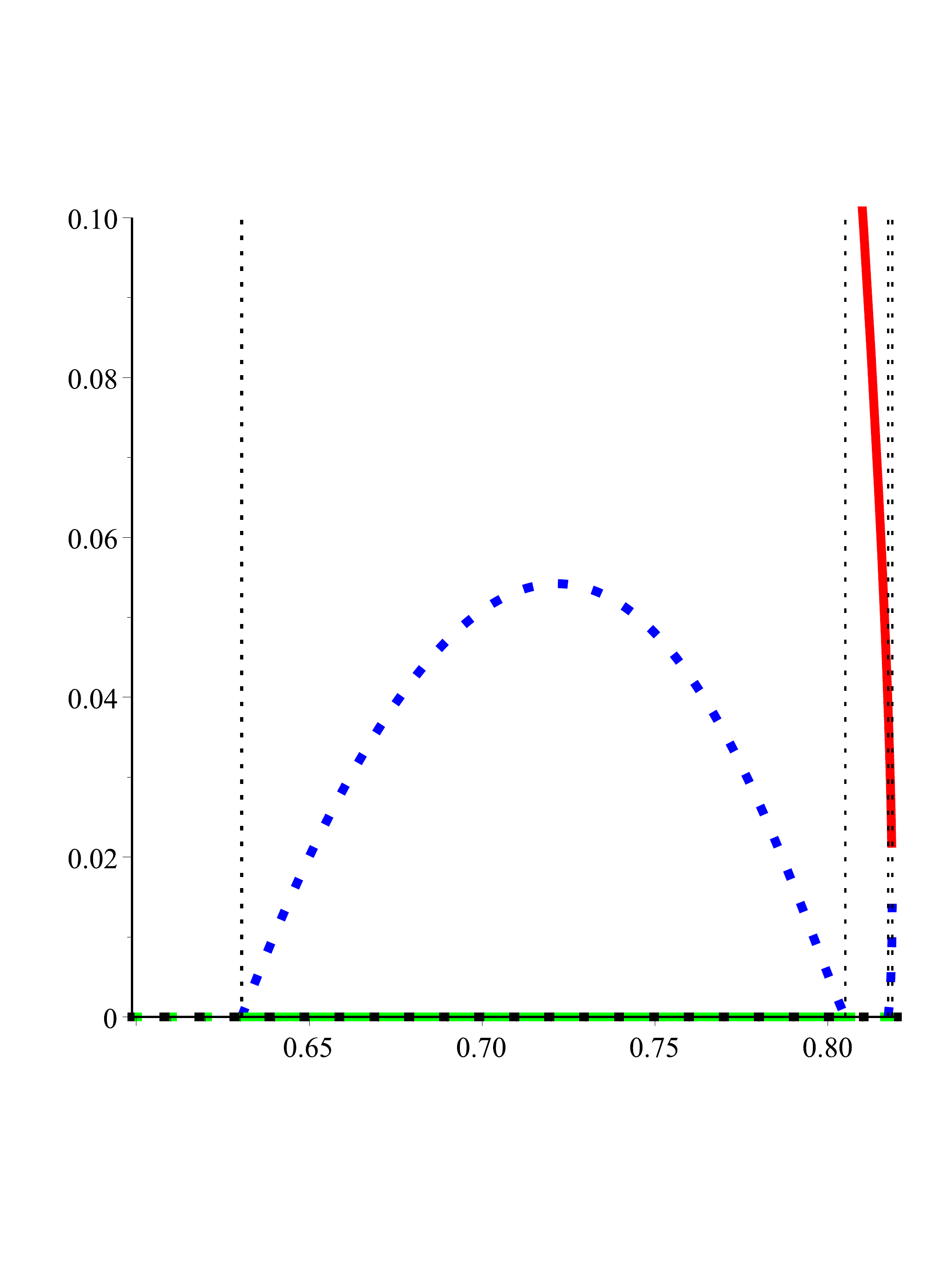}}}
\put(7,0){{\includegraphics[scale=0.25]{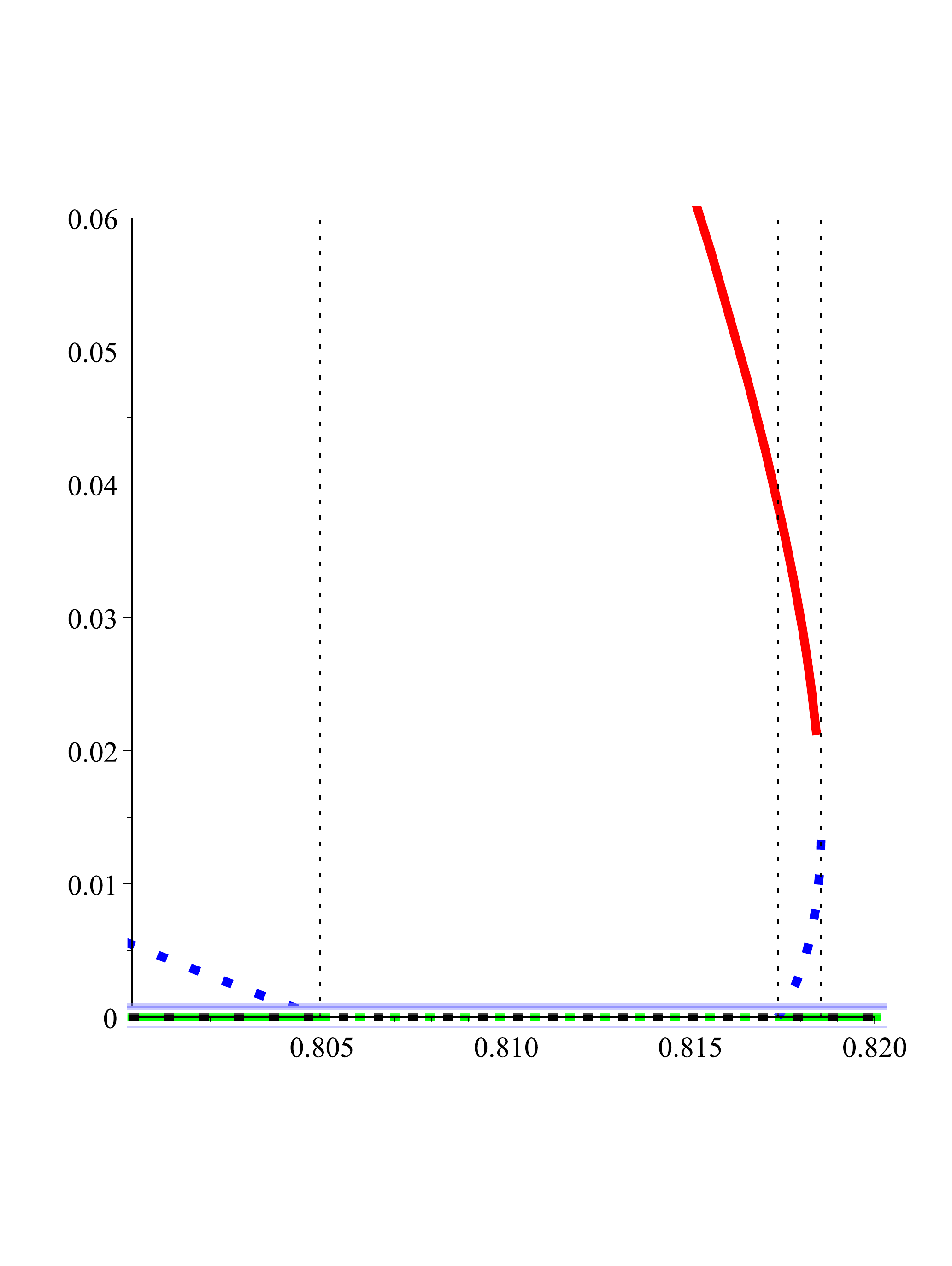}}}
%==============================================
\put(-1.5,0.5){(a)}
\put(4,0.5){(b)}	
\put(9.5,0.5){(c)}	
%%%%%%%%%%%%%%%%%%
\put(-1,0.9){\scriptsize$D_{10}$}
\put(0.1,0.9){$\uparrow$}
\put(-0.3,0.6){\scriptsize$D_9D_8D_2$}
\put(0.3,1.1){\scriptsize$D_4$}
\put(2.7,0.9){\scriptsize$D_{10}$}
\put(5.9,0.9){\scriptsize$D_9$}
\put(6.3,0.9){$\uparrow$}
\put(6.05,0.6){\scriptsize$D_8D_2$}
\put(8.6,0.9){\scriptsize$D_9$}
\put(11.05,0.9){\scriptsize$D_8$}
\put(11.45,0.9){\scriptsize$D_2$}
\put(1,1.3){$D$}
\put(6.5,1.3){$D$}	
\put(12,1.3){$D$}	
\put(-3.4,5.9){$X_2$}
\put(2.1,6){$X_2$}	
\put(7.6,6){$X_2$}
%%%%%%%%%%%%%%%%%%%%%
\put(-2.2,1.55){{$E_1^0$~$E_2^0$}}
\put(-0.5,1.8){{$E_2^2$}}
\put(-1.6,5.2){{$E_2^1$}}
%%%%%%%%%%%%%%%%%%%%%%%%%%%%%%
\put(4,1.55){{$E_1^0$~$E_2^0$}}
\put(4.4,3.9){{$E_2^2$}}
\put(5.7,5.5){{$E_2^1$}}
%%%%%%%%%%%%%%%%%%%%%%%%
\put(8,1.8){{$E_2^2$}}
\put(11.6,1.8){{$E_2^2$}}
\put(10.4,5.5){{$E_2^1$}}
\put(9.5,1.55){{$E_1^0$ $E_2^0$}}
\end{picture}
\end{center}
\vspace{-0.9cm}
\caption{Bifurcation diagram with $D$ as the bifurcation parameter, corresponding to Fig. \ref{DS1inBenyahia5}(a) and 
$S_{1{\rm in}}=14$. 
(a): The $X_2$-components of the steady states $E_1^0$ (in Black), $E_2^0$ (in Green), $E_2^1$ (in Red) and $E_2^2$ (in Blue). 
(b): A magnification showing the bifurcation values 
$D_9$ and $D_{10}$, where $D_2$ and $D_{8}$ are indistinguishable.
 (c): A larger magnification showing the bifurcation values 
$D_2$,  
$D_8$ and $D_9$. 
Solid lines and dotted lines correspond to stable and unstable steady states respectively.  \label{figBifSin13.3}}
\end{figure}
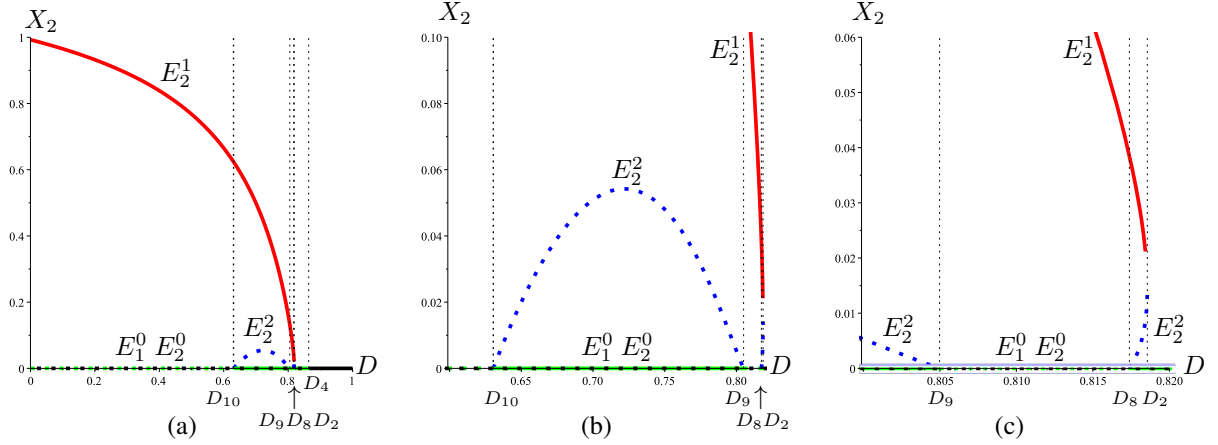
%=======================================

Consider now the case where $S_{1{\rm in}}=13.3$. This case corresponds also to the situation where we can go from the bistability region 
(colored in Pink) 
to the global asymptotic stability region (colored in Green), when the dilution rate $D$ increases. 
Since $13.11<S_{1{\rm in}}<13.37$, as it is seen in Fig. \ref{figBifS1inBenyahia5S2in0}, with increasing $D$, there is a transition from 
$\mathcal{I}_4$ to $\mathcal{I}_5$ for 
$D=D_{10}\approx 0.6304$, 
then from $\mathcal{I}_5$ to 
$\mathcal{I}_4$ for 
$D=D_9\simeq 0.8050$, 
then from $\mathcal{I}_4$ to 
$\mathcal{I}_5$ for 
$D=D_8\simeq 0.8173$, then from $\mathcal{I}_5$ to 
$\mathcal{I}_3$ for 
$D=D_2\simeq 0.8186$, then from $\mathcal{I}_3$ to 
$\mathcal{I}_0$ for 
$D=D_4\simeq 0.8636$.
The bifurcation values 
$D_2$, $D_4$, $D_8$, $D_9$ and $D_{10}$ are defined by
\begin{equation}
\label{Dk3}
D_4=\frac{\mu_1\left(S_{1{\rm in}}\right)}{\alpha},
\quad
D_2=\frac{\mu_2\left(S_2^M\right)}{\alpha}
\quad \mbox{ and }
D_8, D_9, D_{10}
\mbox{ are the solutions of }
S_{1{\rm in}}=\frac{k_1}{k_2}H_2(D)
\end{equation}
The bifurcation value 
$D_4$ corresponds to a transcritical bifurcation of $E_2^0$ and $E_1^0$; 
$D_2$ corresponds to a saddle node bifurcation of  $E_2^1$ and $E_2^2$ and $D_8$, $D_9$ and $D_{10}$ correspond to transcritical bifurcations of $E_2^0$ and $E_2^2$. The plot of the $X_2$ component of all existing steady states with respect of $D$ is shown in Fig.~\ref{figBifSin13.3}. Solid lines and dotted lines correspond to stable and unstable steady states respectively. 
Since two magnifications are necessary to represent all bifurcations, the plot of the $X_1$ component is omitted in Fig.~\ref{figBifSin13.3}. However, it is similar to those plots given in Figs.~\ref{figBifSin14}(a) and 
\ref{figBifSin13}(a).
On Fig.~\ref{figBifSin13.3} 
for $0<D<D_{10}$ and $D_9<D<D_8$, the $X_2=0$-component of $E_j^0$, $j=1,2$, is colored with Green and Black dots, showing the instability of $E_2^0$ and $E_1^0$; 
For $D_{10}<D<D_9$ 
and $D_8<D<D_4$ it is colored in Green, with Black dots, showing the stability of $E_2^0$ and the instability of $E_1^0$. For $D>D_4$ it is colored in Black showing the stability of $E_1^0$. Notice that for $D_{10}<D<D_9$ 
and $D_8<D<D_2$ both steady states $E_2^0$ and $E_2^1$ are stable.
%======================================

%=======================================
\section{Discussion}\label{sec5}

Our main contribution is to present the operating diagram and to show how it depends on the biological parameters.

The parameter space of model (\ref{AM2}), where $\mu_1$ and $\mu_2$ are given by (\ref{MonodHaldane}) is twelve dimensional: nine biological or physical parameters ($m_1$, $m_2$, $K_1$, $K_2$,$K_I$, $k_1$, $k_2$, $k_2$ and $\alpha$) and three operating parameters ($D$, $S_{1{\rm in}}$ and $S_{1{\rm in}}$).
The former parameters are called biological
parameters since they depend on the organisms, and substrate considered.
These parameters are measurable in the laboratory. In contrast, the later parameters are 
called operating parameters since they are under the control of the experimenter.

Exploring all of the twelve dimensional parameter space is almost possible. 
Fixing the biological parameters and constructing the operating diagram is a powerful answer for the discussion of the
behavior of the model with respect of the parameters. 
Therefore our approach to handle the question of the dependence with respect of the parameters of the model is to split the question in
two intermediary questions. 
First we fix the biological parameters and
present the operating diagram. Second we explore how the operating
diagram varies when the biological parameters are changed. For instance, Figs. \ref{DS1inBenyahia6}, Fig. \ref{DS1inBenyahia5} and Fig. \ref{DS1inBenyahia4} show how the operating diagram changes when the biological parameter $m_1$ is changed.

The operating
diagrams shown in the figures summarize the effect of the operating conditions on the long-term dynamics of the AM2 model and shows six type of behavior: 1) the washout of the two populations (regions colored in Red), 2) the washout of the first population while the second population is maintained (regions colored in Blue), 2) the occurrence of these two behaviors, according to initial conditions (regions colored in Cyan), 4) the washout of the second population while the first is maintained (regions colored in Yellow), 5) the persistence of both populations (regions colored in Green) and finally 6) the occurrence of these two behaviors according to initial conditions (regions colored in Pink).

In the operating diagrams shown in Figs.
\ref{DS1inBenyahia6}(a), \ref{DS1inBenyahia5}(a) and \ref{DS1inBenyahia4}(a), obtained for $S_{2{\rm in}}=0$, only regions $\mathcal{I}_0$, 
$\mathcal{I}_3$, $\mathcal{I}_4$ and 
$\mathcal{I}_5$ exist, that is to say, the steady states $E_1^i$, $i=1,2$ without acidogenic bacteria, cannot exist. 
This property is in accordance with the fact that the system being commensalistic, and without input concentration $S_{2{\rm in}}$, it is impossible for the commensal population (the methanogenic bacteria) to survive if the host population (the acidogenic bacteria) is washed out.

The operating diagram 
shows how robust or how extensive is the parameter region where
coexistence occurs, where the corresponding steady state is GAS, where the steady states, with extinction both or one of the population, is stable and
where it is unstable.

%\section{Conclusion}

\appendix
\section{Proofs}
\subsection{Proof of Proposition \ref{proSSi}}\label{ProofproSSi}

%============================
\begin{table}[ht]
\caption{Necessary and sufficient conditions of existence and local stability of the steady states of (\ref{AM2}) obtained in \cite{Benyahia}. $S_1^{*}(D)$, $S_2^{i*}(D)$ and ${S}^*_{2{\rm in}}\left(D,S_{1{\rm in}},S_{2{\rm in}}\right)$ are defined in Table \ref{tableFunctions}.} \label{TableSumExisStabBenyahia}
\begin{center}	
\begin{tabular}		{lll}
\hline
                        &  Existence conditions                          &    Stability conditions
                        \\\hline
%=====================================================================
$E_1^0$
&
Always exists
&
$S_{1{\rm in}}<S_1^*(D)$ and $S_{2{\rm in}}\notin\left[S_2^{1*}(D),S_2^{2*}(D)\right]$\\
$E_1^1$
&
$S_{2{\rm in}}>S_2^{1*}(D)$&
$S_{1{\rm in}}<S_1^*(D)$\\
$E_1^2$
&
$S_{2{\rm in}}>S_2^{2*}(D)$
&
Unstable if it exists\\
$E_2^0$
&
$S_{1{\rm in}}>S_1^*(D)$
&
${S}^*_{2{\rm in}}\left(D,S_{1{\rm in}},S_{2{\rm in}}\right)
\notin\left[S_2^{1*}(D),S_2^{2*}(D)\right]$\\
$E_2^1$
&
$S_{1{\rm in}}>S_1^*(D)$ and 
${S}^*_{2{\rm in}}\left(D,S_{1{\rm in}},S_{2{\rm in}}\right)>S_2^{1*}(D)$
&
Stable if it exists\\
$E_2^2$
&
$S_{1{\rm in}}>S_1^*(D)$ and 
${S}^*_{2{\rm in}}\left(D,S_{1{\rm in}},S_{2{\rm in}}\right)>S_2^{2*}(D)$
& 
Unstable if it exists\\
\hline
\end{tabular}
	\end{center}
\end{table}
%=====================================================================

The proof follows from \cite{Benyahia}.
It is seen from Proposition 1 of \cite{Benyahia} that the steady states are given by Table \ref{tableSSi}, where $S_1^{*}$, $S_2^{i*}$, 
$S_{2{\rm in}}^{*}$,
$X_1^{*}$, $X_2^{i}$ and $X_2^{i*}$ are defined in Table \ref{tableFunctions}. 
The necessary and sufficient conditions of existence of the steady state given in Proposition 1 of \cite{Benyahia} are summarized in the second column of Table \ref{TableSumExisStabBenyahia}. The necessary and sufficient conditions of local stability of these steady states, obtained in Table A.1 of \cite{Benyahia}, are summarized in the third column of  Table \ref{TableSumExisStabBenyahia}. 

Let us prove the following result which shows that the existence conditions of $X_2^{i*}$, $i=1,2$ steady states  given in Table \ref{TableSumExisStabBenyahia}, can be stated using the functions  $H_i(D)$, $i=1,2$, defined in Table \ref{tableFunctions}. These functions were considered  also by \cite{Sbarciog}.

\begin{lem}\label{lem1}
The conditions 
${S}^*_{2{\rm in}}\left(D,S_{1{\rm in}},S_{2{\rm in}}\right)=S_2^{i*}(D)$
 and 
${S}^*_{2{\rm in}}\left(D,S_{1{\rm in}},S_{2{\rm in}}\right)<S_2^{i*}(D)$,
for $i=1,2$, 
are equivalent to the conditions
$S_{2{\rm in}}+\frac{k_2}{k_1}S_{1{\rm in}}=H_{i}(D)$ 
and 
$S_{2{\rm in}}+\frac{k_2}{k_1}S_{1{\rm in}}<H_{i}(D)$, for $i=1,2$,
respectively.
\end{lem}

\begin{proof}
The result follows from the definitions of
${S}^*_{2{\rm in}}\left(D,S_{1{\rm in}},S_{2{\rm in}}\right)$ and $H_i(D)$, given in  
Table \ref{tableFunctions}. Indeed
${S}^*_{2{\rm in}}\left(D,S_{1{\rm in}},S_{2{\rm in}}\right)=S_2^{i*}(D)$
is equivalent to
$$S_{2{\rm in}}+\frac{k_2}{k_1}\left(S_{1{\rm in}}-S_1^*(D)\right)=S_2^{i*}(D)
\Longleftrightarrow
S_{2{\rm in}}+\frac{k_2}{k_1}S_{1{\rm in}}=
S_2^{i*}(D)+\frac{k_2}{k_1}S_1^*(D).
$$
That is to say
$S_{2{\rm in}}+\frac{k_2}{k_1}S_{1{\rm in}}=
H_i(D)$. The proof for the inequality is the same.
\end{proof}

Therefore, the results in Table \ref{TableSumExisStabBenyahia} are equivalent to those in Table \ref{TableSumExisStab} which completes the proof of Proposition \ref{proSSi}.

\subsection{Proof of Proposition \ref{proIk}}\label{ProofproIk}
%%%%%%%%%%%%%%%%%%%%%%%

\begin{table}[ht]
\caption{The 9 cases of existence and stability of steady states of (\ref{AM2}) obtianed in \cite{Benyahia},  where S and U stand for {\em stable} and {\em unstable} respectively.
}\label{Table9casesBenyahia}
\begin{center}
\begin{tabular}{c|c|c|cccccc}
\hline
  Condition 1
  & Condition 2
  & {Case}
  & $E_{1}^{0}$
  &$E_{1}^{1}$ & $E_{1}^{2}$ & $E_{2}^{0}$ & $E_{2}^{1}$ & $E_{2}^{2}$
  \\ 
  \hline
  \multirow{3}{*}{%\rotatebox{90}
  {$S_{1{\rm in}}\!<\!S_1^*(D)$} }
  & $S_{2{\rm in}}<S_2^{1*}(D)$ 
   & {\bf 1.1} 
  & S& &\\
&$S_2^{1*}(D)\!<\!S_{2{\rm in}}\!\leq\!S_2^{2*}(D)$
&{\bf 1.2}
& U& S&\\
&$S_2^{2*}(D)<S_{2{\rm in}}$
&{\bf 1.3}
&S&S&U 
\\ \hline
  \multirow{6}{*}{%\rotatebox{90}
  {$S_{1{\rm in}}\!>\!S_1^*(D)$} } 
  &
 ${S}_{2{\rm in}}<{S}^*_{2{\rm in}}<S_2^{1*}<S_2^{2*}$
 & {\bf 2.1} 
  & U & & & S & & \\
&
$S_{2{\rm in}}\leq S_2^{1*}<{S}^*_{2{\rm in}}\leq S_2^{2*}$
&{\bf 2.2}
& U & & & U & S & \\
&
$S_{2{\rm in}}\leq S_2^{1*}<S_2^{2*}<{S}^*_{2{\rm in}}$ 
&{\bf 2.3}
& U & & & S & S & U \\
&
$S_2^{1*}<S_{2{\rm in}}<{S}^*_{2{\rm in}}\leq S_2^{2*}$
&{\bf 2.4}
& U & U && U & S & \\
&
$S_2^{1*}<S_{2{\rm in}}\leq S_2^{2*}
<{S}^*_{2{\rm in}}$ 
&{\bf 2.5}
& U & U & & S & S & U \\
&
$S_2^{1*}<S_2^{2*}<S_{2{\rm in}}<{S}^*_{2{\rm in}}$ 
&{\bf 2.6}
& U & U & U & S & S & U   
  \\
 \hline
\end{tabular}  
\end{center}
  \end{table}

The proof follows from \cite{Benyahia}.
The existence and stability conditions of the steady states of (\ref{AM2}) given in 
Table \ref{TableSumExisStabBenyahia}
depend only on the relative positions of the values of $S_{1in}$ and $S_1^*(D)$ and of the values of 
$S_2^{1*}(D)$, $S_2^{2*}(D)$, $S_{2in}$, 
and 
${S}^*_{2in}\left(D,S_{1in},S_{2in}\right)$.
Actually, as stated in Theorem 1 of \cite{Benyahia}, we can distinguish nine cases, according to the relative positions of these numbers. These cases are summarized in Table \ref{Table9casesBenyahia}. 

The cases {\bf 1.1}, {\bf 1.2} and {\bf 1.3} correspond to the regions 
$\mathcal{I}_0$, $\mathcal{I}_1$ and 
$\mathcal{I}_2$ respectively, defined in Table \ref{Table9Regions}. Now we use Lemma \ref{lem1} to show that the remaining six cases
{\bf 2.1} to {\bf 2.6} correspond to the six regions $\mathcal{I}_3$ to $\mathcal{I}_8$ defined in Table \ref{Table9Regions}.

Since 
${S}_{2{\rm in}}<{S}^*_{2{\rm in}}$ the case
{\bf 2.1} corresponds to the condition  
${S}^*_{2{\rm in}}<S_2^{1*}$ which is equivalent, using Lemma \ref{lem1}, to
$S_{2{\rm in}}+\frac{k_2}{k_1}S_{1{\rm in}}<H_1(D)$. Therefore the case {\bf 2.1} corresponds to the region  
$\mathcal{I}_3$ defined in Table \ref{Table9Regions}. Using again Lemma \ref{lem1}, the condition 
$S_2^{1*}<{S}^*_{2{\rm in}}<S_2^{2*}$ in the case {\bf 2.2} is equivalent to 
$H_1(D)<S_{2{\rm in}}+\frac{k_2}{k_1}S_{1{\rm in}}<H_2(D)$ and
the condition 
${S}^*_{2{\rm in}}>S_2^{2*}$ in the case {\bf 2.3} is equivalent to 
$S_{2{\rm in}}+\frac{k_2}{k_1}S_{1{\rm in}}>H_2(D)$.
Therefore the cases {\bf 2.2} and {\bf 2.3} correspond to the regions  
$\mathcal{I}_4$ and $\mathcal{I}_5$ respectively, defined in Table \ref{Table9Regions}. Using similar arguments we show that the cases {\bf 2.4}, {\bf 2.5} and {\bf 2.6} correspond to the regions 
$\mathcal{I}_6$, $\mathcal{I}_7$ and 
$\mathcal{I}_8$ respectively, defined in Table \ref{Table9Regions}. 

Excepted for cases  {\bf 1.3},  {\bf 2.3},  {\bf 2.5} and  {\bf 2.6} of bistability, the system (\ref{AM2}) has a unique globally asymptotically stable (GAS) steady state. Therefore, in the case  {\bf 1.1}, $E_1^0$ is GAS; in the case  {\bf 1.2}, $E_1^1$ is GAS, in the case  {\bf 2.1}, $E_2^0$ is GAS, and in the cases  {\bf 2.2} and  {\bf 2.4}, $E_2^1$ is GAS.  In the case {\bf 1.3}, $E_1^2$ is a saddle point whose attractive manifold is a 3-dimensional hyper-surface surface which separates the phase space of
(\ref{AM2}) into the basins of attractions of the stable steady states $E_1^0$ and $E_1^1$. 
In the cases {\bf 2.3}, {\bf 2.5} and {\bf 2.6}, $E_2^2$ is a saddle point whose stable manifold is a 3-dimensional hyper-surface which separates the phase space of
(\ref{AM2}) into the basins of attractions of the stable steady states $E_2^0$ and $E_2^1$. For details and complements on the global behaviour, see section 2.4 of \cite{Benyahia}.
This completes the proof of Proposition \ref{proIk}.

\subsection{Proof of Proposition \ref{proBifurcations}}\label{ProofproBifurcations}

Part of the proof follows from \cite{Benyahia}.
It is seen from Theorem 1 of \cite{Benyahia} that {\em non hyperbolic} steady states, that correspond to coalescence of some of the steady state, occur when two (or more) of the values of  
$S_2^{1*}(D)$, $S_2^{2*}(D)$, $S_{2in}$, and 
${S}^*_{2in}\left(D,S_{1in},S_{2in}\right)$ are equal. Notice that the condition 
$S_2^{1*}(D)=S_2^{2*}(D)$, arising in cases {\bf 1.6}, {\bf 2.11} and {\bf 2.14} of Theorem 1 of \cite{Benyahia}, corresponds of the saddle node bifurcations of $E_1^1=E_1^2$ or $E_2^1=E_2^2$. 
This condition holds on $\Gamma_6$,

Notice the condition
$S_{2{\rm in}}=S_2^{1*}(D)$, arising in cases 
{\bf 1.4}, {\bf 2.8} and {\bf 2.9} of Theorem 1 of \cite{Benyahia}, corresponds of the transcritical bifurcation $E_1^0=E_1^1$. 
This condition holds on $\Gamma_2$.
Similarly, the condition
$S_{2{\rm in}}=S_2^{2*}(D)$, arising in cases 
{\bf 1.5} and {\bf 2.13} of Theorem 1 of \cite{Benyahia}, corresponds of the transcritical bifurcation $E_1^0=E_1^2$. 
This condition holds on $\Gamma_3$.

On the other hand the condition
$S^*_{2{\rm in}}=S_2^{1*}(D)$, arising in cases 
{\bf 2.7} of Theorem 1 of \cite{Benyahia}, corresponds of the transcritical bifurcation $E_2^0=E_2^1$. 
Using Lemma \ref{lem1},  
this condition holds on $\Gamma_4$.
Similarly, the condition
$S^*_{2{\rm in}}=S_2^{2*}(D)$, arising in cases 
{\bf 2.12} and {\bf 2.15} of Theorem 1 of \cite{Benyahia}, corresponds of the transcritical bifurcation $E_2^0=E_2^2$. 
Using Lemma \ref{lem1}, 
this condition holds on $\Gamma_5$.

Finally we consider the bifurcations occuring when
$S_{1{\rm in}}=S_1^{*}(D)$. These bifurcations were not considered in Theorem 1 of \cite{Benyahia}. The condition $S_{1{\rm in}}=S_1^{*}(D)$ holds on $\Gamma_1$ and corresponds to the transcritical bifurcations $E_1^0=E_2^0$, 
$E_1^1=E_2^1$ and $E_1^2=E_2^2$. This completes the proof of Proposition \ref{proBifurcations}.
%%%%%%%%%%%%%%%%%%%%%%%%%%%%

\section{Tables}

\begin{table}[ht]
\caption{Nominal parameters values used in  \cite{Benyahia} and corresponding to the figures.
}\label{parametervalues}
\begin{center}
\begin{tabular}{l|ccccccccc}
\hline
 Parameter
  & $m_1$
  & $K_1$
  & $m_2$
  &$K_2$ 
  & $K_I$ 
  & $\alpha$ 
  & $k_1$ 
  & $k_2$
  & $k_3$
  \\ 
  %\hline
  Unit
  & ${\rm d}^{-1}$
  & g/L
  & ${\rm d}^{-1}$
  & mmol/L
  & mmol/L
  &
  &
  & mmol/g
  & mmol/g
  \\
  \hline
  Case (A): Figs. \ref{figHi}(a),
  \ref{figOD3D},
  \ref{S1inS2inBenyahia6},
  \ref{DS1inBenyahia6},
  \ref{figBifS1inBenyahia5S2in0},
  \ref{figBifSin14},
  \ref{figBifSin13},
  \ref{figBifSin13.3}
  &
  $0.6$
  &
  \multirow{3}{*}{2.1}
  &
  \multirow{3}{*}{0.95}
  &
  \multirow{3}{*}{24}
  &
  \multirow{3}{*}{55}
  &
  \multirow{3}{*}{0.5}
  &
  \multirow{3}{*}{25}
  &
  \multirow{3}{*}{250}
  &
  \multirow{3}{*}{268}
  \\
  %\hline
  Case (B): Figs. \ref{figHi}(b),
  \ref{S1inS2inBenyahia5},
  \ref{DS1inBenyahia5}
  &
  $0.5$
  &
  \\
  %\hline
  Case (C): Figs. \ref{figHi}(c),
  \ref{S1inS2inBenyahia4},
  \ref{DS1inBenyahia4}
  &
  $0.4$
  & 
 
  \\  
 \hline
\end{tabular}  
\end{center}
  \end{table}

%%%%%%%%%%%%%%%%%%%%%%%%%%%%%
\begin{table}[ht]
\caption{Intersections of the 
$\Gamma_k$ surfaces, 
$k=0\cdots8$ with a
$\left(S_{1{\rm in}},S_{2{\rm in}}\right)$ plane, where $D$ is kept constant.}\label{IntersectionD}
\begin{center}	
\begin{tabular}{c|l}
$\Gamma_k$
&
\qquad $\Gamma_k\cap\left\{D=\mbox{constant}\right\}$
\\
\hline
$\Gamma_1$&
\begin{tabular}{lcl}
Vertical line  
$S_{1{\rm in}}=S_1^*(D)$
&if $D< D_1$\\
Empty 
&if $D\geq D_1$
\end{tabular}
\\
\hline
$\Gamma_2$&
\begin{tabular}{lcl}
Horizontal line  
$S_{2{\rm in}}=S_2^{1*}(D)$
&if $D\leq D_2$\\
Empty 
&if $D> D_2$
\end{tabular}
\\
\hline
$\Gamma_3$&
\begin{tabular}{lcl}
Horizontal line 
$S_{2{\rm in}}=S_2^{2*}(D)$
&if $D\leq D_2$\\
Empty 
&if $D> D_2$
\end{tabular}
\\
\hline
$\Gamma_4$&
\begin{tabular}{lcl}
Oblique line 
$S_{2{\rm in}}+\frac{k_2}{k_1}S_{1{\rm in}}=H_1(D)$
&if $D<\min(D_1,D_2)$\\
Empty 
&if $D\geq \min(D_1,D_2)$
\end{tabular}
\\
\hline
$\Gamma_5$&
\begin{tabular}{lcl}
Oblique line 
$S_{2{\rm in}}+\frac{k_2}{k_1}S_{1{\rm in}}=H_2(D)$
&if $D<\min(D_1,D_2)$\\
Empty 
&if $D\geq \min(D_1,D_2)$
\end{tabular}
\\
\hline
$\Gamma_6$&
\begin{tabular}{lcl}
The whole plane
&if $D=D_2$\\
Empty 
&if $D\neq D_2$
\end{tabular}
\\
%\hline
\end{tabular}
\end{center}
\end{table}
%%%%%%%%%%%%%%%%%%%%%%%%%%%%%% 

%%%%%%%%%%%%%%%%%%%%%%%%%%%%%%%%%%%%%%%%%%%%%
\begin{table}[ht]
\caption{The intersections of the 
$\Gamma_k$ surfaces, 
$k=0\cdots8$ with a
$\left(D,S_{1{\rm in}}\right)$ plane, where $S_{2{\rm in}}$ is kept constant.}\label{IntersectionS2in}
\begin{center}	
\begin{tabular}{c|ll}
$\Gamma_k$
&
\qquad
$\Gamma_k\cap\left\{S_{2{\rm in}}=\mbox{constant}\right\}$
\\
\hline
$\Gamma_1$
&
\begin{tabular}{l}
Curve of function  
$S_{1{\rm in}}=S_1^*(D)$
\end{tabular}
\\
\hline
$\Gamma_2$&
\begin{tabular}{lcl}
Vertical line 
$D=\frac{1}{\alpha}\mu_2\left(S_{2{\rm in}}\right)$
&if $S_{2{\rm in}}\leq S_2^M$\\
Empty 
&if $S_{2{\rm in}}> S_2^M$
\end{tabular}
\\
\hline
$\Gamma_3$&
\begin{tabular}{lcl}
Vertical line 
$D=\frac{1}{\alpha}\mu_2\left(S_{2{\rm in}}\right)$
&if $S_{2{\rm in}}\geq S_2^M$\\
Empty 
&if $S_{2{\rm in}}< S_2^M$
\end{tabular}
\\
\hline
$\Gamma_4$&
\begin{tabular}{l}
Curve of function 
$S_{1{\rm in}}=
\frac{k_1}{k_2}\left(H_1(D)-S_{2{\rm in}}\right)$
restricted to the domain $S_{1{\rm in}}>S_1^*(D)$
\end{tabular}
\\
\hline
$\Gamma_5$&
\begin{tabular}{l}
Curve of function 
$S_{1{\rm in}}=
\frac{k_1}{k_2}\left(H_2(D)-S_{2{\rm in}}\right)$
restricted to the domain $S_{1{\rm in}}>S_1^*(D)$
\end{tabular}
\\
\hline
$\Gamma_6$&
\begin{tabular}{l}
Vertical line
$D=D_2$
\end{tabular}
\\
%\hline
\end{tabular}
\end{center}
\end{table}
%%%%%%%%%%%%%%%%%%%%%%

%%%%%%%%%%%%%%%%%%%%%%%%%%%%%%%%%%%%%%%%%%%%%
\begin{table}[ht]
\caption{Auxiliary function in the case given by  (\ref{MonodHaldane}).}\label{AuxiliaryFunctionsMonodHaldane}
\begin{center}	
\begin{tabular}{l}
\hline
$\mu_1\left(S_1\right)=\displaystyle\frac{m_{1}S_1}{K_1+S_1}$
\\[2mm]
$\mu_1(+\infty)=m_1$
\\[2mm]
$S_1^*(D)=\displaystyle
\frac{\alpha DK_1}{m_1-\alpha D}$
\\[2mm]
$S_1^*(D)$ is defined for $0<D<D_1$, where 
$D_1=\displaystyle\frac{m_1}{\alpha}$
\\[2mm]
\hline
$
\mu_2\left(S_2\right)
=\displaystyle
\frac{m_{2}S_2}{K_2+S_2+\frac{S_2^2}{K_I}}$
\\[2mm]
$S_2^M=\sqrt{K_2K_I}$
\\[2mm]
$\mu_2\left(S_2^M\right)=\displaystyle
\frac{m_2}{1+2\sqrt{K_2/K_I}}$
\\
$S_2^{1*}(D)=\displaystyle
\frac{(m_2-\alpha D)K_I-\sqrt{(m_2-\alpha D)^2K_I^2-4(\alpha D)^2K_2K_I}}{2\alpha D}
$
\\[2mm]
$S_2^{2*}(D)=\displaystyle
\frac{(m_2-\alpha D)K_I+\sqrt{(m_2-\alpha D)^2K_I^2-4(\alpha D)^2K_2K_I}}{2\alpha D}$
\\[2mm]
$S_2^{1*}(D)$ and $S_2^{2*}(D)$ 
are defined for $0<D<D_2$, where 
$D_2=
\displaystyle
\frac{\mu_2\left(S_2^M\right)}{\alpha}
=\frac{m_2}{\alpha}\frac{1}{1+2\sqrt{K_2/K_I}}
$
\\
\hline
$H_i(D)=S_2^{i*}(D)+\frac{k_2}{k_1}S_1^*(D)$,
$i=1,2$, defined for $0<D<\min(D_1,D_2)$
\\[2mm]
$
{S}^*_{2{\rm in}}\left(D,S_{1{\rm in}},
S_{2{\rm in}}\right)=
S_{2{\rm in}}+\frac{k_2}{k_1}S_{1{\rm in}}-\frac{k_2}{k_1}S_1^*(D)$,
defined for $0<D<D_1$
\\[2mm]
$
X_{2}^{i}\left(D,S_{2{\rm in}}\right)=
\frac{1}{k_3\alpha}
\left(S_{2{\rm in}}-S_{2}^{i*}(D)\right)$, $i=1,2$,
defined for $0<D<D_2$
\\[2mm]
%=====================================================================
$
X_{2}^{i*}\left(D,S_{1{\rm in}},S_{2{\rm in}}\right)=\frac{1}{k_3\alpha}
\left(
S_{2{\rm in}}+\frac{k_2}{k_1}S_{1{\rm in}}-\frac{k_2}{k_1}H_i(D)\right)$, $i=1,2$,
defined for $0<D<\min(D_1,D_2)$
\\
\hline
\end{tabular}
\end{center}
\end{table}
%%%%%%%%%%%%%%%%%%%%%%

In this section, we give several tables that are used in the paper. 
In Table \ref{parametervalues}, we provide the biological parameter values used in the figures. Tables \ref{IntersectionD} and \ref{IntersectionS2in}, we give the description of the intersection of the $\Gamma_k$ surfaces with a two dimesnional operating plane where $D$ or $S_{2in}$ is kept constant respectively. 
In Table \ref{AuxiliaryFunctionsMonodHaldane}, we present the functions defined in Table \ref{tableFunctions}
in the particular case of the Monod and Haldane growth function \ref{MonodHaldane}.

\section*{Acknowledgments}
The authors thank the Euro-Mediterranean research network TREASURE 
(\url{http://www.inra.fr/treasure}) for support.
The authors thank Jérôme Harmand for valuable and fruitful discussions. During the preparation of this work, the second author was publicly funded through ANR (the French National Research Agency) under the ``Investissements d’avenir'' programme with the reference ANR-16-IDEX-0006. The second author thankks Direction G\'en\'erale de la Recherche Scientifique et du D\'eveloppement Technologique (DG RSDT), Algeria, for support.

%\section*{Conflict of interest}

\bibliographystyle{unsrt}
%\bibliography{references}  %%% Remove comment to use the external .bib file (using bibtex).
%%% and comment out the ``thebibliography'' section.

%%% Comment out this section when you \bibliography{references} is enabled.

\end{document}